\documentclass[aos,noinfoline]{imsart}

\RequirePackage[OT1]{fontenc}
\usepackage[table]{xcolor}
\RequirePackage{amsmath,amssymb,amsthm,mathrsfs,enumerate,bm,makeidx,xcolor,multirow,pbox,tikz}
\RequirePackage[colorlinks,citecolor=blue,urlcolor=blue]{hyperref}

\setattribute{journal}{name}{}

\numberwithin{equation}{section}
\newcommand{\N}{\mathbb{N}}
\newcommand{\R}{\mathbb{R}}

\newcommand{\pnorm}[2]{\lVert#1\rVert_{#2}}
\newcommand{\abs}[1]{\left\lvert#1\right\rvert}
\newcommand{\iprod}[2]{\left\langle#1,#2\right\rangle}
\renewcommand{\epsilon}{\varepsilon}

\renewcommand{\d}[1]{\mathrm{d}#1}

\DeclareMathOperator*{\argmin}{argmin}
\DeclareMathOperator*{\argmax}{argmax}

\newcommand{\beq}{\begin{equation}}
\newcommand{\eeq}{\end{equation}}
\newcommand{\beqa}{\begin{equation} \begin{aligned}}
\newcommand{\eeqa}{\end{aligned} \end{equation}}
\newcommand{\beqas}{\begin{equation*} \begin{aligned}}
\newcommand{\eeqas}{\end{aligned} \end{equation*}}

\newcommand{\bit}{\begin{itemize}}
	\newcommand{\eit}{\end{itemize}}
\newcommand{\bmat}{\begin{bmatrix}}
	\newcommand{\emat}{\end{bmatrix}}

\makeindex

\theoremstyle{definition}\newtheorem{problem}{Problem}[section]
\theoremstyle{definition}
\theoremstyle{remark}\newtheorem{assumption}{Assumption}
\theoremstyle{remark}\newtheorem{smooth}{Smoothness Assumption}
\theoremstyle{remark}
\theoremstyle{remark}\newtheorem{remark}[problem]{Remark}
\theoremstyle{definition}
\theoremstyle{plain}\newtheorem{theorem}[problem]{Theorem}
\theoremstyle{plain}\newtheorem{lemma}[problem]{Lemma}
\theoremstyle{plain}\newtheorem{proposition}[problem]{Proposition}
\theoremstyle{plain}\newtheorem{corollary}[problem]{Corollary}
\theoremstyle{plain}

\begin{document}

\begin{frontmatter}
\title{Multivariate convex regression: \\ global risk bounds and adaptation}
\runtitle{Risk bounds for Multivariate Convex Regression}

\begin{aug}
	\author{\fnms{Qiyang} \snm{Han}\ead[label=e1]{royhan@uw.edu}}
	\and
	\author{\fnms{Jon A.} \snm{Wellner}\thanksref{t2}\ead[label=e2]{jaw@stat.washington.edu}}
	\ead[label=u1,url]{http://www.stat.washington.edu/jaw/}
	
	\thankstext{t2}{Supported in part by NSF Grant DMS-1104832 and NI-AID grant R01 AI029168}
	\runauthor{Han and Wellner}
	
	\affiliation{University of Washington}
	
	\address{Department of Statistics, Box 354322\\
		University of Washington\\
		Seattle, WA 98195-4322\\
		\printead{e1}}
	
	\address{Department of Statistics, Box 354322\\
		University of Washington\\
		Seattle, WA 98195-4322\\
		\printead{e2}\\
		\printead{u1}}
\end{aug}

\begin{abstract}
In this paper, we study the problem of estimating a multivariate 
convex function defined on a convex body in a regression setting with 
random design. 
We are interested in the attainability of optimal rates of convergence under a squared global continuous $l_2$ loss in the multivariate setting $(d\geq 2)$. 
One crucial fact is that the minimax risks depend heavily 
on the shape of the support of the regression function. 
It is shown that the global minimax risk is on the order of 
$n^{-2/(d+1)}$ when the support is sufficiently smooth, 
but that the rate  $n^{-4/(d+4)}$ is achieved automatically when the support is a polytope. Such tremendous differences in rates are due to difficulties in estimating the regression function near the boundary of smooth regions.

We then study the natural bounded least squares estimators (BLSE):  we 
show that the BLSE nearly attains the optimal rates of convergence in 
low dimensions, while suffering rate-inefficiency in high dimensions. 
Remarkably, we show that the BLSE adapts nearly parametrically to polyhedral functions when the support is polyhedral in low dimensions by a local entropy method. We also show that the boundedness constraint cannot be dropped when risk is assessed via continuous $l_2$ loss.

Given rate suboptimality of the BLSE in higher dimensions, we further study rate-efficient adaptive estimation procedures. 
Two general model selection methods are developed to provide sieved adaptive estimators (SAE) that achieve nearly optimal rates of convergence for particular ``regular'' classes of convex functions, while maintaining nearly parametric rate-adaptivity to polyhedral functions in arbitrary dimensions. Interestingly, the uniform boundedness constraint is unnecessary when risks are measured in discrete $l_2$ norms. As a byproduct, we obtain nearly rate-optimal adaptive estimators for unknown convex sets from noisy support function measurements in arbitrary dimensions. 
\end{abstract}



\end{frontmatter}

\newpage
\tableofcontents

\section{Introduction}\label{section:introduction}

\subsection{Overview}
Nonparametric estimation under convexity constraints has received much attention in recent years. In this paper, we study the the problem of estimating an unknown convex function $f_0$ on a convex body $\Omega\subset \R^d$ from observations $(X_i,Y_i)_{i=1}^n$ where $X_1,\ldots,X_n$ are i.i.d. according to a probability law $\nu$ on $\Omega$, and $Y_i|X_i$ follows the model
\beqa\label{model}
Y_i=f_0(X_i)+\epsilon_i,\quad \textrm{ for all } i=1,\cdots,n.
\eeqa
Here the $\epsilon_i$'s are i.i.d. mean zero errors with variance $\sigma^2$. This is a random design regression model. We are interested in determining the optimal rates of convergence for estimating the unknown convex $f_0$ based on random observations from the above model under the natural associated continuous $l_2$ norm for the probability measure $\nu$ defined by
\beqa\label{eqn:l_nu}
l^2_\nu(f_0,g):=\int_{\Omega}(f_0-g)^2\ \d{\nu}.
\eeqa

Convex nonparametric regression has a long history.
\cite{hildreth1954point,hanson1976consistency} studied least squares estimation in the case of dimension $d=1$. In the multidimensional case,
\cite{matzkin1991semiparametric,kuosmanen2008representation,lim2012consistency} and others studied different aspects of the problem in more restricted setups, before
\cite{seijo2011nonparametric} studied the  statistical properties of least squares estimation and related computational techniques in a general setting. Global convexity also proves useful in faithfully selecting relevant variables under sparse additive modelling in the high-dimensional setting in \cite{xu2014faithful}.

The rates of convergence for convexity/concavity restricted estimators have been investigated primarily in dimension $1$. From a global point of view, \cite{dumbgen2004consistency} showed that the supremum loss of convex least squares estimators (LSEs) on any compacta within the domain is of order $(\log n/n)^{2/5}$ (no squaring). \cite{guntuboyina2013global} established global risk bounds of order $n^{-4/5}$ modulo logarithmic factors under squared discrete $l_2$ norm for the LSE on $\Omega=[0,1]$ in the regression setup with almost equi-distributed design points. The interesting feature is that, the LSEs are nearly parametrically rate-adaptive to piecewise linear convex functions. In a different setting of density estimation, \cite{doss2013global} concluded the global rates of convergence of order no worse than $n^{-2/5}$ (no squaring) under Hellinger metric for the maximum likelihood estimators (MLEs) of $\log$- and $s$-concave densities.

From a local point of view, \cite{mammen1991nonparametric} established rates of convergence on the order of $n^{-2/5}$ (no squaring) at fixed smooth points for LSEs in the regression setup. The pointwise limit distribution theory of the MLEs of convex decreasing densities \cite{groeneboom2001estimation}, log-concave densities \cite{balabdaoui2009limit} and R\'enyi divergence estimators for $s$-concave densities \cite{han2015approximation} follows this same rate of $n^{-2/5}$ (no squaring) at such smooth points. Adaptive phenomenon is also observed at a local scale in \cite{chen2014convex} and \cite{cai2011framework} with various degrees of differentiability assumptions. 

Such phenomenon as global and local adaptation have been found also in estimation procedures with monotonicity constraints, see e.g. 
 \cite{zhang2002risk}, 
 \cite{chatterjee2015risk}, \cite{bellec2015sharp}. In particular, \cite{chatterjee2015risk} characterized the global adaptive nature of the LSEs with general shape restrictions induced by conic inequalities in terms of the statistical dimension of the cone. This covers isotonic ($1$-monotone), convex ($2$-monotone) and general $k$-monotone regression problems. However, their conic inequalities require a strong order relationship between the design points, and thus render extension to high dimensions difficult.

In higher dimensions ($d\geq 2$), rates of convergence for estimating convex functions are far less well understood.  \cite{lim2014convergence} and \cite{balazs2015near} studied least squares estimation over the class of uniformly Lipschitz and uniformly bounded convex functions on $[0,1]^d$. In the presence of such restrictions, the (slightly weaker) results readily follow from classical entropy bounds (cf. Corollary 2.7.10 \cite{van1996weak}) and empirical process theory. In a related problem of estimating convex sets in higher dimensions, it is shown in \cite{guntuboyina2012optimal} that estimation of an unknown convex set via support functions enjoys minimax optimal rates of convergence on the order of $n^{-4/(d+3)}$ under discrete squared $l_2$ norm. On the other hand, \cite{brunel2014adaptive} showed that in the setting of estimating the support of a uniform density known to be convex, the optimal rates of convergence under Nikodym metric\footnote{The Nikodym metric between two measurable sets $K,K'$ is defined as $\abs{K\Delta K'}$.} is of order $\log n/n$ when the support is a polytope, and $n^{-2/(d+1)}(d\geq 2)$ when the support is a general convex body. 

In the setting of multivariate density estimation with convexity constraints, \cite{seregin2010nonparametric} derived a minimax lower bound on the order of $n^{-2/(d+4)}$ (no squaring) for estimating a concave-transformed density at a fixed point under curvature conditions. More recently, \cite{kim2014global} show that estimating log-concave densities via the MLEs yields different rates from the conjectured rates as above in low dimensions and the rates conjectured in \cite{seregin2010nonparametric}. The key observation in their paper is that the bracketing entropy of a suitable subclass of log-concave densities is on the order of $\max\{\epsilon^{-d/2},\epsilon^{-(d-1)}\}$ rather than the  $\epsilon^{-d/2}$ in higher dimensions as conjectured in \cite{seregin2010nonparametric}. 
The new entropy estimate gives global minimax risks on the order of  $n^{-1/(d+1)}$ (no squaring) for $d\geq 2$, which is strictly worse than the pointwise rate $n^{-2/(d+4)}$ (no squaring).  The larger entropy $\epsilon^{-(d-1)}$ exhibited in \cite{kim2014global} actually comes from uniform densities with smooth boundaries. Similar ideas have been explored further in \cite{gao2015entropy} where it is shown that the metric entropy of convex functions on a polyhedral region $\Omega$ differs significantly from the metric entropy of convex functions
on a domain $\Omega$ with a smooth boundary such as a ball.  This quickly leads to the conjecture that the smoothness of the boundary of the domain $\Omega$ plays a significant role in determining the degree of difficulty in estimation of a convex function defined on some set $\Omega \subset \R^d$, especially for higher dimensions
$d$.  

In this paper we investigate this issue in detail.  We adopt a minimax approach and show that the difficulty in estimating 
a convex function $f_0$ in the regression framework with a random design depends heavily on the  smoothness 
of the support of $f_0$. We first show that, the global minimax risks for convex regression under squared $l_\nu$ loss as defined in (\ref{eqn:l_nu}) are generally on the order of $n^{-2/(d+1)}$ for smooth supports (to be defined in Section \ref{section:notation}), while a faster rate of $n^{-4/(d+4)}$ is possible when the support is a polytope. Such sharp differences in global minimax risk are due to boundary perturbations of smooth supports that lead to a class of least favorable regression functions to be distinguished from the true one.

We then turn to study a variant of LSEs studied by \cite{seijo2011nonparametric}, with a uniform bound constraint, which we call bounded least squares estimators (BLSE). The uniform boundedness constraint, as we shall see in Section \ref{section:discussion_uniform_bound}, cannot be relaxed in studying risk bounds under random design. We summarize our risk bounds for the BLSE in squared $l_\nu$ norm obtained in the following table.
\begin{center}
	\begin{tabular}{ |c|c|c|c| } 
		\hline
		$(\Omega,f_0)$ & $(\mathscr{P}_k, \mathcal{P}_{m_{f_0}}(\Gamma))$ & $(\mathscr{P}_k,\mathcal{C}(\Gamma))$  & $(\mathscr{C}, \mathcal{C}(\Gamma))$\\
		\hline
		$d=1$ & $n^{-1}(\log n)^{5/4}$ & \multicolumn{2}{|c|}{$n^{-4/5}$}\\
		\hline
		$d=2$ & \multirow{2}{*}{$n^{-1}(\log n)^{d(d+4)/4}$} & \multirow{2}{*}{$n^{-4/(d+4)}$}           & $n^{-2/3}\log n$\\ 
		\cline{1-1}
		\cline{4-4}
		$d=3$ &                        &                                & $n^{-1/2}\log n$\\ 
		\hline
		$d=4$ & $n^{-1}(\log n)^{10}$                  & $n^{-1/2}\log n$                   & \multirow{2}{*}{$n^{-1/(d-1)}$}\\ 
		\cline{1-3}
		$d\geq 5$ & $n^{-4/d}(\log n)^{d+4}$           & $n^{-2/d}$                &                     \\ 
		\hline
	\end{tabular}
\end{center}
Notation can be found in Section \ref{section:notation}. To summarize, the BLSEs behave differently for different shapes of support and the true regression functions in that adaptive estimation occurs when (1) the support is polytopal with consequent smaller entropy of the class of convex functions; (2) the support is polytopal and the regression function is polyhedral. This is in agreement with the adaptive properties obtained in \cite{brunel2013adaptive,brunel2014adaptive,cai2015adaptive} in that the epigraph of such a regression function is of polyhedral type. In particular, nearly parametric risks in dimensions $d\leq 4$ when the support is polytopal and the regression function is polyhedral are established by a local entropy method, as we shall discuss in detail in Section \ref{section:lse}. It is natural to wonder if adaptation occurs when the support is a general convex body and the regression function is polyhedral. We conjecture that the answer is negative within the current methods via local entropy. For further discussion see Section \ref{section: discussion_adaptive_smooth}. 

It is worthwhile to note that,  when the support is polytopal, the BLSEs achieve nearly optimal risks for $d\leq 4$, while such optimality only holds for $d\leq 3$ when the support is a general smooth convex body. 
Such rate inefficiency is also observed in \cite{birge1993rates} in the context of density estimation via minimum constrast estimators for H\"olderian classes, and conjectured for the MLEs of log-concave densities in \cite{seregin2010nonparametric} in higher dimensions.

Given rate-suboptimality of the BLSEs, we further study rate-efficient adaptive estimation procedures. We show that the notion of `pseudo-dimension' coined in \cite{pollard1990empirical} (see also Section \ref{section:sieved_lse}) effectively characterizes the complexity for the low-dimensional models, i.e. polyhedral functions, within the class of multivariate convex functions. We then develop general model selection methods, from which two different types of sieved adaptive estimators (SAE) are studied and shown to achieve nearly optimal rates of convergence while being rate-adaptive simultaneously to all these low-dimensional models up to universal constants. Risks for these SAEs are both considered in continuous and discrete $l_2$ norms. Interestingly, the uniform boundedness constraint is not necessary when the discrete $l_2$ norm is used. See Theorems \ref{thm:risk_fixed_model} and \ref{thm:lepski_generic} for precise statements.

Applying these methods to the multivariate convex regression setup, we show that the risks of the SAEs are on the order of $\log n/n$ for polyhedral functions and $n^{-4/(d+4)}(\log n)^{\gamma_d}$ for uniformly Lipschitz (regular) convex functions for some $\gamma_d>0$, whatever the shape of the support. This is not a contradiction with the global minimax risk $n^{-2/(d+1)}$ for smooth domains since the faster rate $n^{-4/(d+4)} ( \log n )^{\gamma_d}$ is only achieved when the regression 
function behaves nicely near the boundary of the domain, a setting which excludes the global least favorable case. The BLSE is unlikely to be rate-adaptive for such regular classes since the Lipschitz behavior of the BLSE near the boundary can be arbitrarily bad; see the further discussion in Section \ref{section: discussion_adaptive_smooth}.

As a byproduct of our general framework, we obtain a nearly rate-optimal estimator for an unknown convex set from support function measurements that adapts simultaneously to all polytopes with nearly parametric rate. This gives a solution to this problem in arbitrary dimensions; the case $d=2$ was previous considered by \cite{cai2015adaptive}.

The rest of the paper is organized as follows. We study the global minimax risks in Section \ref{section:minimax_rates}. Section \ref{section:lse} is devoted to risk bounds for the BLSEs. The model selection methods and the associated SAEs, are presented in Section \ref{section:sieved_lse} with some further results discussed in Appendix \ref{section:sieved_sobolev}. Related issues and problems are discussed in Section \ref{section:discussion}. For clarity of presentation, proofs are relegated to Appendices \ref{section:proof_intro}-\ref{section:technical_lemma}. Auxiliary results from empirical process theory and convex geometry are collected in Appendix \ref{section:auxiliary_results}.

\subsection{Notation and conventions}\label{section:notation}

$\pnorm{\cdot}{p}$ denotes the $p$-norm for an Euclidean vector and $\pnorm{\cdot}{}$ is usually understood as $\pnorm{\cdot}{2}$. $B_p(x,r)$ denotes the $l_p$ ball of radius $r$ centered at $x$ in $\R^d$. $B_d$ is an abbreviation for $B_2(0,1)$. $\Delta_d:=\{x\in \R^d:x_i\geq 0,\sum_i x_i\leq 1\}$ is used for the canonical simplex in $\R^d$. The volume of a measurable set $A$ in Lebesgue measure is usually denoted $\abs{A}$. The symbols $:=$ and $\equiv$ are used for definitions. $\mathbb{P}$ and $\mathbb{E}$ are sometimes abused (in proofs) for \emph{outer} probability and expectation to handle possible measurability issues.

For a probability measure $\nu$ on $\Omega$, we denote the \emph{continuous} $l_2$ metric under $\nu$ by $l_\nu$ as defined in (\ref{eqn:l_nu}), while $l_2$ is used when $\nu$ is Lebesgue measure $\lambda\equiv\lambda_d$. We assume that $\nu$ is absolutely continuous with respect to Lebesgue measure $\lambda$, and write $\nu_{\max}\equiv\sup_{x \in \Omega}\d{\nu}/\d{\lambda}(x)$ and $\nu_{\min}\equiv\inf_{x \in \Omega}\d{\nu}/\d{\lambda}(x)$. For $\underline{X}^n=(X_1,\cdots,X_n)\in \R^{d\times n}$, define the \emph{discrete} $l_2$ metric by $
l^2_{\underline{X}^n}(f,g):=n^{-1}\sum_{i=1}^n \big(f(X_i)-g(X_i)\big)^2$.

\subsubsection{Conventions on constants}
$C_{x}$ will denote a generic constant that depends only on $x$, which may change from line to line unless otherwise specified. $a\lesssim_{x} b$ and $a\gtrsim_x b$ mean $a\leq C_x b$ and $a\geq C_x b$ respectively, and $a\asymp_x b$ means $a\lesssim_{x} b$ and $a\gtrsim_x b$. $C$ is understood as an absolute constant unless otherwise specified. Constants in theorems are stated in German letters (e.g. $\mathfrak{c,C,k,K}$). $\mathfrak{C}$ will denote a generic constant with specified dependence whose value may change from line to line. For two real numbers $a,b$, $a\vee b:=\max\{a,b\}$ and $a\wedge b:=\min\{a,b\}$.

\subsubsection{Conventions on convex bodies}
Let $\mathscr{P}_k$ denote the collection of polytopes with at most $k$ simplices. For a polytope $\Omega \in \mathscr{P}_k$, we call $\Omega=\cup_{i=1}^k \Omega_i$ a \emph{simplical decomposition} of $\Omega$ if all the $\Omega_i$'s are simplices with non-overlapping interiors. Let $\mathscr{C}$ denote the set of all smooth convex bodies in $\R^d$\footnote{Here 'smooth' will mean that Assumption \ref{smooth2} (below) holds.}. Note in dimension $d=1$, $\mathscr{C}=\mathscr{P}_1$. The width of a convex body $\Omega$ is denoted by $w(\Omega):=\sup_{x,y \in \Omega}\pnorm{x-y}{2}$. 
A convex body $\Omega$ is \emph{smooth} if the following two conditions are satisfied:
\begin{smooth}\label{smooth1}
For $\epsilon>0$ small enough, there exist disjoint caps $\{C_i\}_{i=1}^m$ such that $\abs{C_i}\lesssim_d \epsilon\abs{\Omega}$ and $m\lesssim_d (\epsilon\abs{\Omega})^{-(d-1)/(d+1)}$.
\end{smooth}

This is a slightly stronger version of the \emph{Economic Covering Theorem} (cf. Theorem \ref{thm:economic_cap_covering}) studied in the convex geometry literature, where we require $C_i$ to be caps instead of simple convex sets. See also Remark \ref{rmk:num_cover_boundary}.

Now we state our second assumption. A sequence of simplices $\{D_i\}_{i=1}^\infty$ is called \emph{admissible} if their interiors are pairwise disjoint. Let $
S(t,\{D_i\}_{i=1}^\infty ;\Omega):=\min\{m \in \N: \abs{\Omega\setminus \cup_{i=1}^m D_i}\leq t\abs{\Omega}\}$.
Now the \emph{simplicial approximation number} is defined by $
S(t;\Omega):=\inf_{\{D_i\}_{i=1}^\infty}S(t,\{D_i\}_{i=1}^\infty ;\Omega)$
where the infimum is taken over all admissible sequences.
\begin{smooth}\label{smooth2}
The \emph{simplicial approximation number} $S(t,\Omega)$ satisfies the growth condition
\beqa\label{limit:smooth_condition_2}
\limsup_{t \to 0} t^{(d-1)/2}S(t,\Omega)<\infty.
\eeqa
\end{smooth}

The power $(d-1)/2$ is natural in the sense that it agrees with \cite{bronshteyn1975approximation}: Any convex body can be approximated by a polytope with $n$ vertices within Hausdorff distance no more than $O(n^{2/(d-1)})$. Here we require the approximation to hold in a sense so that such a bound is valid constructively.
\begin{lemma}\label{lemma:smooth}
Any ellipsoid satisfies Smoothness Assumptions \ref{smooth1} and \ref{smooth2}.
\end{lemma}

\subsubsection{Conventions on convex functions}
For a multivariate real-valued function $f:\R^d \to \R$, let $\pnorm{f}{L}\equiv L(f)\equiv\sup_{x\neq y}{\abs{f(x)-f(y)}}/{\pnorm{x-y}{2}}$ denote the Lipschitz constant for $f$.  $\pnorm{f}{l_p}$ will denote the standard $l_p$ norm ($p\geq 1$).

We denote the class of all convex functions that are bounded by $\Gamma$ in $l_p$ norm, whose Lipschitz constants are bounded by $L$ and whose domains are contained in $\Omega$ by $\mathcal{C}_p(\Gamma,L;\Omega)$. Dependence on the domain is often suppressed. Dependence on $p,\Gamma,L$ is also suppressed when they equal $\infty$\footnote{For example, $\mathcal{C}(\Gamma,L)=\mathcal{C}_\infty(\Gamma,L)$ and $\mathcal{C}(\Gamma)=\mathcal{C}_\infty(\Gamma,\infty)$.}. We also let $\mathcal{P}_m(\Gamma)$ be the collection of polyhedral convex functions $f \in \mathcal{C}(\Gamma)$ with at most $m$ facets\footnote{Here by a facet of a polyhedral convex function $f$ we mean any $d$-dimensional polytope within $\Omega$ on which $f$ is affine.}. Alternatively, we can represent $f \in \mathcal{P}_m(\Gamma)$ as
$ f(x)\equiv \max_{i=1,\cdots,m} \big(a_i^T x+b_i\big) $ for some $\{(a_i,b_i)\in \R^d\times \R\}_{i=1}^d$ 
so that $\pnorm{f}{l_\infty}\leq \Gamma$. Similarly, we often simply denote $\mathcal{P}_m(\infty)$ by $\mathcal{P}_m$. 

For a given support $\Omega$, we call the class of polyhedral convex functions as the \emph{simple class}, the class of all convex functions with pre-specified uniformly bounded Lipschitz constant as the \emph{regular class}.

\subsubsection{Conventions on entropy numbers}
Let $(\mathcal{F},\pnorm{\cdot}{})$ be a subset of the normed space of real functions $f:\mathcal{X}\to \R$. The metric entropy number $\mathcal{N}(\epsilon,\mathcal{F},\pnorm{\cdot}{})$ is the minimal number of $\epsilon$-balls in $\pnorm{\cdot}{}$ norm needed to cover $\mathcal{F}$, while the bracketing entropy number $\mathcal{N}_{[\,]}(\epsilon,\mathcal{F},\pnorm{\cdot}{})$ is the minimum number of $\epsilon$-brackets needed to cover $\mathcal{F}$. By an $\epsilon$-bracket we mean the subset of functions $f \in \mathcal{F}$ determined by a pair of functions $l\leq u$ as follows: $[l,u]:=\{f\in\mathcal{F}:l\leq f\leq u\}$ with $\pnorm{l-u}{}\leq \epsilon$.

\section{Global minimax risks}\label{section:minimax_rates}
We will be interested in the global minimax risk defined by
\beqa\label{eqn:generic_minimax_rate}
R_\nu(n;\mathcal{F}):=\inf_{\hat{f}_n}\sup_{f \in \mathcal{F}}\mathbb{E}_f l_\nu^2(f,\hat{f}_n),
\eeqa
where $\mathcal{F}$ is the function class of interest, and the infimum runs over all possible estimators based on the observations $(X_i,Y_i)_{i=1}^n$.

\subsection{Minimax risk upper bounds}

We first derive a general minimax upper bound.

\begin{theorem}\label{thm:generic_minimax_upper_bound_exponential}
	Suppose $\mathcal{F}$ is uniformly bounded by $\Gamma$, and the errors $\{\epsilon_i\}$ are independently sub-Gaussian with parameter $\sigma^2$: $\mathbb{E}e^{u\epsilon_i}\leq \exp\big(u^2\sigma^2/2\big)$. Let the rate function be defined by
	\beqa
	r_n:=\inf_{\delta>0}\bigg(\frac{1}{\mathfrak{z}_0n}\log  N(\delta)+34\delta^2\bigg),
	\eeqa
	where $N(\delta)\geq \mathcal{N}(\delta,\mathcal{F},l_\nu)$ for all $\delta>0$.
	Then there exists an estimator $\hat{f}_n \in \mathcal{F}$ such that for any $t>0$,
	\beqas
	\sup_{f_0 \in \mathcal{F}}\mathbb{P}\big(n(l_\nu^2(\hat{f}_n,f_0)-r_n)>t\big)\leq \exp(-\mathfrak{z}_0 t),
	\eeqas
	Here the constant $\mathfrak{z}_0$ is defined via (\ref{ineq:minimax_exponential_bound_zeta}).
\end{theorem}

The proof is a generalization of the method of sieves by progressively choosing `theoretical' sieves constructed via knowledge of the metric entropy of the function class to be estimated. As a direct corollary, we obtain
\begin{corollary}\label{cor:generic_minimax_upper_bound}
$R_\nu(n;\mathcal{F})\leq r_n+\frac{1}{\mathfrak{z}_0 n}$.
\end{corollary}
Typically $r_n$ is of larger order than $1/n$ and hence the right hand side of the display is on the order of $r_n$. 

Now we shall use the above results to establish a minimax risk upper bound for the convex regression problem. This is a direct consequence of Corollary \ref{cor:generic_minimax_upper_bound} in view of the entropy result Lemma \ref{lem:entropy_cvx_func}.
\begin{theorem}[Minimax risk upper bounds]\label{thm:upper_bound_poly}
	For a polytopal domain  $\Omega \in \mathscr{P}_k$, we have
	\beqas
	R_\nu(n;\mathcal{C}(\Gamma))\leq \mathfrak{C}_{d,\abs{\Omega},\Gamma,\sigma,\nu}(k/n)^{4/(d+4)}.
	\eeqas
	For a smooth domain $\Omega$, we have
	\beqas
	R_\nu(n;\mathcal{C}(\Gamma))\leq \mathfrak{C}_{d,\abs{\Omega},\Gamma,\sigma,\nu}
	\begin{cases}
		n^{-2/3} (\log n) & d=2;\\
		n^{-2/(d+1)} & d\geq 3.
	\end{cases}
	\eeqas
	Here the conclusion for $d=2$ holds for $n$ large enough. Explicit forms for the constants can be found in (\ref{const:minimax_upper_bound_1}) and (\ref{const:minimax_upper_bound_2}).
\end{theorem}

\subsection{Minimax risk lower bounds}

In this section, errors will be assumed i.i.d. Gaussian, i.e. $\epsilon_i\sim \mathcal{N}(0,\sigma^2)$.

\subsubsection{General class}
We consider global minimax risk lower bounds for two types of supports: (1) polytopes; (2) smooth convex bodies. 

\begin{theorem}\label{thm:minimax_lower_bound_poly}
For a polyhedral domain $\Omega \in \mathscr{P}_k$,  we have
\beqas
R_\nu(n;\mathcal{C}(\Gamma))\geq \mathfrak{C}_{d,\Omega,\Gamma,\sigma,\nu} n^{-4/(d+4)}.
\eeqas
An explicit form for the constant can be found in (\ref{const:minimax_lower_bound_poly}).
\end{theorem}

\begin{theorem}\label{thm:minimax_lower_bound_cvxbody}
Let $d\geq 2$. Suppose the domain $\Omega$ satisfies Smoothness Assumption \ref{smooth1}. Then
\beqas
R_\nu(n;\mathcal{C}(\Gamma))\geq \mathfrak{C}_{d,\Omega,\Gamma,\sigma,\nu}n^{-2/(d+1)}.
\eeqas
An explicit form for the constant can be found in (\ref{const:minimax_lower_bound_smooth}).
\end{theorem}

Notably, the least favorable functions $\{f_\tau\}$ achieving the rate $n^{-4/(d+4)}$ for polytopal domains in Theorem \ref{thm:minimax_lower_bound_poly} and the class $\{f_\tau\}$ yielding the rate $n^{-2/(d+1)}$ for smooth domains in Theorem \ref{thm:minimax_lower_bound_cvxbody} are radically different. In fact, the class $\{f_\tau\}$ yielding the rate $n^{-4/(d+4)}$ involves perturbations of a reference convex function in the interior of the domain while maintaining sufficient curvature to ensure convexity so that the resulting rate corresponds to the rate in the classical case of a function space with smoothness index $2$. On the other hand, the slower rate in Theorem \ref{thm:minimax_lower_bound_cvxbody} involves showing that a smooth boundary allows more 
perturbations than in the interior, and thus the boundary behavior ultimately drives the slower rate.

\subsubsection{Simple class}
It is not difficult to establish the general lower bound on the order of $1/n$ under squared $l_\nu$ norm, so we shall examine the case where a slower minimax rate is possible. We shall illustrate this by considering the minimax rates for polyhedral functions supported on a smooth region.

\begin{theorem}\label{thm:minimax_lower_bounds_simple}
Let $d\geq 2$. Suppose the domain $\Omega$ satisfies Smoothness Assumption \ref{smooth1}. Then for $n\geq \mathfrak{n}_0$ with $\mathfrak{n}_0$ being some constant depending on $k,d,\sigma,\Gamma,\nu$, it holds that
\beqas
R_\nu(n,\mathcal{P}_k)\geq \mathfrak{C}_d \frac{\sigma^2\nu_{\min}}{\nu_{\max}}\frac{k\log n}{n}.
	\eeqas
\end{theorem}

Note that $\mathfrak{n}_0$ depends on $\Gamma$.

\section{Least squares estimation}\label{section:lse}

\subsection{The estimator}
In the convex regression setting, the least squares estimator (LSE) given observation $(X_i,Y_i)_{i=1}^n$ is
\beqa
\hat{f}_n^{\mathrm{LS}}:=\argmin_{f \in \mathcal{C}} \sum_{i=1}^n(Y_i-f(X_i))^2.
\eeqa
By a canonical construction (see also (\ref{eqn:lse_constraint_cvxreg})), such LSEs exist and are consistent in view of \cite{seijo2011nonparametric} in the sense that $\hat{f}^{\mathrm{LS}}_n$ converges uniformly on any compact set in the interior of the domain of the true regression function $f_0$. Here we shall study the \emph{bounded} LSE with the constraint that $f \in \mathcal{C}(\Gamma)$ with some specified $\Gamma$. It is shown in Section \ref{section:discussion_uniform_bound} that a boundedness condition is necessary in studying the risk for the LSE for convex regression due to the bad behavior of the estimator near the boundary. 

The bounded least squares estimators (BLSEs) can also be formulated as follows.
\beqa\label{eqn:cvx_lse_def}
\min_{\{y_i\},\{g_i\}} &\quad \sum_{i=1}^n (Y_i-y_i)^2\\
\mathrm{subject \,to} &\quad y_j\geq y_i+g_i^T(X_j-X_i),\\
&\quad y_i \geq -\Gamma, -y_i\geq -\Gamma,  \qquad\qquad\qquad\qquad\textrm{ for all }i,j=1,\cdots,n,\\
&\quad \Gamma \geq y_i+g_i^T(v-X_i),\textrm{ for all }v\in \partial\Omega.
\eeqa
When the support is known to be a polytope with vertices $\{v_i\}_{i=1}^k$, then the last condition of (\ref{eqn:cvx_lse_def}) can be replaced with
\beqa
\Gamma \geq y_i+g_i^T(v_l-X_i),\textrm{ for all }l=1,\ldots,k, \textrm{ and } i=1,\ldots,n.
\eeqa
This is a quadratic programming (QP); see page 338 in \cite{boyd2004convex} for more  details.

The existence of the solution $\{\hat{y}_i\}$ and $\{\hat{g}_i\}$ of (\ref{eqn:cvx_lse_def}) is clear, and we will use these estimated interpolating function values and subgradients to define a canonical estimator as follows:
\beqa\label{eqn:lse_constraint_cvxreg}
\hat{f}^{\mathrm{LS}}_n(x):= \max_{i=1,\cdots,n}(\hat{y}_i+\hat{g}_i^T(x-X_i)).
\eeqa
For notational convenience, we suppress explicit dependence of this estimator on the uniform bound $\Gamma$. As we shall see, the uniform bound $\Gamma$ does not affect the rates of convergence  as long as it exceeds $\pnorm{f_0}{\infty}$.

\subsection{Risk bounds via local entropy}
The rates of convergence of the least squares estimators (LSEs) in the regression setting are well studied in the empirical process literature, see for example \cite{birge1993rates,van1996weak,van2000empirical}. In particular, the local geometry near the true regression function, measured via the size of the local metric/bracketing entropy, drives the rates of convergence for LSEs. 
In our specific random design setting, we shall first strengthen the results of Theorem 9.1 \cite{van2000empirical} for \emph{risk bounds} in the fixed design setting and discussions in page 335 \cite{van1996weak} for \emph{rates of convergence} in the random design, to \emph{risk bounds} for LSEs in the general regression setting with random design.

To fix notation, suppose $Y_i=f_0(X_i)+\epsilon_i$, and $X_i$'s are i.i.d. from a probability measure $\nu$. The LSE $\hat{f}_n^{\mathrm{LS}}$ is then defined to be the minimizer of $\sum_{i=1}^n(Y_i-f(X_i))^2$ over $\mathcal{F}$. We assume that the LSEs exist for simplicity.

\begin{theorem}\label{thm:risk_bounds_random_design}
	Let $\mathcal{F}$ be a function class uniformly bounded by some $\Gamma\geq 1/2$.  Let the errors $\epsilon_i$ be i.i.d. subexponential with $\mathbb{E}\exp(2\Gamma \abs{\epsilon_1})\leq \Phi_{\Gamma}^2/2$. Let $\bar{\mu}_\Gamma:=\Phi_\Gamma \vee 4\sqrt{2}\Gamma \exp(4\Gamma^2)$ and $\underline{\mu}_\Gamma:=\Phi_\Gamma \wedge 4\sqrt{2}\Gamma \exp(4\Gamma^2)$. Set
	\beqa\label{cond:risk_bounds_entropy_integral}
	J_{[\,]}(r,f_0,l_\nu)\geq \int_{r^2/192\bar{\mu}_\Gamma }^{2r}\sqrt{\log \mathcal{N}_{[\,]}(\epsilon,S(f_0,r),l_\nu)}\ \d{\epsilon}
	\eeqa
	where $S(f_0,r)\equiv S(f_0,r,l_\nu):=\{f \in \mathcal{F}:l_\nu(f,f_0)\leq r\}$. 
	Suppose $J_{[\,]}(r,f_0,l_\nu)/r^2$ is non-increasing on $(0,\infty)$, and
	\beqa\label{cond:risk_bounds_random_design}
	\sqrt{1/96}\vee (\Gamma/48)\leq \underline{\mu}_\Gamma \leq  \bar{\mu}_\Gamma  \leq \sqrt{n}r_n/18C.
	\eeqa
	Then
	\beqas
	\mathbb{P}\left(l_\nu(\hat{f}_n^{\mathrm{LS}},f_0)>r\right)\leq 2C\sum_{j\geq 1}\exp\left(-\frac{2^{2j}nr^2}{1296C^2\bar{\mu}_\Gamma^2}\right)
	\eeqas
	holds for all $r>r_n$, where $r_n$ is chosen such that 
	\beqas
	\frac{J_{[\,]}(r_n,f_0,l_\nu)}{\sqrt{n}r_n^2}\leq \frac{1}{13C\bar{\mu}_\Gamma}.
	\eeqas
	Consequently, 
	\beqa\label{ineq:risk_bounds_generic}
	\mathbb{E}\big[l_\nu^2(\hat{f}_n^{\mathrm{LS}},f_0)\big]\leq r_n^2 +\frac{864C^3\bar{\mu}_\Gamma^2}{n}.
	\eeqa
	The constant $C$ is taken from Lemma \ref{lem:inequality_sup_empirical_process}.
\end{theorem}

\begin{remark}
We note that this result does not necessarily follow directly from Theorem 9.1 \cite{van2000empirical} since we want to work with bracketing entropy in \emph{continuous} norm, rather than the discrete norm. We also note that (\ref{cond:risk_bounds_random_design}) is actually very weak; in fact we can first require $\Gamma, \Phi_\Gamma$ large enough to ensure the inequality in the far left holds, and then require $n$ large enough to ensure the inequality in the far right holds since typically $\sqrt{n}r_n \to \infty$ at least with logarithmic rates as $n \to \infty$. The risk bound (\ref{ineq:risk_bounds_generic}) has two terms; typically the first term dominates the second for $n$ large.
\end{remark}

Now in order to derive the rate results, since the $l_2$ and $l_\nu$ metrics are essentially the same under our assumptions on $\d{\nu}/\d{\lambda}$, it suffices to study the local geometry of $S(f_0,\delta)$ in terms of $\mathcal{N}_{[\,]}(\epsilon,S(f_0,\delta),l_2)$ for any $\delta>0$. We shall establish the following key estimate for local entropy. 

\begin{lemma}\label{lem:local_entropy_estimate}
	For any convex function $f_0$ defined on $\Omega \in \mathscr{P}_k$, and any $g_0 \in \mathcal{P}_m$, it holds that
	\beqas
	\log\mathcal{N}_{[\,]}(\epsilon,S(f_0,r),l_2)&\leq \mathfrak{C}_d m(m\vee k)^{d(d+4)/4}(2r^2+l_2^2(f_0,g))^{d/4}\\
	&\quad\times \epsilon^{-d/2}\bigg(\log\frac{\mathfrak{C}_d(m\vee k)^{d}(\Gamma^2\vee w(\Omega)^2L^2(g))\abs{\Omega}}{\epsilon^2}\bigg)^{d(d+4)/4}
	\eeqas
	with $\epsilon\leq \mathfrak{C}_{d,0}\min_{1\leq i\leq m}\sqrt{\abs{\Omega_i}}\Gamma$. Here $\{\Omega_i\}$ is a partition of $\Omega$ for which $g$ is affine on each $\Omega_i$. 
\end{lemma}
\begin{remark}
In particular, if $f_0 \in \mathcal{P}_{m_{f_0}}$ for some $m_{f_0}<\infty$, then Lemma \ref{lem:local_entropy_estimate} entails that
\beqas
\log\mathcal{N}_{[\,]}(\epsilon,S(f_0,r),l_2)& \lesssim_{m_{f_0},k,d} \bigg(\frac{r}{\epsilon}\bigg)^{d/2}\times \textrm{poly-logarithmic terms}
\eeqas
so that $\epsilon$ and $r$ scale at the same rate. This property will play a crucial role in deriving nearly parametric rates for the BLSEs in low dimensions in Theorem \ref{thm:rates_conv_lse}.
\end{remark}

We will also need the following result.
\begin{lemma}\label{lem:entropy_cvx_func}[Theorems 1.1 \& 1.3 \cite{gao2015entropy}]
Let $\Omega$ be a convex body on $\R^d$ satisfying Smoothness Assumption \ref{smooth2}, and $\mathcal{C}_p(\Gamma)$ be the collection of convex functions on $\Omega$ with $l_p$-norm bounded by $\Gamma$, $2< p \leq \infty$.
\begin{enumerate}
\item If $\Omega$ can be triangulated into $k$ simplices, then 
\beqa\label{ineq:entropy_polytope_bounded_k}
\log \mathcal{N}(\epsilon,\mathcal{C}_p(\Gamma),l_2)\leq \mathfrak{C}_d k(\abs{\Omega}^{1/2-1/p}\Gamma\epsilon^{-1})^{d/2}.
\eeqa
\item Otherwise for a general smooth convex body $\Omega$, it holds that
\beqa
& \log \mathcal{N}(\epsilon,\mathcal{C}_p(\Gamma),l_2)\\
&\qquad\leq \mathfrak{C}_d
\begin{cases}
(\abs{\Omega}^{1/2-1/p}\Gamma \epsilon^{-1})^{(d-1)} & d\geq 3,\\
(\abs{\Omega}^{1/2-1/p}\Gamma\epsilon^{-1})\abs{\log(\abs{\Omega}^{1/2-1/p}\Gamma\epsilon^{-1})}^{3/2} & d=2.
\end{cases}
\eeqa
\end{enumerate}
When $p = \infty$, the above result can be strengthened to bracketing entropy bounds.
\end{lemma}
When $d=1$, the entropy estimate follows from (\ref{ineq:entropy_polytope_bounded_k}) with $k=1$.

Now we are in position to establish global risk bounds for the BLSE.
\begin{theorem}\label{thm:rates_conv_lse}
	Assume $\pnorm{f_0}{\infty}\leq \Gamma$, and $\hat{f}_n^{\mathrm{LS}}$ be the  BLSE defined via (\ref{eqn:cvx_lse_def}). Then in (\ref{ineq:risk_bounds_generic}) of Theorem \ref{thm:risk_bounds_random_design},  $r_n^2$ is given by the following table:
	\begin{center}
		\begin{tabular}{ |c|c|c|c| } 
			\hline
			$(\Omega,f_0)$ & $(\mathscr{P}_k, \mathcal{P}_{m_{f_0}}(\Gamma))$ & $(\mathscr{P}_k,\mathcal{C}(\Gamma))$  & $(\mathscr{C}, \mathcal{C}(\Gamma))$\\
			\hline
			$d=1$ & $n^{-1}(\log n)^{5/4}$ & \multicolumn{2}{|c|}{$n^{-4/5}$}\\
			\hline
			$d=2$ & \multirow{2}{*}{$n^{-1}(\log n)^{d(d+4)/4}$} & \multirow{2}{*}{$n^{-4/(d+4)}$}           & $n^{-2/3}\log n$\\ 
			\cline{1-1}
			\cline{4-4}
			$d=3$ &                        &                                & $n^{-1/2}\log n$\\ 
			\hline
			$d=4$ & $n^{-1}(\log n)^{10}$                  & $n^{-1/2}\log n$                   & \multirow{2}{*}{$n^{-1/(d-1)}$}\\ 
			\cline{1-3}
			$d\geq 5$ & $n^{-4/d}(\log n)^{d+4}$           & $n^{-2/d}$                &                     \\ 
			\hline
		\end{tabular}
	\end{center}
	for $n$ large enough. Here we only state the dependence of the rates on $n$ (For complete results, see (\ref{eqn:rates_polytope_polyhedral}), (\ref{eqn:rates_polytope_general}) and (\ref{eqn:rates_cvxbody_general})).
\end{theorem}

\begin{remark}\label{rmk:log_factors}
The logarithmic factors in the above table appear for several very different reasons: those in the second column come from logarithmic factors in the local entropy bounds in Lemmas \ref{lem:local_entropy_estimate}. Those in the fourth row of the third column and the third row of the fourth colume come from convergence properties of the entropy integral (\ref{cond:risk_bounds_entropy_integral}) at $0$. It is not yet clear if this is an artifact of the proof techniques or the nature of the estimators. Interestingly, in another different but related setting of global rates of convergence for MLEs of log-concave densities, \cite{kim2014global} also obtained a rate coming with a logarithmic factor in dimension $d=2$. Some potential drawbacks in terms of logarithmic factors resulting from the local entropy method (cf. Lemma \ref{lem:local_entropy_estimate}) will be further discussed in Section \ref{section:discussion_log}.
\end{remark}

\begin{remark}
It is natural to wonder if adaptation happens when the support is smooth. We conjecture that the answer is negative. For further details see Section \ref{section: discussion_adaptive_smooth}.
\end{remark}

\subsection{On the uniform boundedness assumption}\label{section:discussion_uniform_bound}
We show below that a result stated in risk bounds without a uniform boundedness condition is impossible by the following example in dimension $d=1$ due to \cite{balazs2015near}: Let $\Omega=[0,1]$, the regression function $f_0\equiv 0$, and the design $X\sim \mathrm{unif}[0,1]$ and response $Y\sim\mathrm{unif}\{-1,1\}$. Hence the error is subexponential. Then consider the event
\beqas
A&=\{X_1\in[1/4,1/2],X_2\in(1/2,3/4],X_3,\cdots,X_n\in (3/4,1],\\
&\qquad \qquad Y_1=1,Y_2=\cdots=Y_n=-1\}\cap \{X\in[0,1/4]\}
\eeqas
Then the unconstrained least squares estimator is $\hat{f}_n(x)=\frac{2x-X_1-X_2}{X_1-X_2}$ on the interval $[0,3/4]$. Restricting our attention to a smaller interval $[0,1/4]$, we have $
\mathbb{E}\big[(\hat{f}_n(X)-f_0(X))^2\big]\geq \mathbb{E}\bigg[{(2X-X_1-X_2)^2}/{(X_1-X_2)^2}\bigg\lvert A\bigg]\mathbb{P}(A)$. 
Note that $X_1+X_2 \in [0.75,1.25]$ and $-2X\in[-0.5,0]$, hence $X_1+X_2-2X\geq 1/4$. This implies that the right hand side of the above display is bounded below by
$\mathbb{E}\big[{2^{-4}(X_1-X_2)^{-2}}\big\lvert A\big]\mathbb{P}(A)$. Since $\mathbb{P}(A)>0$ and 
$\mathbb{E}\big[{(X_1-X_2)^{-2}}\big\lvert A \big]=\infty$, we see that the risk is unbounded. This stands in sharp contrast to the fixed design setting considered in \cite{guntuboyina2013global} where no uniform boundedness constraint is required due to the fact the boundary effect in risk is killed off by the nature of the discrete $l_2$ norm. 

\section{Model selection and adaptive estimation}\label{section:sieved_lse}
In this section, we study a general  model selection approach that selects among highly non-linear low-dimensional models whose complexity is characterized by the notion of `pseudo-dimension' (as defined below). In the classical regression setup with fixed design and Gaussian errors regression setup, the estimator $\hat{f}_m$ obtained by minimizing the empirical loss typically has risk:
\beqa\label{ineq:generic_risk_bound_fixed_model_heuristic}
\mathbb{E}l_{\underline{X}^n}^2(\hat{f}_m,f_0)\lesssim \inf_{g \in \mathcal{P}_m}l_{\underline{X}^n}^2(f_0,g)+\frac{D_m}{n}.
\eeqa
The task for model selection is to design a data-driven choice of $\hat{m}$ so that approximately the resulting estimator $\hat{f}_{\hat{m}}$ simultaneously achieves the optimal rate for each true $f_0 \in \mathcal{F}$:
\beqa\label{ineq:generic_risk_bound_adaptive_heuristic}
\mathbb{E}l_{\underline{X}^n}^2(\hat{f}_{\hat{m}},f_0)\lesssim \inf_{m \in \N}\bigg(\inf_{g \in \mathcal{P}_m}l_{\underline{X}^n}^2(f_0,g)+\frac{D_m}{n}\bigg),
\eeqa
Here we show in Section \ref{section:general_theory_adaptive_estimation} that results analogous to (\ref{ineq:generic_risk_bound_fixed_model_heuristic}) and (\ref{ineq:generic_risk_bound_adaptive_heuristic}) hold (up to logarithmic factors) by use of the notion of  `pseudo-dimension' when the risks are measured both in discrete and continuous $l_2$ norms. The subtle differences between these two norms in terms of apriori uniform boundedness constraint on the parameters will become clear in the sequel.

To place our results in the context of the existing literature, it is worthwhile to note that when the low-dimensional models are of a certain non-linear type, certain Lipschitz condition has been imposed to increase linearity (cf. Section 3.2.2 \cite{barron1999risk}) and certain coupled entropy conditions under \emph{continuous} $l_2$ and $l_\infty$ norms are required (cf. page 372 \cite{barron1999risk}). On the other hand, combinatorial complexity such as VC dimension is often used in the context of learning theory (cf. Chapter 8 \cite{massart2007concentration}), where the constrast function is usually required to be bounded, which apparently fails in the general regression setup. We refer the readers to \cite{birge1997model}, \cite{barron1999risk}, \cite{massart2007concentration} and \cite{birge2008model} for more details in this direction. However in either case the subtleness between the discrete and continuous norms have not been systematically addressed. 

Our work here can be viewed as occupying the ground between the previous approaches: we provide adaptive procedures in a general regression setup with random design when the function class exhibits certain combinatorial low complexity. Two adaptive procedures are developed. The first procedure is inspired by the idea of bandwidth selection in the context of nonparametric kernel estimate as in \cite{lepskii1992asymptotically} and \cite{lepski1997optimal}; we call this method the L-adaptive procedure. The second method is based on penalized least squares in the spirit of \cite{barron1999risk} and \cite{massart2007concentration}; we call this method the P-adaptive procedure. This framework is particularly interesting in convexity-restricted nonparametric estimation problems. We show in Section \ref{section:best_achievable_rate_sieve} (resp.  Section \ref{section:sieved_set_estimation}) that the `pseudo-dimension' effectively captures the dimension of the class of polyhedral functions (resp. polytopes) within the class of convex functions (resp. convex bodies), and hence nearly rate-optimal estimators that simultaneouly adapt to polyhedral convex funtions (resp. polytopes) are obtained as simple corollaries of our general results. 

\subsection{General theory}\label{section:general_theory_adaptive_estimation}
Consider the regression model (\ref{model}) with $f_0 \in l_2(\nu)$. Assume that the errors $\epsilon_i$ in the regression model are i.i.d. sub-Gaussian with parameter $\sigma^2$, i.e. $\mathbb{E}\exp(t\epsilon_1)\leq \exp(\sigma^2t^2/2)$ unless otherwise specified.

Following \cite{pollard1990empirical} Section 4, a subset $V$ of $\R^d$ is called to have \emph{pseudo-dimension} $t$, denoted as $\mathrm{pdim}(V)=t$, if for every $x \in \R^{t+1}$ and indices $I=(i_1,\cdots,i_{t+1})\in\{1,\cdots,n\}^{t+1}$ with $i_{\alpha}\neq i_{\beta}$ for all $\alpha\neq \beta$, we can always find a sub-index set $J\subset I$ such that no $v \in V$ satisfies both (1) $
v_i> x_i  \textrm{ for all } i \in J$  and $
v_i< x_i  \textrm{ for all } i \in I\setminus J$.

The following lemma is due to \cite{mendelson2003entropy}; see their Theorem 1. The current statement is from \cite{guntuboyina2012optimal}, Theorem B.2.
\begin{lemma}\label{lem:entropy_pdim}
	Let $V$ be a subset of $\R^n$ with $\sup_{v \in V}\pnorm{v}{\infty}\leq B$ and pseudo-dimension at most $t$. Then, for every $\epsilon>0$, we have
	\[\mathcal{N}(\epsilon,A,\pnorm{\cdot}{2})\leq\bigg(4+\frac{2B\sqrt{n}}{\epsilon}\bigg)^{\kappa t},\]
	holds for some absolute constant $\kappa\geq 1$.
\end{lemma}
In the sequel, we shall assume that the constant $\kappa$ is known. We shall also assume the knowledge of $\sigma^2$. \emph{No} effort has been made to obtain optimal constants.

Let $\mathcal{F}\subset l_2(\nu)$ be a function class, and $\{\mathcal{P}_m\subset l_2(\nu)\}_{m \in \N}$ be a sequence of (low-dimensional) models. Typically $\mathcal{P}_m$ is a submodel of $\mathcal{F}$, but this is not apriori required in our theory. Now the key descriptor for $\mathcal{P}_m$ that exhibits low dimensional structure is defined as follows:
\beqa\label{def:F_m}
F(\underline{X}^n;\mathcal{P}_m)\equiv F_m(\underline{X}^n)\equiv \{x \in \R^n: x=\big(g(X_1),\ldots,g(X_n)\big),g \in \mathcal{P}_m\}.
\eeqa

\subsubsection{Risk bounds for fixed models}
We first derive the analogous results in the spirit of (\ref{ineq:generic_risk_bound_fixed_model_heuristic}): We consider risk bounds for the least squares estimator on each model, i.e. for each $m \in \N$, consider the estimator defined by
\beqa
\hat{f}_{\mathcal{P}_m}\equiv \hat{f}_m \in \arg\min_{f \in \mathcal{P}_m}\frac{1}{n}\sum_{i=1}^n(Y_i-f(X_i))^2.
\eeqa
For simplicity we assume the existence of $\hat{f}_m$. The working assumptions are:
\begin{assumption}\label{assump:uniform_bound_cont}
When risk is measured in a continuous norm, we suppose that both $\mathcal{F}$ and $\mathcal{P}_m$ are uniformly bounded by $\Gamma$, and $f_0$ is also bounded by $\Gamma$.
\end{assumption}
\begin{assumption}\label{assump:pdim}
Let $D_m\geq 1$ be such that $D_m\geq \mathrm{pdim}(F(\underline{X}^n;\mathcal{P}_m))$ almost surely.
\end{assumption}
Now we are in position to state our risk bounds and large deviation bounds for fixed models:
\begin{theorem}\label{thm:risk_fixed_model}
Suppose Assumptions \ref{assump:uniform_bound_cont}-\ref{assump:pdim} hold. Let $n\geq 7$. Then,

\noindent \textbf{[Continuous norm.]}
\beqa\label{ineq:fixed_model_bound_deviation_ineq}
l_{\nu}^2(\hat{f}_m,f_0)\leq \mathfrak{c}^c\inf_{g \in \mathcal{P}_m}l_{\nu}^2(f_0,g)+\mathfrak{k}^c\frac{(\sigma^2\vee \Gamma^2) \kappa D_m\log n}{n}+\mathfrak{d}^c\frac{(\sigma^2\vee \Gamma^2) t}{n}
\eeqa
holds with probability at least $(1-4\sum_{j\geq 0}\exp(-2^{j}t/\mathfrak{v})-6\exp(-t))\vee 0$. Furthermore, it holds that
\beqa\label{ineq:fixed_model_bound}
\mathbb{E}\big[l_\nu^2(\hat{f}_m,f_0)\big]\leq \bar{\mathfrak{c}^c}\inf_{g \in \mathcal{P}_m}l_\nu^2(f_0,g)+\bar{\mathfrak{k}^c}\frac{(\sigma^2\vee \Gamma^2)\kappa D_m\log n}{n}.
\eeqa
\noindent \textbf{[Discrete norm.]} 
\beqa\label{ineq:fixed_model_bound_deviation_ineq_discrete}
l_{\underline{X}^n}^2(\hat{f}_m,f_0)\leq \mathfrak{c}^d\inf_{g \in \mathcal{P}_m}l_{\underline{X}^n}^2(f_0,g)+\mathfrak{k}^d\frac{\sigma^2 \kappa D_m\log n}{n}+\mathfrak{d}^d\frac{\sigma^2 t}{n}
\eeqa
holds with probability at least $(1-4\sum_{j\geq 0}\exp(-2^{j}t/\mathfrak{v}))\vee 0$, and 
\beqa\label{ineq:fixed_model_bound_discrete}
\mathbb{E}\big[l_{\underline{X}^n}^2(\hat{f}_m,f_0)\big]\leq \bar{\mathfrak{c}^d}\inf_{g \in \mathcal{P}_m}l_\nu^2(f_0,g)+\bar{\mathfrak{k}^d}\frac{\sigma^2\kappa D_m\log n}{n}.
\eeqa
The numerical constants $\mathfrak{c}^c,\mathfrak{k}^c,\mathfrak{d}^c,\mathfrak{c}^d,\mathfrak{k}^d,\mathfrak{d}^d,\bar{\mathfrak{c}^c},\bar{\mathfrak{k}^c},\bar{\mathfrak{c}^d},\bar{\mathfrak{k}^d}$ can be taken as given in (\ref{const:fixed_model_deviation_ineq_discrete}), (\ref{const:fixed_model_deviation_ineq_cont}), (\ref{const:fixed_model_bound_discrete}) and (\ref{const:fixed_model_bound_cont}), $\mathfrak{v}$ can be found in (\ref{const:fixed_model_prob}), and $\kappa$ is taken from Lemma \ref{lem:entropy_pdim}.
\end{theorem}

\begin{remark}
In view of the minimax lower bound achieved in Theorem \ref{thm:minimax_lower_bounds_simple}, the logarithmic factors in Theorem \ref{thm:risk_fixed_model} cannot be removed. The dependence on $D_m$ is also optimal by considering $\mathcal{P}_m$ to be a linear space of dimension $D_m$.
\end{remark}

\subsubsection{L-adaptive procedure}
Next we establish the analogy for (\ref{ineq:generic_risk_bound_adaptive_heuristic}): For a given sample $\underline{X}^n$, we want to choose a suitable `tuning parameter' $\hat{m}$ so that the expected loss of $l_\nu^2(\hat{f}_{\hat{m}},f_0)$ is about the same magnitude as
\beqa\label{ineq:adaptive_desire_generic}
\inf_{m \in \N}\bigg(\inf_{g \in \mathcal{P}_m}l_\nu^2(f_0,g)+\frac{D_m\log n}{n}\bigg).
\eeqa

In this section we construct a data-dependent scheme for choosing such $\hat{m}$ based on the idea of \cite{lepskii1992asymptotically} and \cite{lepski1997optimal}: we first determine a benchmark choice of the tuning parameter $\mathfrak{M}$ that yields the most conservative risk as obtained in (\ref{ineq:fixed_model_bound}) and (\ref{ineq:fixed_model_bound_discrete}) for general $f_0$, while forcing the tuning parameter to be substantially smaller for $f_0 \in \mathcal{P}_{m_0}$ by comparing the risks of the resulting estimators. Importantly, the benchmark choice $\mathfrak{M}$ should be independent of the oracle information contained in (\ref{ineq:fixed_model_bound}) and (\ref{ineq:fixed_model_bound_discrete}). We consider two cases as follows:

\noindent \textbf{(Case 1).} The approximation error $\inf_{g \in \mathcal{P}_m}l_\nu^2(f_0,g)$ can be separated into knowledge concerning the unknown regression function $f_0$ and the complexity of the model indexed by $m$ via
\beqa\label{ineq:upper_bound_approximation_error}
\inf_{g \in \mathcal{P}_m}l_\nu^2(f_0,g)\leq \mathfrak{e}(f_0,\Omega)\mathfrak{G}(m).
\eeqa
\noindent \textbf{(Case 2).} Otherwise, we use a uniform upper bound:
\beqa\label{ineq:homo_class}
 \inf_{g \in \mathcal{P}_m}l_\nu^2(f_0,g)\leq \sup_{f_0 \in \mathcal{F}}\inf_{g \in \mathcal{P}_m}l_\nu^2(f_0,g).
\eeqa
As we shall see below in Theorem \ref{thm:lepski_generic}, the case (\ref{ineq:upper_bound_approximation_error}) allows the resulting estimator to be risk adaptive to \emph{each} regression function, while such localized information will be lost in the case (\ref{ineq:homo_class}). If $\mathfrak{e}(f_0,\Omega)$ can be controlled uniformly in $f_0 \in \mathcal{F}$, then Case 1 and Case 2 are essentially the same; this is indeed the case when the loss function is the continuous $l_2$ norm in our specific applications in Sections \ref{section:best_achievable_rate_sieve} and \ref{section:sieved_set_estimation}, since a uniform boundedness constraint on the parameter space entails a uniform control of $\mathfrak{e}(f_0,\Omega)$. However, when the loss function is the discrete $l_2$ norm where no uniform boundedness constraint is imposed, $\mathfrak{e}(f_0,\Omega)$ cannot be uniformly controlled, and hence only (\ref{ineq:upper_bound_approximation_error}) will be useful.

To fix notation, let $\ast \in \{up,un\}$ index the cases corresponding to (\ref{ineq:upper_bound_approximation_error}) and (\ref{ineq:homo_class}), and $\#\in \{c,d\}\equiv \{\mathrm{cont},\mathrm{disc}\}$ index continuous and discrete norms, respectively. Let 
\beqa\label{def:L}
\mathfrak{L}^{\#}:=\sigma^2\vee(\Gamma^2\bm{1}_{\#=c}).
\eeqa
We will also use the following simplified notation for norms in definitions and theorem statements in the sequel:
\beqa
l_c(f,g):=l_\nu(f,g),\quad l_d(f,g):=l_{\underline{X}^n}(f,g).
\eeqa
Now we define the benchmark choice
\beqa\label{def:cutoff_M_upper_bound}
\mathfrak{M}_n^{up,\#}&:=\arg\min\bigg\{m\in \N:\mathfrak{G}(m)+\frac{\mathfrak{L}^{\#}\kappa D_m\log n}{n}\bigg\};\\
\mathfrak{M}_n^{un,\#}&:=\arg\min\bigg\{m\in \N:\sup_{f_0\in \mathcal{F}}\inf_{g \in \mathcal{P}_m}l_\nu^2(f_0,g)+\frac{\mathfrak{L}^{\#}\kappa D_m\log n}{n}\bigg\}
\eeqa
in order to balance the approximation error and variance terms in (\ref{ineq:fixed_model_bound}) and (\ref{ineq:fixed_model_bound_discrete}). 

It should be noted here that in both cases for continuous and discrete norms, the definition of the $\mathfrak{M}^{\ast,\#}$ only involves a bias term measured in \emph{continuous} norm.  We make the following assumption on $\mathfrak{M}_n^{\ast,\#}$:
\begin{assumption}\label{assump:M_n}
$\mathfrak{M}_n^{\ast,\#}\leq n$ and $\mathfrak{M}_n^{\ast, \#} \to \infty$ as $n \to \infty$.
\end{assumption}

On the other hand, when $f_0 \in \mathcal{P}_{m_0}$, $\hat{m}^{\ast,\#}$ should not be too large compared with $m_0$. This can be accomplished by risk comparisons as follows:
\beqa\label{def:tuning_lepski}
\hat{m}^{\ast,\#}&:=\min\left\{1\leq m \leq \mathfrak{M}_n^{\#}: l_{\#}^2(\hat{f}_{m},\hat{f}_{m'})\leq \mathfrak{t}^{\ast,\#}\frac{\mathfrak{L}^{\#}\kappa D_{m'}\log n}{n},\forall m'\in\{m,\ldots,\mathfrak{M}_n^{\ast,\#}\}\right\}.
\eeqa
Note that here we use different norms in different cases. The numerical constants $\mathfrak{t}^{\ast,\#}$ are set to be
\beqa\label{def:t_lepski_generic}
\mathfrak{t}^{\#}\equiv  \mathfrak{t}^{\ast,\#}\equiv2(\mathfrak{k}^{\#}+4(\mathfrak{v}\vee 1)\mathfrak{d}^{\#}).
\eeqa
Now we have formally defined the L-adaptive procedures. Before we formally state our results, we will need some more notation:
\beqa
\mathfrak{Bias}^{\ast}(m):=\mathfrak{G}(m)\bm{1}_{\ast=up}+\sup_{f_0 \in \mathcal{F}}\inf_{g \in \mathcal{P}_m}l_\nu^2(f_0,g)\bm{1}_{\ast=un},
\eeqa
and 
\beqa
\mathfrak{e}^{\ast}:=\mathfrak{e}(f_0,\Omega)\bm{1}_{\ast=up}+\bm{1}_{\ast=un}.
\eeqa
\begin{theorem}\label{thm:lepski_generic}
Suppose Assumptions \ref{assump:uniform_bound_cont}-\ref{assump:M_n} hold and $\mathfrak{t}^{\#}$ is set according to (\ref{def:t_lepski_generic}). Suppose further that $\{\mathcal{P}_m\}_{m\in \N}$ is a nested family of submodels of $\mathcal{F}$ with $\mathcal{P}_\infty:=\mathcal{F}$. 

\noindent\textbf{[Continuous norm.]} For $f_0 \in \mathcal{P}_{m_0}(1\leq m_0\leq \infty)$, 
\beqa\label{ineq:lepski_risk_bound_cont}
\mathbb{E}\left[l_\nu^2(\hat{f}_{\hat{m}^{\ast,c}},f_0)\right]&\leq \min\bigg\{2(\mathfrak{t}^{c}+\bar{\mathfrak{k}^c}+44)\frac{(\sigma^2\vee \Gamma^2)\kappa D_{m_0}\log n}{n},\\
&\qquad 2(\bar{\mathfrak{c}^c} \mathfrak{e}^{\ast}+\bar{\mathfrak{k}^c}+\mathfrak{t}^{c})\inf_{m \in \N}\bigg(\mathfrak{Bias}^{\ast}(m)+\frac{(\sigma^2\vee \Gamma^2)\kappa D_{m}\log n}{n}\bigg)\bigg\}
\eeqa
holds for $n\geq \big(\bm{1}_{m_0<\infty}\inf\{n:\mathfrak{M}_n^{\ast,c}\geq m_0\}\big)\vee 7$.

\noindent\textbf{[Discrete norm.]} Suppose that $f_0 \in \mathcal{P}_{m_0}(1\leq m_0 \leq \infty)$ and is uniformly bounded. Then
\beqa\label{ineq:lepski_risk_bound_discrete}
\mathbb{E}\left[l_{\underline{X}^n}^2(\hat{f}_{\hat{m}^{\ast,d}},f_0)\right]&\leq \min\bigg\{2(\mathfrak{t}^{d}+\bar{\mathfrak{k}^d}+84)\frac{(\sigma^2\vee \pnorm{f_0}{\infty}^2)\kappa D_{m_0}\log n}{n},\\
&\qquad 2(\bar{\mathfrak{c}^d} \mathfrak{e}^{\ast}+\bar{\mathfrak{k}^d}+\mathfrak{t}^{d})\inf_{m \in \N}\bigg(\mathfrak{Bias}^{\ast}(m)+\frac{\sigma^2\kappa D_{m}\log n}{n}\bigg)\bigg\}
\eeqa
holds for $n\geq \big(\bm{1}_{m_0<\infty}\inf\{n:\mathfrak{M}_n^{\ast,d}\geq m_0\}\big)\vee 7$.
\end{theorem}
\begin{remark}\label{rmk:set_M_n}
In principle, $\mathfrak{M}^{\ast,\#}$ is defined by (\ref{def:cutoff_M_upper_bound}). In application when logarithmic terms appear in $D_m$, we may simply drop these terms; other factors such as $\sigma^2,\Gamma^2,\kappa$ can also be dropped if our interest is only in the rate. The effect in the final bounds will be up to a constant depending on these dropped factors (usually the order of logarithms in $n$ is correct since $m$ scales at most polynomially in $n$).
\end{remark}

\begin{remark}
The difference between using a continuous norm or a discrete norm mainly lies in the uniform boundedness constraint on the function class $\mathcal{F}$. By our choice of tuning parameters (\ref{def:tuning_lepski}), when risk is assessed with discrete norm, apriori information concerning $\sup_{f \in \mathcal{F}}\pnorm{f}{\infty}$ is not necessary. Our final bound (\ref{ineq:lepski_risk_bound_discrete}) for the discrete norm requires the true regression function to be bounded. This is a condition for sake of simplicity; we can weaken this condition by assuming that $\mathbb{E}[\abs{f_0(X)}^{2+\eta}]<\infty$ for some $\eta>0$ and adjust the constants accordingly.\footnote{In fact we can proceed with H\"older's inequality for the first term in (\ref{ineq:second_term_discrete_1}) and require faster rates of convergence in $n$ in (\ref{ineq:second_term_discrete_2}) by boosting the numerical constant $4$ to larger constant in the definition of $\mathfrak{t}^{\#}$ in (\ref{def:t_lepski_generic}). }
\end{remark}

\subsubsection{P-adaptive procedure}
One main drawback for the L-adaptive procedure is that, the oracle information contained in the approximation error needs to be well separated in the sense of  (\ref{ineq:upper_bound_approximation_error}), or the model needs to be  homogeneous in $f_0\in \mathcal{F}$ in the sense of (\ref{ineq:homo_class}) so that the adaptive procedure can be useful. Below we develop a model selection scheme in the spirit of \cite{barron1999risk} and \cite{massart2007concentration} so that the resulting estimator is rate-adaptive to each $f_0\in \mathcal{F}$, i.e. exactly achieving  (\ref{ineq:adaptive_desire_generic}) up to numerical constants, at the cost of searching for the whole solution path (i.e. search for all $m \in \N$). 
The advantage of this approach will appear in some applications, see Appendix \ref{section:sieved_sobolev}. 

For notational convenience, we denote 
\beqa\label{def:gamma_n}
\gamma_n(g)=\pnorm{g}{n}^2-2\iprod{Y}{g}_n,
\eeqa
where $\pnorm{\cdot}{n}$ stands for $\pnorm{\cdot}{l_{\underline{X}^n}}$ and $\iprod{Y}{g}_n={1 \over n}\sum_{i=1}^nY_ig(X_i)$. 

\begin{theorem}\label{thm:adaptive_estimator_bound_discrete}
Suppose Assumption \ref{assump:pdim} holds and the errors $\epsilon_i$'s are i.i.d $\mathcal{N}(0,\sigma^2)$. Let $n\geq 7$ and $L_m$ be a sequence of numbers so that $\sum_{m \in \N}\exp(-L_mD_m)=\Sigma<\infty$.  Let the penalty function 
\beqas
\mathrm{pen}(m):=\frac{\mathfrak{c}_{p,1}\sigma^2D_{m}}{n}\big(\mathfrak{c}_{p,2} \kappa \log n+L_m).
\eeqas
where $\kappa$ is the constant in Lemma \ref{lem:entropy_pdim}, and $\mathfrak{c}_{p,i} (i=1,2)$ are absolute constants which can be found in (\ref{const:p_adaptive}). Let the model selection criteria be defined by
\beqas
\hat{m}^{ms}:=\argmin_{m \in \N} \bigg(\gamma_n(\hat{f}_m)+\mathrm{pen}(m)\bigg).
\eeqas
 Then
\beqa\label{ineq:adaptive_estimator_bound_discrete}
\mathbb{E}\big[l_{\underline{X}^n}^2(\hat{f}_{\hat{m}^{ms}},f_0)\big]\leq \inf_{m \in \N}\bigg(3\inf_{g \in \mathcal{P}_m}l_\nu^2(f_0,g)+\mathfrak{h}_{\kappa,\Sigma} \frac{\sigma^2L_m D_m\log n}{n}\bigg).
\eeqa
where $\mathfrak{h}_{\kappa,\Sigma}$ is a constant depending only on $\kappa,\Sigma$. For the explicit form of this constant, see (\ref{const:adaptive_generic_discrete}).
\end{theorem}

\begin{remark}
Gaussianity of the errors is assumed in Theorem \ref{thm:adaptive_estimator_bound_discrete} since we rely on Gaussian process techniques in the proof. We suspect that new tools like tail control of weighted empirical process connecting discrete and continuous norms, in the spirit of Theorem 5 in \cite{birge1998minimum}, or the more general Proposition 7 in \cite{barron1999risk} are needed to establish corresponding results in the  continuous norm. 
\end{remark}

\subsubsection{On the uniform boundedness assumption}
Finally we comment on the uniform boundedness assumption when risks are measured in continuous $l_2$ norm in a regression model under random design. Previous work and results (cf. \cite{juditsky2000functional}, \cite{wegkamp2003model}, \cite{yang1999aggregating}) all require some boundedness assumption on both the parameter and the estimators. It is shown in Proposition 4 in \cite{birge2006model} and Proposition 3 in \cite{birge2008model} that such apriori uniform boundedness constraint is actually \emph{necessary} for any \emph{universal} risk bounds measured in continuous $l_2$ loss in the density estimation setting. Notably, in the specific case for estimating Besov spaces $B^\alpha_{p,\infty}([0,1])$ with $\alpha>\alpha_l$ for some $\alpha_l\in(1/p-1/2,1/p)$ when $1\leq p<2$, it is shown in \cite{baraud2002model} that the usual rate $n^{-2\alpha/(2\alpha+1)}$ can be recovered in squared continuous $l_2$ norm without the uniform boundedness assumption, while the rate becomes $n^{-1+2(1/p-\alpha)}$ for $\alpha\in(1/2-1/p,\alpha_l)$ as shown in \cite{birge2004model}. This stands in sharp constrast with the results obtained in fixed design setting in the classical paper \cite{donoho1994ideal}, where the usual rate $n^{-2\alpha/(2\alpha+1)}$ is observed for all $\alpha>1/p-1/2$ in discrete norms. To remedy this problem, \cite{birge2004model} showed that in the above specific example, the usual rate can be recovered down to $\alpha>1/p-1/2$ without a uniform boundedness assumption by using Hellinger metric. Further results in this direction can be found in \cite{baraud2011estimator}. 

\subsection{Application in multivariate convex regression}\label{section:best_achievable_rate_sieve}
With the general methods developed in the previous section, we study adaptive estimation in the specific context of multivariate convex regression. To this end, we will need to (i) control the pseudo-dimension defined in (\ref{def:F_m}); (ii) control the approximation error $\inf_{g \in \mathcal{P}_m}l_\nu^2(f,g)$. This is accomplished in the following lemmas.
\begin{lemma}\label{lem:pdim_P_m}
$\mathrm{pdim}\big(F(\underline{X}^n;\mathcal{P}_m)\big)\leq 6md \log 3m$.
\end{lemma}
\begin{lemma}\label{lem:approximation_error_P_m}
Suppose $f_0$ is Lipschitz continuous (and hence $f_0$ is necessarily bounded). Then
\beqas
\inf_{g \in \mathcal{P}_m}l_\nu^2(f_0,g)\leq \inf_{g \in \mathcal{P}_m(\pnorm{f_0}{\infty})}l_\nu^2(f_0,g)\leq \mathfrak{K}_{d,f_0,\Omega} m^{-4/d}.
\eeqas
For the explicit form of this constant, see (\ref{const:approximation_error}).
\end{lemma}

For simplicity of notation we assume that the true regression function $f_0$ is bounded by $\Gamma$, and we shall content ourselves by discussing estimators that adapt to \emph{each} regression function (i.e. the case (\ref{ineq:upper_bound_approximation_error}) in L-adaptive procedure and P-adaptive procedure), and only be interested in the dependence of the risk in terms of the sample size $n$.

\noindent\textbf{(L-adaptive procedure).} By Lemma \ref{lem:approximation_error_P_m}, define $\mathfrak{G}(\cdot)$ in (\ref{ineq:upper_bound_approximation_error}) in the multivariate convex regression setting to be $\mathfrak{G}(m):=m^{-4/d}$. Then we can set $\mathfrak{M}_n:=n^{d/(d+4)}$(see Remark \ref{rmk:set_M_n}), and define the data-driven tuning parameters according to (\ref{def:tuning_lepski}) as follows:
\beqa
\hat{m}^{\#}&:=\min \bigg\{1\leq m \leq \mathfrak{M}_n: l_r^2\big(\hat{f}_{\mathcal{P}_m(\Gamma^{\#})},\hat{f}_{\mathcal{P}_{m'}(\Gamma^{\#})}\big)\leq 6d\mathfrak{t}^{\#}\frac{\mathfrak{L}^{\#}\kappa m'\log(3m')\log n}{n},\\
&\qquad\qquad \forall m'\in\{m,\ldots,\mathfrak{M}_n\}\bigg\},\\
\eeqa
where $\mathfrak{L}^{\#}$ is defined in (\ref{def:L}) and 
\beqa\label{def:gamma}
\Gamma^{\#}:=\Gamma\bm{1}_{\#=c}+\infty\bm{1}_{\#=d}.
\eeqa
Now by Theorem \ref{thm:lepski_generic} we immediately get the following result.
\begin{corollary}\label{cor:adaptive_lepski_multivariate_convex_fcn}
	Let $t^{\#}$ be defined as in Theorem \ref{thm:lepski_generic}, and $f_0 \in \mathcal{P}_{m_0}$ be uniformly bounded by $\Gamma$. If $m_0=\infty$, further suppose $f_0$ is Lipschitz continuous. Then,
	\beqas
	\mathbb{E}\big[l_{\#}^2(\hat{f}_{\mathcal{P}_{\hat{m}^{\#}}(\Gamma^{\#})},f_0)\big]&\leq \mathfrak{C}_{d,\kappa,f_0,\Omega,\sigma,\Gamma} \min\bigg\{\frac{m_0\log(3m_0)\log n}{n},n^{-4/(d+4)}(\log n)^{8/(d+4)}\bigg\}
	\eeqas
	holds for $n\geq \max\{m_0^{(d+4)/d}\bm{1}_{m_0<\infty},7\}$. Here $\mathfrak{C}_{d,\kappa,f_0,\Omega}$ is a constant depending on $d,\kappa,f_0,\Omega,\sigma,\Gamma$.
\end{corollary}

\noindent\textbf{(P-adaptive procedure).} We now apply Theorem \ref{thm:adaptive_estimator_bound_discrete} to obtain another adaptive estimator as follows.
\begin{corollary}\label{cor:adaptive_birge_multivariate_convex_fcn}
Define the model selection criteria to be 
\beqas
\hat{m}^{ms}:=\argmin_{m \in \N}\bigg\{{1 \over n}\sum_{i=1}^n \hat{f}_m(X_i)^2-{2 \over n}\sum_{i=1}^n Y_i\hat{f}_m(X_i)+\frac{6\mathfrak{c}_{p,1}\sigma^2md\log(3m)}{n}\big(\mathfrak{c}_{p,2}\kappa \log n+1\big)\bigg\}.
\eeqas
If $m_0=\infty$, further suppose $f_0$ is Lipschitz continuous. Then for $f_0 \in \mathcal{P}_{m_0}$, 
\beqas
\mathbb{E}\big[l_{\underline{X}^n}^2(\hat{f}_{\hat{m}^{ms}},f_0)\big]\leq \mathfrak{C}_{d,\kappa,\sigma,f_0}\min\bigg\{\frac{m_0\log(3m_0)\log n}{n}, n^{-4/(d+4)}(\log n)^{8/(d+4)}\bigg\}
\eeqas
where $\mathfrak{C}_{d,\kappa,\sigma,f_0}$ is a constant depending only on $d,\kappa,\sigma$.
\end{corollary}

\subsection{Application in estimation of an unknown convex set from support function measurements}\label{section:sieved_set_estimation}
To further illustrate the applicability of our general framework derived in Section \ref{section:general_theory_adaptive_estimation}, we consider the problem of nonparametric estimation of a compact, convex set $K\subset \R^d$ from noisy support function measurements. Here the \emph{support function} $h_K$ of a compact convex set $K$ is a real-valued function defined on the unit sphere $\mathbb{S}^{d-1}:=\{u \in \R^d:\pnorm{u}{2}=1\}$ by $h_K(u):=\sup_{x \in K} \iprod{x}{u}_d$ for $u \in \mathbb{S}^{d-1}$. We observe $(U_i,Y_i)$ drawn according to the model
\beqa
Y_i=h_K(U_i)+\epsilon_i,\quad i=1,\ldots,n
\eeqa
where $U_1,\ldots,U_n$'s are i.i.d. generated from a probability measure $\nu$ on $\mathbb{S}^{d-1}$, and $\epsilon_i$ are i.i.d. sub-Gaussian errors with parameter $\sigma^2$. 
To put the problem into our general setup, for $\Gamma\leq \infty$, denote $\bar{\mathcal{F}}(\Gamma):=\{h_K:K\textrm{ convex body}, K\subset B_2(0,\Gamma)\}$, and $\bar{\mathcal{P}}_m(\Gamma):=\{h_P\in \bar{\mathcal{F}}(\Gamma): P \textrm{ is a polytope with }m\textrm{ vertices}\}$. For notational convenience, $\bar{\mathcal{P}}_\infty:=\bar{\mathcal{F}}$. Notational dependence on $\Gamma$ is often omitted when $\Gamma=\infty$. Now the loss functions for two convex bodies $K,K'$ in continuous and discrete norms become 
\beqa\label{def:loss_fun_set_estimation_cont}
l_c^2(K,K')\equiv l_h^2(K,K')\equiv \int_{\mathbb{S}^{d-1}}\big(h_K(u)-h_{K'}(u)\big)^2\ \d{\nu(u)},
\eeqa
and
\beqa\label{def:loss_fun_set_estimation_discrete}
l_d^2(K,K')\equiv l_{\underline{U}^n}^2(K,K')\equiv {1 \over n}\sum_{i=1}^n \big(h_K(U_i)-h_{K'}(U_i)\big)^2.
\eeqa
The least squares criteria over submodels is simply
\beqas
\hat{K}_m(\Gamma):=\argmin_{K \in \bar{\mathcal{P}}_m(\Gamma)}\sum_{i=1}^n(Y_i-h_K(U_i))^2.
\eeqas
Theoretical advances have been pioneered by \cite{gardner2006convergence}, who showed consistency of the least squares estimator and derived rates of convergence of the estimators under fixed design. 
\cite{guntuboyina2012optimal} studied minimax optimal rates of convergence under both fixed and random designs. In the special case for dimension $2$, \cite{cai2015adaptive} developed adaptive estimators based on point estimators $\{\widehat{h(u_i)}\}_{i=1^n}$ for a uniform grid $u_1,\ldots,u_n$ on unit circle under both loss functions (\ref{def:loss_fun_set_estimation_cont}) and (\ref{def:loss_fun_set_estimation_discrete}). Here using the general framework established in Section \ref{section:general_theory_adaptive_estimation}, we obtain another adaptive estimator in arbitrary dimensions.

Our key observations are given by the following two results.
\begin{lemma}\label{lem:pseudo_dim_set_estimation}
$\mathrm{pdim}\big(F(\underline{U}^n;\bar{\mathcal{P}}_m)\big)\leq 3md\log 3m.$
\end{lemma} 
\begin{lemma}\label{lem:approximation_error_set_estimation}
Let $K$ be any convex body.
\beqas
\inf_{P \in \bar{\mathcal{P}}_m}\bigg(l_{h}^2(K,P)\vee l_{\underline{U}^n}^2(K,P)\bigg)\leq \mathfrak{h}_d \abs{K}^2 m^{-4/(d-1)},
\eeqas
where $\mathfrak{h}_d$ is a constant only depending on $d$.
\end{lemma}

Suppose $K$ is a convex body so that $K\subset B_2(0,\Gamma)$. 

\noindent\textbf{(L-adaptive procedure).} Now according to (\ref{def:cutoff_M_upper_bound}), define $\bar{\mathfrak{M}}_n:=n^{(d-1)/(d+3)}$ and set the tuning parameters according to (\ref{def:tuning_lepski}) as follows:
\beqa
\hat{m}^{\#}&:=\min \bigg\{1\leq m \leq \bar{\mathfrak{M}}_n: l_r^2(\hat{K}_m(\Gamma^{\#}),\hat{K}_{m'}(\Gamma^{\#}))\leq 3d\mathfrak{t}^{\#}\frac{\mathfrak{L}^{\#}\kappa m'\log(3m')\log n}{n},\\
&\qquad\qquad \forall m'\in\{m,\ldots,\bar{\mathfrak{M}}_n\}\bigg\},\\
\eeqa
where $\mathfrak{L}^{\#}$ is defined in (\ref{def:L}), and $\Gamma^{\#}$ defined in (\ref{def:gamma}).
\begin{corollary}\label{cor:adaptive_lepski_set_estimation}
Let $\mathfrak{t}^{\#}$ be defined as in Theorem \ref{thm:lepski_generic}, and $K \in \bar{\mathcal{P}}_{m_K}$ contained in a ball with radius $\Gamma$. Then for $\#\in \{c,d\}$,
\beqas
\mathbb{E}\big[l_{\#}^2(\hat{K}_{\hat{m}^{\#}}(\Gamma^{\#}),K)\big]\leq \mathfrak{C}_{d,\kappa,\sigma,\Gamma}\min\bigg\{\frac{m_K\log(3m_K)\log n}{n},n^{-4/(d+3)}(\log n)^{8/(d+3)}\bigg\}
\eeqas
holds for $n\geq \max\{m_K^{(d+3)/(d-1)}\bm{1}_{m_K<\infty},7\}$. 
\end{corollary}
\noindent\textbf{(P-adaptive procedure).} Similarly by Theorem \ref{thm:adaptive_estimator_bound_discrete},
\begin{corollary}\label{cor:adaptive_birge_set_estimation}
Define the model selection criteria to be 
\beqas
\hat{m}^{ms}:=\argmin_{m \in \N}\bigg\{{1 \over n}\sum_{i=1}^n h_{\hat{K}_m}(U_i)^2-{2 \over n}\sum_{i=1}^n Y_ih_{\hat{K}_m}(U_i)+\frac{3\mathfrak{c}_{p,1}\sigma^2md\log(3m)}{n}\big(\mathfrak{c}_{p,2}\kappa \log n+1\big)\bigg\}.
\eeqas
Then for $K \in \bar{\mathcal{P}}_{m_K}$ ,
\beqas
\mathbb{E}\big[l_d^2(\hat{K}_{\hat{m}^{ms}},K)\big]\leq \mathfrak{C}_{d,\kappa,\sigma}\min\bigg\{\frac{m_K\log(3m_K)\log n}{n}, n^{-4/(d+3)}(\log n)^{8/(d+3)}\bigg\}.
\eeqas
\end{corollary}
\begin{remark}
By Theorems 4.1 and 4.2 in \cite{guntuboyina2012optimal}, the minimax optimal rates of convergence for uniform probability measure on $\mathbb{S}^{d-1}$ is $n^{-4/(d+3)}$. The lower bounds hold for arbitrary measures. Here by Corollary \ref{cor:adaptive_lepski_set_estimation} and Corollary \ref{cor:adaptive_birge_set_estimation} we achieve the lower bound within a poly-logarithmic factor.
\end{remark}

\section{Discussion}\label{section:discussion}
In this section, we will discuss some related problems.
\subsection{Adaptation of the LSE when the support is smooth}\label{section: discussion_adaptive_smooth}
In Section \ref{section:lse}, it is shown that the least squares estimator (LSE) adapts to polyhedral functions when the support is polytopal, while the sieved least squares estimators proposed in Section \ref{section:sieved_lse} are rate-adaptive to regular subclasses whatever the shape of the support. Hence it is natural to ask: (1) Do the LSEs adapt to polyhedral functions when the support is smooth? (2) Do the LSEs adapt to regular subclasses when the support is smooth? We will discuss the above questions separately in Sections \ref{section:discussion_adapt_simple} and \ref{section:discussion_adapt_regular}.
\subsubsection{Adaptation to the class of simple convex functions}\label{section:discussion_adapt_simple}
We observe that at the technical level, the nearly parametric rate when the support is polytopal and the regression function is polyhedral is achieved via the nice local property of the entropy characterized in Lemma \ref{lem:local_entropy_estimate}.  Similarly, when the regression function is polyhedral and the support is smooth, in order that adaptation occurs we would like an estimate of the form
\beqa\label{ineq:key_adaptive_entropy_estimate}
\log \mathcal{N}_{[\,]}(\epsilon,\mathcal{C}_2(r)\cap \mathcal{C}(\Gamma),l_2)\lesssim \bigg(\frac{r}{\epsilon}\bigg)^{d-1}\times \textrm{poly-logarithmic terms}.
\eeqa
in view of Theorem \ref{lem:entropy_cvx_func} when $r$ is fixed. Recall that in the proof of Theorem \ref{thm:minimax_lower_bound_cvxbody}, a class of convex functions $\{f_\tau\}$ is constructed so that (1) the cardinality equals to $2^m$ where $m\approx \eta^{-(d-1)/(d+1)}$; (2) each function satisfies $\pnorm{f_\tau}{l_2}^2\lesssim m\eta $ and $\pnorm{f_\tau}{\infty}\leq \Gamma$; (2) the distance between any pair $(f_\tau,f_{\tau'})$ under squared $l_2$ norm is at least $\eta$. Now set $m\eta \approx \eta^{2/(d+1)}\approx \epsilon^2$, giving $\eta\approx \epsilon^{(d+1)}$, we see that $m\approx \epsilon^{-(d-1)}$, and hence the cardinality is $\exp(\epsilon^{-(d-1)})$. This means that for smooth support, 
\beqa\label{ineq:entropy_adapt}
\log \mathcal{N}_{[\,]}(\epsilon^{(d+1)/2},\mathcal{C}_2(\epsilon)\cap \mathcal{C}(\Gamma),l_2)\gtrsim \epsilon^{-(d-1)}.
\eeqa
On the other hand, (\ref{ineq:key_adaptive_entropy_estimate}) reduces to
\beqas
\log \mathcal{N}_{[\,]}(\epsilon^{(d+1)/2},\mathcal{C}_2(\epsilon)\cap \mathcal{C}(\Gamma),l_2)\lesssim \epsilon^{-(d-1)^2/2}\times \textrm{poly-logarithmic terms},
\eeqas
which violates (\ref{ineq:entropy_adapt}) when $d=2$.  This suggests that within the current local entropy method by searching for bounds of form (\ref{ineq:key_adaptive_entropy_estimate}) shall not work for the smooth support in $d=2$ at least. 
\subsubsection{Adaptation to the class of regular convex functions}\label{section:discussion_adapt_regular}
We restrict our attention to $d=1$ and assume without loss of generality that $\Omega=[0,1],\Gamma=1$ and $\nu\equiv \mathrm{Unif}[0,1]$. Let the true regression function be $f_0\equiv 0$. 
\begin{lemma}\label{lem:inconsist_blse_zero}
Suppose the errors are the same as in Section \ref{section:lse}. Then $\hat{f}_n$ converges uniformly to $f_0\equiv 0$ on any compact set within $(0,1)$. Moreover, $\hat{f}_n(0)\nrightarrow_p 0$.
\end{lemma}
The above lemma implies that the Lipschitz constant for $\hat{f}_n$ blows up with non-trivial probability. It suggests that, when the support is smooth, it is unlikely that the LSE can adapt locally to regular convex funtions when the underlying true regression function has bounded Lipschitz constant to achieve a faster rate $n^{-4/(d+4)}$ as observed in the sieved adaptive estimators in Corollaries \ref{cor:adaptive_lepski_multivariate_convex_fcn} and \ref{cor:adaptive_birge_multivariate_convex_fcn}.

\subsection{On global and local smoothness}
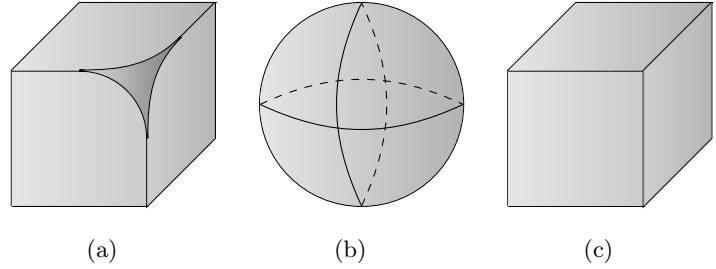
\begin{figure}
	\begin{tikzpicture}[xscale=0.6,yscale=0.6,>=latex]
	\path[thick,draw] (0,0)--(0,3)--(1.5,3) to [out=0, in=90] (3,1.5)--(3,0)--(0,0);
	\shade[left color=gray!20, right color=gray!40] (0,0)--(0,3)--(1.5,3) to [out=0, in=90] (3,1.5)--(3,0)--(0,0);
	\path[thick,draw] (0,3)--(1.5,4.5)--(4.5,4.5)--(3.75,3.75) to [out=-135, in=0](1.5,3)--(0,3);
	\shade[left color=gray!20, right color=gray!60] (0,3)--(1.5,4.5)--(4.5,4.5)--(3.75,3.75) to [out=-135, in=0](2,3)--(0,3);
	\path[thick,draw] (4.5,4.5)--(4.5,1.5)--(3,0)--(3,1.5) to [out=90, in=-135](3.75,3.75)--(4.5,4.5);
	\shade[left color=gray!40, right color=gray!60] (4.5,4.5)--(4.5,1.5)--(3,0)--(3,1.5) to [out=90, in=-135](3.75,3.75)--(4.5,4.5);
	\path[thick,draw](3.75,3.75) to [out=-135, in=0](1.5,3) to [out=0, in=90] (3,1.5) to [out=90, in=-135](3.75,3.75);
	\shade[left color=gray!30, right color=gray!100] (3.75,3.75) to [out=-135, in=0](1.5,3) to [out=0, in=90] (3,1.5) to [out=90, in=-135](3.75,3.75);
	
	\path[thick,draw](7.75,0) to [out=180, in =270] (5.5,2.25) to [out=90, in=180](7.75,4.5) to [out=0, in=90](10,2.25) to [out=-90,in=0](7.75,0);
	\shade[left color=gray!20, right color=gray!60] (7.75,0) to [out=180, in =270] (5.5,2.25) to [out=90, in=180](7.75,4.5) to [out=0, in=90](10,2.25) to [out=-90,in=0](7.75,0);
	\path[draw](7.75,0) to [out=115, in=-115] (7.75,4.5);
	\draw[dashed](7.75,0) to [out=65, in=-65] (7.75,4.5);
	\path[draw](5.5,2.25) to[out=-25, in =-155] (10,2.25);
	\draw[dashed] (5.5,2.25) to[out=25, in =155] (10,2.25);
	
	\path[thick,draw] (11,0)--(11,3)--(14,3)--(14,0)--(11,0);
	\shade[left color=gray!20, right color=gray!40] (11,0)--(11,3)--(14,3)--(14,0)--(11,0);
	\path[thick,draw] (11,3)--(12.5,4.5)--(15.5,4.5)--(14,3)--(11,3);
	\shade[left color=gray!20, right color=gray!60](11,3)--(12.5,4.5)--(15.5,4.5)--(14,3)--(11,3);
	\path[thick,draw] (15.5,4.5)--(15.5,1.5)--(14,0)--(14,3)--(15.5,4.5);
	\shade[left color=gray!40, right color=gray!60] (15.5,4.5)--(15.5,1.5)--(14,0)--(14,3)--(15.5,4.5);
	
	\node at (2,-1) {(a)};
	\node at (7.5,-1) {(b)};
	\node at (13,-1) {(c)};
	\end{tikzpicture}
	\caption{(a) polytope with partially smooth boundary; (b) smooth body; (c) polytope.}
	\label{pic:illustration_smooth}
\end{figure}
The Smoothness Assumptions \ref{smooth1} and \ref{smooth2} are imposed at a global scale, corresponding to the case (b) in Figure \ref{pic:illustration_smooth}. We assumed this for simplicity of statements; the slower minimax risks and risk bounds in squared $l_2$ loss of order $O(n^{-1/2})=O(n^{-2/(d+1)})(d=3)$ in Theorems \ref{thm:minimax_lower_bound_cvxbody} and \ref{thm:rates_conv_lse} also apply to the case (a), which is a polytope with one smooth corner locally satisfying Smoothness Assumptions \ref{smooth1} and \ref{smooth2}. On the other hand, case (c) (a polytope) corresponds to squared risks of order $O(n^{-4/7})=O(n^{-4/(d+4)})(d=3)$ as stated in Theorems \ref{thm:minimax_lower_bound_poly} and \ref{thm:rates_conv_lse}.

\subsection{Redundancy of logarithmic factors}\label{section:discussion_log}
It is discussed in Remark \ref{rmk:log_factors} that the logarithmic factors obtained in the risks of the BLSE come from very different reasons. While it is unclear whether these factors can be reduced or not, 
we will show that they are actually redundant in the fixed design setting considered by \cite{guntuboyina2013global}\footnote{During the preparation of the paper we become aware of the independent work of \cite{chatterjee2015improved} who derived essentially the same conclusion as our Theorem \ref{thm:lse_without_log}.}. We show this by considering a simplified fixed design $\{x_k:=\frac{k-1}{n-1}\}_{k=1}^n$. The noise level is assumed to be $\sigma^2=1$ for simplicity. Extension to almost equi-distributed design points is immediate.
\begin{theorem}\label{thm:lse_without_log}
	Let $f_0$ be the ground convex function, and $\hat{f}_n$ be the LSE of $f_0$ considered in \cite{guntuboyina2013global}. Then
	\beqas
	\mathbb{E}l_{\underline{X}^n}^2(\hat{f}_n,f_0)\leq \mathfrak{C}(1+\pnorm{f_0}{\infty})^{2/5}n^{-4/5},
	\eeqas
	when $n\geq \min \{n \in \N:(1+\pnorm{f_0}{\infty})^{1/10} \log n\leq \mathfrak{C}' n^{1/5}\}$. Here $\mathfrak{C},\mathfrak{C}'$ are absolute constants.
\end{theorem}

The proof makes use of a recent result by \cite{chatterjee2014new}, which will be detailed in Section \ref{section:proof_withoutlog_lse}. This shows that LSE without boundedness constraints for univariate convex regression achieves exact optimal rates of convergence for general convex functions, in view of the lower bounds in \cite{guntuboyina2013global} under discrete $l_2$ norm. This also suggests that the local entropy method may have logarithmic losses in deriving the risk bounds.

\appendix

\section{Further application of the general framework in Section 4}\label{section:sieved_sobolev}
In this section, we shall further illustrate the applicability of our general framework developed in Section \ref{section:general_theory_adaptive_estimation} by considering classical Sobolev spaces. To start with, we shall identify the class $l_2(\nu)$ consisting of all square integrable functions under measure $\nu$ with the class of $\mathfrak{l}_2$ of square summable infinite sequences by selecting a suitable orthonormal basis $\{\varphi_j\}_{j \in \N}$, and mapping $f \in l_2(\nu)$ to $\beta_f \in \R^\N$ with $(\beta_f)_j = \int f\varphi_j \d{\nu}$. Then the Sobolev class $W^\alpha([0,1])$ can be identified as all infinite sequences $\beta \in \R^\N$ so that $\sum_{i \in \N} j^{2\alpha}\beta_j^2<\infty$. Let $\mathcal{F}:=\cup_{\alpha>0} W^\alpha([0,1])$, and we assume that $f_0 \in \mathcal{F}$. Let the low dimensional models $\mathcal{P}_m$ be the infinite sequences $\beta$'s for which $\beta_j=0$ if $j>m$. Then $F_m(\underline{X}^n)$ has Euclidean dimension at most $m$ and hence psuedo-dimension at most $m$ (cf. page 15 in \cite{pollard1990empirical}). Now since $f_0 \in \mathcal{F}$, by definition we see that $f \in W^\alpha([0,1])$ for some $\alpha>0$.  With $\pnorm{f}{\alpha}:=\sum_{j=1}^\infty j^{2\alpha}(\beta_f)_j^2$, the approximation error term can be bounded from above by
\beqa
\inf_{g \in \mathcal{P}_m}l_2^2(f_0,g)\leq \sum_{j\geq m+1} (\beta_{f_0})_j^2\leq m^{-2\alpha}\sum_{j\geq m+1} j^{2\alpha}(\beta_{f_0})_j^2\leq m^{-2\alpha}\pnorm{f_0}{\alpha}.
\eeqa
Note that here the smoothness information $\alpha$ and the complexity of approximation class $m$ cannot be decoupled as (\ref{ineq:upper_bound_approximation_error}), and the function class is not homogeneous in classes with different degrees of smoothness so (\ref{ineq:homo_class}) does not apply either. Hence we turn to Theorem  \ref{thm:adaptive_estimator_bound_discrete}. To this end, define
\beqa
\hat{\beta}_{m}\in \argmin_{\beta \in \R^m}\sum_{i=1}^n\bigg(Y_i-\sum_{j=1}^m\beta_j \varphi_j(X_i)\bigg)^2.
\eeqa
Then the estimator is $\hat{f}_m:=\sum_{j=1}^m (\hat{\beta}_m)_j \varphi_j$. Note even if $\hat{\beta}_m$ is not unique, $(\hat{\theta}_m)_i:=\hat{f}_m(X_i)=\sum_{j=1}^m (\hat{\beta}_m)_j \varphi_j(X_i)$ is unique since $\hat{\theta}_m$ is the projection of the vector $Y=(Y_1,\ldots,Y_n)$ onto the linear space  $\{\bm{\Lambda}\beta:\beta \in \R^m\}$ where $\bm{\Lambda}=(\Gamma_{ij})\in \R^{n\times m}$ is defined by $\Gamma_{ij}:=\varphi_j(X_i)$.
\begin{corollary}
	Let the model selection criteria be defined by
	\beqa
	\hat{m}^{ms}:=\argmin_{m \in \N}\bigg({1 \over n}\sum_{i=1}^n(\hat{\theta}_m)_i^2-\frac{2}{n}\sum_{i=1}^n Y_i(\hat{\theta}_m)_i+\frac{16\sigma^2 m}{n}(72\kappa \log n+1)\bigg).
	\eeqa
	Then
	\beqa
	\mathbb{E}\big[l_{\underline{X}^n}^2(\hat{f}_{\hat{m}^{ms}},f_0)\big]\leq \mathfrak{C}_{\kappa,\sigma,\pnorm{f_0}{\alpha}} \left(\frac{\log n}{n}\right)^{2\alpha/(2\alpha+1)}.
	\eeqa
\end{corollary}
It is well-known that $n^{-2\alpha/(2\alpha+1)}$ is the minimax rates of convergence for estimating $f_0 \in W^\alpha([0,1])$ so our estimator achieves optimality up to some logarithmic factors. It is worthwhile to note that adaptation in Sobolev balls and more general Besov balls can be achieved without additional logarithmic factors. See Section 4.3.5 \cite{massart2007concentration} for more details.

\section{Proofs for Section 1}\label{section:proof_intro}
\begin{proof}[Proof of Lemma \ref{lemma:smooth}]
	The case for $d=1$ for trivial so we shall assume $d\geq 2$. We first prove the claim when $\Omega$ is the unit ball $B_d$. Fix some positive $\eta>0$. Consider spherical caps of $B_d$ of height $\eta$. Then the ($d-1$)-dimensional area of such caps are on the order of $\eta^{(d-1)/2}$ and hence we can find $m\lesssim_d\eta^{-(d-1)/2}$ many disjoint caps with center of the cap denoted as $\{x_1,\cdots,x_m\}$. Let $w(\eta)$ denote the $d$-dimensional volume of the spherical cap with height $\eta$. It is well known that $w(\eta)=\frac{\pi^{d/2}}{2\Gamma(1+d/2)}I_{2\eta-\eta^2}\big((d+1)/2,1/2\big)$ where $I_t$ is the regularized incomplete beta function. Thus we can take
	$w(\eta)\lesssim_d\eta^{(d+1)/2}$. Now let $\epsilon=\eta^{(d+1)/2}$ and thus $m\lesssim_d \epsilon^{-(d-1)/(d+1)}$. This shows that $B_d$ satisfies Smoothness Assumption \ref{smooth1}. That $B_d$ satisfies Smoothness Assumption $\ref{smooth2}$ follows by the same argument in Section 2.8 of \cite{gao2015entropy}. Now for any ellipsoid, we can find an affine transformation mapping it to the unit ball.
\end{proof}
\section{Proofs for Section 2}\label{section:proof_minimax}
\begin{proof}[Proof of Theorem \ref{thm:generic_minimax_upper_bound_exponential}]
	We will explicitly construct a theoretical sieved estimator based on metric entropy as follows. For any $\delta>0$, let $\mathcal{F}_\delta\equiv \{f_1,\cdots,f_N\}$ denote a $\delta$-net for $\mathcal{F}$ under the metric $l_\nu$ with $N\equiv N(\delta)\equiv \mathcal{N}(\epsilon,\mathcal{F},l_\nu)$. Now we define our estimator to be
	\beqas
	\hat{f}_{n,\delta}:=\argmin_{f \in \mathcal{F}_\delta}M_n(f)
	\eeqas
	where $M_n$ is the least squares criterion
	$M_n(f)\equiv \sum_{i=1}^n (Y_i-f(X_i))^2$. For the true regression function $f_0$, let $f^\ast =\arg\min_{f \in \mathcal{F}_\delta}l_\nu(f,f_0)$. Note that
	\beqas
	M_n(f)-M_n(f^\ast)
	&=\sum_{i=1}^n\bigg\{\big(f^2(X_i)-(f^\ast(X_i))^2\big)-2f_0(X_i)\big(f(X_i)-f^\ast(X_i)\big)\\
	&\qquad -2\epsilon_i\big(f(X_i)-f^\ast(X_i)\big)\bigg\}\equiv \sum_{i=1}^n Z_i
	\eeqas
	where 
	\beqa\label{def:minimax_exponential_bound_Z}
	Z_i&:=\big(f^2(X_i)-(f^\ast(X_i))^2\big)-2f_0(X_i)\big(f(X_i)-f^\ast(X_i)\big)-2\epsilon_i\big(f(X_i)-f^\ast(X_i)\big)\\
	&= \big(f(X_i)-f^\ast(X_i)\big)\big(f(X_i)+f^\ast(X_i)-2f_0(X_i)\big)-2\epsilon_i\big(f(X_i)-f^\ast(X_i)\big)\\
	&=\big(f(X_i)-f^\ast(X_i)\big)^2-2\epsilon_i\big(f(X_i)-f^\ast(X_i)\big)\\
	&\qquad\qquad +2\big(f(X_i)-f^\ast(X_i)\big)\big(f^\ast(X_i)-f_0(X_i)\big).
	\eeqa
	Now for any $\epsilon>0$,
	\beqa\label{ineq:minimax_exponential_bound_1}
	\mathbb{P}\big(l_\nu^2(\hat{f}_{n,\delta},f_0)\geq \epsilon^2|\underline{X}^n\big)&\leq \sum_{f \in \mathcal{F}_\delta,l_\nu^2(f,f_0)\geq \epsilon^2}\mathbb{P}\big(M_n(f)\leq M_n(f^\ast)|\underline{X}^n\big)\\
	&\leq N\mathbb{P}\big(\sum_{i=1}^n Z_i\leq 0|\underline{X}^n\big)\leq N\prod_{i=1}^n\mathbb{E}\big[\exp(-uZ_i)|X_i\big]
	\eeqa
	holds for any $u>0$. Here in the last inequality we have used Markov's inequality and independence of $Z_i$'s conditioned on $X_i$'s. Now we shall control $\mathbb{E}\big[\exp(-uZ_i)|X_i\big]$. By (\ref{def:minimax_exponential_bound_Z}) we see that
	\beqa\label{ineq:minimax_exponential_bound_7}
	\mathbb{E}\big[\exp(-uZ_i)|X_i\big]&=\exp\bigg[-u\big(f(X_i)-f^\ast(X_i)\big)^2\bigg]\mathbb{E}\exp\bigg[2u\epsilon_i\big(f(X_i)-f^\ast(X_i)\big)\bigg\lvert \underline{X}^n\bigg]\\
	&\qquad\qquad \times \exp\bigg[-2u\big(f(X_i)-f^\ast(X_i)\big)\big(f^\ast(X_i)-f_0(X_i)\big)\bigg]\\
	&\leq \exp\bigg[(-u+2\sigma^2u^2)\big(f(X_i)-f^\ast(X_i)\big)^2\\
	&\qquad -2u\big(f(X_i)-f^\ast(X_i)\big)\big(f^\ast(X_i)-f_0(X_i)\big)\bigg].
	\eeqa
	Now by taking $u = 1/4\sigma^2>0$, the above display can be further bounded by
	\beqa\label{ineq:minimax_exponential_bound_8}
	&\exp\bigg[-\frac{1}{8\sigma^2}\bigg(\big(f-f^\ast\big)^2(X_i)+4(f-f^\ast)(f^\ast-f_0)(X_i)\bigg)\bigg]\\
	&\leq  \exp\bigg[-\frac{1}{8\sigma^2}\bigg({1 \over 2}(f-f^\ast)^2(X_i)-8(f^\ast-f_0)^2(X_i)\bigg)\bigg].
	\eeqa
	The last inequality follows from the fact that for all $a,b \in \R$, it holds that $a^2+4ab\geq a^2/2-8b^2$.
	Now it follows from (\ref{ineq:minimax_exponential_bound_1}), (\ref{ineq:minimax_exponential_bound_7}) and (\ref{ineq:minimax_exponential_bound_8}) that
	\beqa\label{ineq:minimax_exponential_bound_9}
	&\mathbb{P}\big(l_\nu^2(\hat{f}_{n,\delta},f_0)\geq \epsilon^2\big)\\
	&\leq N\mathbb{E}\bigg[\prod_{i=1}^n \exp\bigg[-\frac{1}{8\sigma^2}\bigg({1 \over 2}(f-f^\ast)^2(X_i)-8(f^\ast-f_0)^2(X_i)\bigg)\bigg]\bigg\lvert\underline{X}^n\bigg]\\
	&=N\bigg(\mathbb{E}\exp\bigg[-\frac{1}{8\sigma^2}\bigg({1 \over 2}(f-f^\ast)^2(X_1)-8(f^\ast-f_0)^2(X_1)\bigg)\bigg]\bigg)^n.
	\eeqa
	The last equality holds since $X_i$'s are i.i.d. random variables. By convexity of $u \mapsto e^u$ on the interval $[-5\Gamma^2/\sigma^2,0]$, we see that the inequality
	\beqas
	e^{u}\leq 1+\frac{\sigma^2}{5\Gamma^2}\bigg(1-\exp\bigg(-\frac{5\Gamma^2}{\sigma^2}\bigg)\bigg)u
	\eeqas
	holds for all $u \in [-5\Gamma^2/\sigma^2,0]$. In particular, let 
	\beqas
	u= -\frac{1}{8\sigma^2}\bigg({1 \over 2}(f-f^\ast)^2(X_1)-8(f^\ast-f_0)^2(X_1)\bigg).
	\eeqas
	Then it follows from (\ref{ineq:minimax_exponential_bound_9}) that
	\beqa\label{ineq:minimax_exponential_bound_10}
	\mathbb{P}\big(l_\nu^2(\hat{f}_{n,\delta},f_0)\geq \epsilon^2\big)&\leq N\bigg(1-2\mathfrak{z}_0l_\nu^2(f,f^\ast)+32\mathfrak{z}_0l_\nu^2(f^\ast,f_0)\bigg)^2\\
	&\leq N\bigg(1-\mathfrak{z}_0l_\nu^2(f,f_0)+34\mathfrak{z}_0l_\nu^2(f^\ast,f_0)\bigg)^2
	\eeqa
	where in the last inequality we used the triangle inequality $l_\nu^2(f,f^\ast)\geq {1 \over 2}l_\nu^2(f,f_0)-l_\nu^2(f^\ast,f_0)$. Here
	\beqa\label{ineq:minimax_exponential_bound_zeta}
	\mathfrak{z}_0&:=\frac{1}{160\Gamma^2}\bigg(1-\exp\bigg(-\frac{5\Gamma^2}{\sigma^2}\bigg)\bigg).
	\eeqa
	Note that by (\ref{ineq:minimax_exponential_bound_1}), $f\in \mathcal{F}_\delta$ is chosen so that $l_\nu^2(f,f_0)\geq \epsilon^2$, and $l_\nu^2(f^\ast,f_0)\leq \delta^2$ by definition of $f^\ast$. Therefore (\ref{ineq:minimax_exponential_bound_10}) can be further bounded by
	\beqa\label{ineq:minimax_exponential_bound_11}
	\mathbb{P}\big(l_\nu^2(\hat{f}_{n,\delta},f_0)\geq \epsilon^2\big)& \leq  N\big(1-\mathfrak{z}_0\epsilon^2+34\mathfrak{z}_0\delta^2\big)^n\\
	&\leq \exp\bigg(\log N+n\log\big(1-\mathfrak{z}_0\epsilon^2+34\mathfrak{z}_0\delta^2\big)\bigg)\\
	&\leq \exp\bigg(\log N(\delta)-n\mathfrak{z}_0\epsilon^2+34n\mathfrak{z}_0\delta^2\bigg).
	\eeqa
	Here in the last inequality (\ref{ineq:minimax_exponential_bound_11}) we have used the inequality $\log(1+x)\leq x$ for all $x>-1$. Now we let 
	\[
	\delta_n:=\argmin_{\delta>0}\big(\log N(\delta)+34\mathfrak{z}_0 n\delta^2\big)
	\]
	and 
	\beqas
	r_n:=\inf_{\delta>0}\bigg(\frac{1}{\mathfrak{z}_0n}\log N(\delta)+34\delta^2\bigg)=\frac{1}{\mathfrak{z}_0n}\log N(\delta_n)+34\delta_n^2.
	\eeqas
	Then by (\ref{ineq:minimax_exponential_bound_11}) it follows that
	\beqas
	\mathbb{P}\big(l_\nu^2(\hat{f}_{n,\delta_n},f_0)\geq \epsilon^2\big)\leq \exp\big(-\mathfrak{z}_0n(\epsilon^2-r_n)\big).
	\eeqas
	Setting $\epsilon^2\equiv r_n+t/n$ yields the conclusion.
\end{proof}

\begin{proof}[Proof of Corollary \ref{cor:generic_minimax_upper_bound}]
	Note that $R_\nu(n;\mathcal{F})\leq \mathbb{E}\big[l_\nu^2(\hat{f}_n,f_0)\big]$ where $\hat{f}_n$ is the estimator constructed in Theorem \ref{thm:generic_minimax_upper_bound_exponential}. Then by Fubini's theorem it follows that
	\beqas
	\mathbb{E}\big[l_\nu^2(\hat{f}_n,f_0)\big]&=\int_0^\infty \mathbb{P}\big[l_\nu^2(\hat{f}_n,f_0)>u\big]\ \d{u}\\
	& \leq \int_0^{r_n}\ \d{u}+\int_{r_n}^\infty \mathbb{P}\big[l_\nu^2(\hat{f}_n,f_0)>u\big]\ \d{u}\\
	&= r_n +\int_0^\infty \mathbb{P}\big[n(l_\nu^2(\hat{f}_n,f_0)-r_n)>nv\big]\ \d{v}\leq r_n+\frac{1}{\mathfrak{z}_0 n}.
	\eeqas
	The proof is complete.
\end{proof}

\begin{proof}[Proof of Theorem \ref{thm:upper_bound_poly}]
	We first consider $\nu$ to be the canonical Lebesgue measure. By Lemma \ref{lem:entropy_cvx_func}, for a polytopal region $\Omega \in \mathscr{P}_k$,  solving $r_n$ as defined in Theorem \ref{thm:generic_minimax_upper_bound_exponential} we find that
	\beqa
	r_n=C_d\big(k(\sqrt{\abs{\Omega}}\Gamma)^{d/2}/\mathfrak{z}_0\big)^{4/(d+4)}n^{-4/(d+4)}.
	\eeqa
	Similarly for smooth region $\Omega$, we have
	\beqa
	r_n=C_d
	\begin{cases}
		\big(\sqrt{\abs{\Omega}}\Gamma/\mathfrak{z}_0\big)^{2/3}n^{-2/3} (\log n)& d=2;\\
		\big((\sqrt{\abs{\Omega}}\Gamma)^{(d-1)}/\mathfrak{z}_0\big)^{2/(d+1)}n^{-2/(d+1)} & d\geq 3.
	\end{cases}
	\eeqa
	as desired. Here for $d=2$ we require $n$ large enough. Now for general $\nu$, it can replaced by the Lebesgue measure with a price of an extra term $\nu_{\max}$ in the final bound. Hence by Corollary \ref{cor:generic_minimax_upper_bound}, for a polytopal domain  $\Omega \in \mathscr{P}_k$, we have
	\beqa\label{const:minimax_upper_bound_1}
	R_\nu(n;\mathcal{C}(\Gamma))\leq \bigg(C_d\big(k(\sqrt{\abs{\Omega}}\Gamma)^{d/2}/\mathfrak{z}_0\big)^{4/(d+4)}\vee \mathfrak{z}_0^{-1}\bigg)\nu_{\max}n^{-4/(d+4)}.
	\eeqa
	For smooth domain $\Omega$, we have
	\beqa\label{const:minimax_upper_bound_2}
	R_\nu(n;\mathcal{C}(\Gamma))\leq C_d 
	\begin{cases}
		\big(\big(\sqrt{\abs{\Omega}}\Gamma/\mathfrak{z}_0\big)^{2/3}\vee \mathfrak{z}_0^{-1}\big)\nu_{\max}n^{-2/3} (\log n) & d=2;\\
		\big(\big((\sqrt{\abs{\Omega}}\Gamma)^{(d-1)}/\mathfrak{z}_0\big)^{2/(d+1)}\vee \mathfrak{z}_0^{-1}\big)\nu_{\max}n^{-2/(d+1)} & d\geq 3.
	\end{cases}
	\eeqa
\end{proof}

The proofs for Theorems \ref{thm:minimax_lower_bound_poly} and \ref{thm:minimax_lower_bound_cvxbody} will make use of Assouad's lemma (cf. Lemma 24.3, \cite{van2000asymptotic}) so we briefly describe the machinary below. For two probability measures $\mathbb{P},\mathbb{Q}$, let $K(\mathbb{P},\mathbb{Q})$  and $\pnorm{\mathbb{P}-\mathbb{Q}}{TV}$ denote the Kullback-Leibler divergence and the total variation distance between $\mathbb{P}$ and $\mathbb{Q}$, respectively. Assuoad's lemma asserts that for each $m \in \N$, and any class of test functions $\{f_\tau\in \mathcal{F}\}_{\tau \in \{0,1\}^m}$, the following lower bound holds:
\beqas
R_\nu(n;\mathcal{F})\geq \frac{m}{8}\min_{\tau\neq \tau'}\frac{l_\nu^2(f_\tau,f_{\tau'})}{H(\tau,\tau')}\min_{H(\tau,\tau')=1}\big(1-\pnorm{\mathbb{P}_{f_\tau}-\mathbb{P}_{f_{\tau'}}}{TV}\big).
\eeqas
Here $H(\tau,\tau')$ denotes the Hamming distance between $\tau,\tau'\in \{0,1\}^m$. Note that
\beqas
\pnorm{\mathbb{P}_{f_\tau}-\mathbb{P}_{f_{\tau'}}}{TV}^2\leq {1 \over 2}K(\mathbb{P}_{f_\tau},\mathbb{P}_{f_{\tau'}})=\frac{n}{4\sigma^2}l_\nu^2(f_\tau,f_{\tau'}),
\eeqas
where the first inequality follows from Pinsker's inequality and the second follows by straightforward conditioning arguments. Morever, the test functions are usually constructed with `separate support' so that \newline
$\min_{\tau\neq \tau'}{l_\nu^2(f_\tau,f_{\tau'})}/{H(\tau,\tau')}=\min_{H(\tau,\tau')=1}l_\nu^2(f_\tau,f_{\tau'})$. Hence in this scenario
\beqa\label{ineq:generic_minimax_lower_bound}
R_\nu(n;\mathcal{F})\geq \frac{m}{8}\min_{H(\tau,\tau')=1}l_\nu^2(f_\tau,f_{\tau'})\bigg(1-\sqrt{\frac{n}{4\sigma^2}\max_{H(\tau,\tau')=1}l_\nu^2(f_\tau,f_{\tau'}})\bigg).
\eeqa
Thus to derive sharp lower bounds it is essential to obtain two-sided estimates of $l_\nu^2(f_\tau,f_{\tau'})$ with matching order in terms of the size of the cube $m$.

\begin{proof}[Proof of Theorem \ref{thm:minimax_lower_bound_poly}]
	We first assume that the domain $\Omega$ is $[0,1]^d$. The class of functions we construct is similar to the class constructed in Section 2.9 of  \cite{gao2015entropy}. Choose a fixed function $g_0$ on $[0,1]^d$ so that the following properties hold:
	\begin{enumerate}
		\item $l_\infty$ boundedness: $0\leq g_0\leq 1/20$;
		\item $l_1$ boundedness: $\pnorm{g_0}{l_1}\geq 1/80d$;
		\item For every $x \in [0,1]^d$, the Hessian matrix $\nabla^2 g_0(x)$ is diagonal with each entry bounded by $1$.
	\end{enumerate}
	Such a function exists; for example we can take
	\beqas
	g_0(x)=\frac{1}{20d}\sum_{i=1}^d \sin^3(\pi x_i)\bm{1}_{[0,1]^d}(x).
	\eeqas
	Now for $I=(i_1,\cdots,i_d)\in \mathbb{Z}^d$, let $B_I:=\prod_{j=1}^d [i_j\epsilon,(i_j+1)\epsilon]$, and $\mathcal{I}:=\{I:B_I\subset[0,1]^d\}$. For fixed $\epsilon>0$, define a local function $g_0^{(I)}$ supported on $B_I$ as follows:
	\beqas
	g_0^{(I)}(x):=\epsilon^2g_0\bigg(\frac{x_1-i_1\epsilon}{\epsilon},\cdots,\frac{x_d-i_d\epsilon}{\epsilon}\bigg).
	\eeqas
	Then it is easy to see that $0\leq g_0^{(I)}\leq \epsilon^2/20$ and $\pnorm{g_0^{(I)}}{l_1}\geq \epsilon^{d+2}/80d$. It follows by the  Cauchy-Schwarz inequality that 
	\[
	\frac{\epsilon^{d+4}}{6400d^2}\leq\pnorm{g_0^{(I)}}{l_2}^2\leq\frac{\epsilon^{d+4}}{400}.
	\]
	Now for a general polytope $\Omega \in \mathscr{P}_k$, write $\Omega=\cup_{i=1}^k \Omega_i$ where all the $\Omega_i$'s are simplices. Suppose $R_i\subset \Omega_i$'s are inscribed hypercubes, and $\psi_i$'s are the linear maps that take $R_i$'s to $[0,1]^d$. It is easy to see that $\det \psi_i=\abs{R_i}^{-1}$. Then for $I=(i_1,\cdots,i_d)\in \mathbb{Z}^d$, let $B_{I,i}:=\psi_i^{-1}(B_I)$, $\mathcal{I}_i:=\{I: B_{I,i}\subset R_i\}$ and 
	\beqas
	g_{0,i}^{(I)}(x):=\abs{R_i}^{2/d}g_0(\psi_i(x)).
	\eeqas
	Then the Hessian of $g_{0,i}^{(I)}$ is still a diagonal matrix with each entry bounded by $1$, and 
	\[
	\frac{\abs{R_i}^{(d+4)/d}\epsilon^{d+4}}{6400d^2}\leq \pnorm{g_{0,i}^{(I)}}{l_2}^2 \leq \frac{\abs{R_i}^{(d+4)/d}\epsilon^{d+4}}{400}.
	\]
	Consider the collection of indices $\mathcal{I}:=\cup_{i=1}^k \mathcal{I}_i$ and corresponding functions $\cup_{i=1}^k\{g_{0,i}^{(I_i)}\}$. 
	Since $\abs{\mathcal{I}_i}=\epsilon^{-d}$, we set $m=\abs{\mathcal{I}}=k\epsilon^{-d}$. Note that $\mathcal{I}$ can be identified with the coordinates of $\{0,1\}^m$ so we shall use this convention in the sequel. Let $w(\Omega):=\sup_{x,y\in \Omega}\pnorm{x-y}{2}$ be the width of $\Omega$. 
	By translation we may assume that $\Omega \subset B_2(0,w(\Omega))$ without loss of generality. For any $\tau \in \{0,1\}^m$, set $g_\tau:=\sum_{i=1}^m g_{0,\pi(i)}^{(\tau_i)}\bm{1}_{\tau_i=1}$ where $\tau_i\in \mathcal{I}_{\pi(i)}$ and
	\beqas
	f_\tau(x)=\frac{\Gamma}{(w(\Omega))^2}\big(\pnorm{x}{2}^2-g_\tau(x)\big).
	\eeqas
	Then clearly $f_\tau \in \mathcal{C}(\Gamma)$ by construction. For two indices $\tau,\tau' \in \{0,1\}^m$ with Hamming distance $1$, we see that
	\beqas
	\frac{\min_i\abs{R_i}^{(d+4)/d}\Gamma^2}{6400d^2(w(\Omega))^4}\nu_{\min}\epsilon^{d+4}\leq
	l_\nu^2(f_\tau,f_{\tau'})\leq \frac{\max_i\abs{R_i}^{(d+4)/d}\Gamma^2}{400(w(\Omega))^4}\nu_{\max}\epsilon^{d+4}.
	\eeqas
	Now apply Assouad's lemma (\ref{ineq:generic_minimax_lower_bound}) to see that
	\beqas
	R_\nu(n;\mathcal{C}(\Gamma))&\geq \frac{k\Gamma^2\min_i\abs{R_i}^{(d+4)/d}}{51200d^2(w(\Omega))^4}\nu_{\min}\epsilon^{4}\\
	&\qquad\qquad\times \bigg(1-\sqrt{\frac{n}{1600\sigma^2}\frac{\Gamma^2\max_i\abs{R_i}^{(d+4)/d}}{(w(\Omega))^4}\nu_{\max}k\epsilon^{d+4}}\bigg).
	\eeqas
	Choosing $\epsilon = (400\sigma^2(w(\Omega))^4/\Gamma^2\max_i\abs{R_i}^{(d+4)/d}\nu_{\max})^{1/(d+4)}n^{-1/(d+4)}$ we conclude that
	\beqa\label{const:minimax_lower_bound_poly}
	R_\nu(n;\mathcal{C}(\Gamma))\geq \frac{400^{4/(d+4)}}{102400d^2} k\pi(\Omega) \Gamma^{2d/(d+4)}\nu_{\min}\bigg(\frac{\sigma^2}{\nu_{\max}n}\bigg)^{4/(d+4)}.
	\eeqa
	Here 
	\beqas
	\pi(\Omega)\equiv (w(\Omega))^{-4d/(d+4)}\sup_{\{R_i\subset \Omega_i\}_{i=1}^k} \bigg(\frac{\min_i\abs{R_i}^{(d+4)/d}}{\max_i\abs{R_i}^{4/d}}\bigg)
	\eeqas
	where the supremum is taken over all inscribed hypercubes.
\end{proof}

\begin{proof}[Proof of Theorem \ref{thm:minimax_lower_bound_cvxbody}]
	Fix some positive $\epsilon>0$ small enough. By Smoothness Assumption \ref{smooth1}, we can find pairwise disjoint caps $\{C_i\}_{i=1}^m$ so that $\abs{C_i}\lesssim_d \epsilon\abs{\Omega}$ and $m\asymp_d (\epsilon\abs{\Omega})^{-(d-1)/(d+1)}$. Now write the caps $C_i=\{x \in \R^d : x\cdot x_i\geq a_i\}$ for some $x_i \in \R^d$ and $a_i \in \R$. For $1\leq i\leq m$, define
	\beqa
	h_i(x):=\frac{\Gamma \big(x\cdot x_i-a_i\big)\bm{1}_{C_i}(x)}{\sup_{y \in C_i}(y\cdot x_i-a_i)}.
	\eeqa
	Note that $h_i(\cdot)$ is a non-negative affine function supported only on the cap $C_i$ and is bounded by $\Gamma$. Now for any $\bm{\tau},\bm{\tau'}\in\{0,1\}^m$, we have
	\beqa\label{ineq:minimax_lower_bound_smooth}
	l_\nu^2(f_{\bm{\tau}},f_{\bm{\tau}'})&\lesssim_d \Gamma^2H(\bm{\tau},\bm{\tau}')\nu_{\max}\epsilon;\\
	l_\nu^2(f_{\bm{\tau}},f_{\bm{\tau}'})&\gtrsim_d \Gamma^2H(\bm{\tau},\bm{\tau}')\nu_{\min}\epsilon.
	\eeqa
	By an application of Assouad's lemma we conclude that
	\beqas
	R_\nu(n)\geq c_1\abs{\Omega}^{-(d-1)/(d+1)}\epsilon^{-(d-1)/(d+1)} \Gamma^2\nu_{\min}\epsilon\bigg(1-\sqrt{c_2\frac{n}{4\sigma^2}\Gamma^2\nu_{\max}\epsilon}\bigg)
	\eeqas
	where $c_1,c_2$ are constants only depending on $d$.
	By choosing 
	\beqas
	\epsilon:=c_2^{-1}(\sigma^2/\Gamma^2)\nu_{\max}^{-1}n^{-1},
	\eeqas
	we conclude that
	\beqa\label{const:minimax_lower_bound_smooth}
	R_\nu(n)\geq \mathfrak{C}_d \abs{\Omega}^{-(d-1)/(d+1)} \Gamma^{(2d-2)/(d+1)}\nu_{\min}\bigg(\frac{\sigma^2}{\nu_{\max}n}\bigg)^{2/(d+1)}.
	\eeqa
	The proof is complete.
\end{proof}

\begin{proof}[Proof of Theorem \ref{thm:minimax_lower_bounds_simple}]
	By Theorem 2.7 page 101 in \cite{tsybakov2008introduction}, we need to construct a family of $\{f_0,f_1,\cdots,f_M\}\subset \mathcal{P}_k(\Gamma)$ such that the following conditions hold:
	\begin{enumerate}
		\item[(1)] $l_\nu^2(f_i,f_j)\geq 2s$ holds for all $i\neq j$;
		\item[(2)] Let $\mathbb{P}_j$ be the probability measure for $(X,Y)$ when the regression function is $f_j$, then
		\beqas
		\frac{1}{M+1}\sum_{j=1}^M K(\mathbb{P}_j,\mathbb{P}_0)\leq \alpha \log M
		\eeqas
		holds for some $0<\alpha<1$.
	\end{enumerate}
	Note that $K(\mathbb{P}_j,\mathbb{P}_0)=\frac{n}{2\sigma^2}l_\nu^2(f_j,f_0)$. Then (2) can be replaced with (2'):
	\begin{enumerate}
		\item[(2')] The following holds for some $0<\alpha<1$:
		\beqas
		\frac{n}{2\sigma^2(M+1)}\sum_{j=1}^M l_\nu^2(f_j,f_0)\leq \alpha \log M. 
		\eeqas
	\end{enumerate}
	Now by the same construction as in the proof of Theorem \ref{thm:minimax_lower_bound_cvxbody} we get $\{h_i\}_{i=1}^m$ where $m\asymp_d\epsilon^{-(d-1)/(d+1)}$. Let $M=\lfloor m/k\rfloor$, and write
	\beqas
	c_1^{-1} k^{-1}\epsilon^{-(d-1)/(d+1)}\leq M\leq c_1 k^{-1}\epsilon^{-(d-1)/(d+1)}.
	\eeqas
	Now for $i=1,\cdots, M$, define
	\beqa
	f_i(x)=\sum_{j=(i-1)k+1}^{ik} h_j(x).
	\eeqa
	For $i\neq j$, 
	\beqas
	c_2^{-1}\Gamma^2 \nu_{\min}k\epsilon\leq  l_\nu^2(f_i,f_j)&\leq c_2 \Gamma^2 \nu_{\max} k\epsilon.
	\eeqas
	Now (2') will be satisfied if 
	\beqas
	\frac{M}{M+1}\frac{n}{2\sigma^2}c_2\Gamma^2\nu_{\max} k\epsilon \leq  \frac{\alpha(d-1)}{(d+1)}\log \bigg(\frac{c_3}{k^{(d-1)/(d+1)}\epsilon}\bigg)
	\eeqas
	where $c_3=c_1^{-(d-1)/(d+1)}$.  Since $M\geq 1$ we only have to ensure that
	\beqas
	\frac{n}{\sigma^2/\Gamma^2} \nu_{\max}k\epsilon\leq  c_4\alpha \log \bigg(\frac{c_3}{k^{(d-1)/(d+1)}\epsilon}\bigg).
	\eeqas
	Choose $\epsilon=c_5\frac{\sigma^2}{\Gamma^2\nu_{\max}} \frac{\log n }{n}$ for $c_5$ small enough. Then the above display holds for $n$ large enough depending on $k,d,\alpha$. Then with
	\beqas
	s \asymp_d   \sigma^2 \nu_{\max}^{-1}\nu_{\min}\frac{k\log n}{n},
	\eeqas
	it follows by Theorem 2.7 in \cite{tsybakov2008introduction} that
	\beqas
	\inf_{\hat{f}_n}\sup_{f \in \mathcal{P}_k(\Gamma)}\mathbb{P}_f\big(l_\nu^2(\hat{f}_n,f)\geq s\big)\geq \bigg(\frac{\log(M+1)-\log 2}{\log M}-\alpha\bigg)\geq 1-\alpha-\log 2/ \log M.
	\eeqas
	Choose $\alpha = 1/2-\log 2/\log 10$, then for $n$ large enough depending through $d,\sigma, \Gamma,\nu$, the value of $M$ exceeds $10$, and hence the right hand side of the above display $\geq 1/2$. This completes the proof.
\end{proof}

\section{Proofs for Section 3}

We shall first prove  Theorem \ref{thm:risk_bounds_random_design}. To this end, we will need some tools from empirical process theory. To fix notation, let $X_1,X_2,\ldots$ and $\epsilon_1,\epsilon_2,\ldots$ be coordinate projections of $(\Omega^\infty,\mathcal{B}^\infty,P_\nu^\infty)$ and $(\R^\infty,\mathcal{A}^\infty,P_\epsilon^\infty)$.  Now for a function $f:\Omega\subset \R^d \to \R$ and $g:\R \to \R$, let $\mathbb{P}_n (f\otimes g):={1 \over n}\sum_{i=1}^n f(X_i)g(\epsilon_i)$, $P(f\otimes g):=(P_\nu f)\cdot (P_\epsilon g)$ and 
\beqas
\mathbb{G}_n (f \otimes g):=\sqrt{n}\big(\mathbb{P}_n-P)(f\otimes g)={1 \over \sqrt{n}}\sum_{i=1}^n \big(f(X_i)g(\epsilon_i)-(P_\nu f)(P_\epsilon g)\big).
\eeqas
Let $e:\R\to \R$ denote the identity map, and $1:\R \to \R$ denote the map so that $1(x)\equiv 1$ for all $x \in \R$. We will use the abbreviations $\mathbb{P}_n f, \mathbb{G}_n f$ when $g\equiv 1$, i.e. $\mathbb{P}_n f\equiv \mathbb{P}_n (f\otimes 1) ={1 \over n}\sum_{i=1}^n f(X_i)$ and $\mathbb{G}_n f\equiv \mathbb{G}_n (f\otimes 1)={1 \over \sqrt{n}}\sum_{i=1}^n \big(f(X_i)-P_\nu f)$. For classes $\mathcal{F}, \mathcal{G}$ of measurable functions $f:\Omega\to \R$ and $g:\R \to \R$, let $\pnorm{\mathbb{G}_n}{\mathcal{F}\otimes \mathcal{G}}:=\sup_{f \in \mathcal{F}, g \in \mathcal{G}}\abs{\mathbb{G}_n (f\otimes g)}$. $\pnorm{\mathbb{G}_n}{\mathcal{F}}$ is used to denote $\pnorm{\mathbb{G}_n}{\mathcal{F}\otimes \{1\}}$. For any $f: \Omega \to \R$ and $g:\R \to \R$, define the Bernstein `norm' as follows:
\beqas
\pnorm{f\otimes g}{P,B}:=\big(2P(\exp(\abs{f\cdot g})-1-\abs{f\cdot g})\big)^{1/2}.
\eeqas
We first state a uniform inequality.
\begin{lemma}[Theorems 5.11 and 8.13, \cite{van2000empirical}]\label{lem:inequality_sup_empirical_process}
	Let $\mathcal{H}$ be a collection of functions defined on $(\mathcal{T},\mathcal{B}_{\mathcal{T}},P)$. Suppose $\pnorm{h}{P,B}\leq R$ holds for all $h \in \mathcal{H}$. For $t>0$ satisfying
	\beqa\label{ineq:uniform_inequality_1}
	t\leq 8\sqrt{n}(R^2\wedge R),
	\eeqa
	and
	\beqa\label{ineq:uniform_inequality_2}
	t\geq 3C\bigg(\int_{t/2^6\sqrt{n}}^R \sqrt{\log\mathcal{N}_{[\,]}(\epsilon,\mathcal{H},\pnorm{\cdot}{P,B})}\ \d{\epsilon} \vee R\bigg),
	\eeqa
	the following deviation bound holds:
	\beqas
	\mathbb{P}\left(\pnorm{\mathbb{G}_n}{\mathcal{H}}>t\right)\leq C\exp\left(-\frac{t^2}{9C^2R^2}\right).
	\eeqas
	Here $C$ is a universal constant.
\end{lemma}
In our specific application, $\mathcal{H}$ has the form $\mathcal{H}=\mathcal{F}\otimes \mathcal{G}$.
\begin{proof}[Proof of Theorem \ref{thm:risk_bounds_random_design}]
	The proof is based on a peeling device. Note that the least squares estimator maximizes 
	\beqas
	\mathbb{M}_n (f) \equiv 2 \mathbb{P}_n\big[(f-f_0)\otimes e\big]-\mathbb{P}_n(f-f_0)^2.
	\eeqas
	The population version is $M(f):=-P_\nu(f-f_0)^2$ since the expectation of the first term vanishes. Now for given $r>0$, set 
	\beqas
	S_j(r):=\{f\in\mathcal{F}:2^{j-1}r<l_\nu(f,f_0)\leq 2^jr\}.
	\eeqas
	By the peeling argument we have
	\beqas
	&\mathbb{P}\left(l_\nu(\hat{f}_n^{\mathrm{LS}},f_0)>r\right)\\
	&\leq \sum_{j\geq 1} \mathbb{P}\bigg(\sup_{f \in S_j(r)}\big(\mathbb{M}_n(f)-\mathbb{M}_n(f_0)\big)\geq 0\bigg)\\
	&\leq \sum_{j\geq 1} \mathbb{P}\bigg(\sup_{f \in S_j(r)}\abs{\mathbb{M}_n(f)-\mathbb{M}_n(f_0)-(M(f)-M(f_0))}\geq 2^{2j-2}r^2\bigg).
	\eeqas
	Note that
	\beqas
	&\abs{\mathbb{M}_n(f)-\mathbb{M}_n(f_0)-(M(f)-M(f_0))}\\
	&\leq 2\abs{(\mathbb{P}_n-P)\big[(f-f_0)\otimes e\big]}+\abs{(\mathbb{P}_n-P_\nu)(f-f_0)^2}.
	\eeqas
	Hence the series can be further bounded by
	\beqas
	&\sum_{j\geq 1 }\bigg(\mathbb{P}\left(\pnorm{\mathbb{G}_n}{\mathcal{F}(2^jr)\otimes \mathcal{E}}\geq \sqrt{n}2^{2j-2}r^2/3\right)+\mathbb{P}\left(\pnorm{\mathbb{G}_n}{\mathcal{F}(2^jr)^2}\geq \sqrt{n}2^{2j-2}r^2/3\right)\bigg)\\
	&\qquad\qquad\quad \equiv \sum_{j\geq 1}(P_{1,j}+P_{2,j}).
	\eeqas
	Here $\mathcal{F}(r)\equiv S(f_0,r)\equiv \{f \in \mathcal{F}:l_\nu(f,f_0)\leq r\}$ and $\mathcal{E}:=\{e\}$. We first deal with $P_{1,j}$. We claim that 
	\beqa\label{ineq:entropy_l2_bernnorm}
	\log \mathcal{N}_{[\,]}(\epsilon,\mathcal{F}(2^jr)\otimes \mathcal{E},\pnorm{\cdot}{P,B})\leq \log \mathcal{N}_{[\,]}(\epsilon/\Phi_\Gamma,\mathcal{F}(2^jr),l_\nu).
	\eeqa
	To see this,  note that for any $f_1,f_2 \in \mathcal{F}(2^jr)$,
	\beqa\label{ineq:connect_bern_l2}
	P\left(\exp(\abs{f_1-f_2}\abs{e})-1-\abs{f_1-f_2}\abs{e}\right)&=\sum_{m=2}^\infty \frac{P_\nu \abs{f_1-f_2}^m P_\epsilon \abs{e}^m}{m!}\\
	&\leq P_\nu (f_1-f_2)^2\cdot \mathbb{E}\bigg[\sum_{m=2}^\infty \frac{(2\Gamma)^{m-2}\abs{\epsilon_1}^m}{m!}\bigg]\\
	&\leq P_\nu (f_1-f_2)^2 \Phi_\Gamma^2/2,
	\eeqa
	implying that $\pnorm{(f_1-f_2)\otimes e}{P,B}\leq \pnorm{f_1-f_2}{l_\nu}\Phi_\Gamma$. Note that we used the fact $\Gamma\geq 1/2$ in the last line. On the other hand, there is a one-to-one correspondence of brackets between $\mathcal{F}(2^jr)$ and $\mathcal{F}(2^jr)\otimes \mathcal{E}$: $[f_1,f_2]\leftrightarrow[f_1e_+-f_2e_-,f_2e_+-f_1e_-]$, where $e_+ := e \vee 0$ and $e_-:= -(e\wedge 0)$, and $0(x)=0$ for all $x \in \R$. This shows (\ref{ineq:entropy_l2_bernnorm}). Now we will apply Lemma \ref{lem:inequality_sup_empirical_process} to $\mathcal{F}(2^j r)\otimes \mathcal{E}$ for all $j\geq 1$. We choose $r_n$ such that 
	\beqas
	\frac{J_{[\,]}(r_n,f_0,l_\nu)}{\sqrt{n}r_n^2}\leq \frac{1}{13C(\Phi_\Gamma \vee 4\sqrt{2}\Gamma \exp(4\Gamma^2))},
	\eeqas
	where 
	\beqas
	J_{[\,]}(r,f_0,l_\nu)=\int_{r^2/3\cdot 2^6(\Phi_\Gamma\vee (4\sqrt{2}\Gamma \exp(4\Gamma^2)))}^{2r}\sqrt{\log \mathcal{N}_{[\,]}(\epsilon,S(f_0,r),l_\nu)}\ \d{\epsilon}.
	\eeqas
	With $R=2^{j}r \Phi_\Gamma$ and $t = 2^{2j-2}\sqrt{n}r^2/3$, the conditions (\ref{ineq:uniform_inequality_1}) and (\ref{ineq:uniform_inequality_2}) are satisfied for all $j\geq 1$ and $r\geq r_n$, if furthermore the following holds:
	\beqa\label{cond:risk_bound_random_design_1}
	\Phi_\Gamma \geq \sqrt{1/96}\vee (\Gamma/48);\quad \Phi_\Gamma \leq \sqrt{n}r_n/18C,
	\eeqa
	and $J_{[\,]}(r,f_0,l_\nu)/r^2$ is non-increasing. Here we used the fact that $j$ cannot be too large given the apriori uniform bound $\Gamma$ of the function class $\mathcal{F}$: $2^jr \leq 2\Gamma$. Now invoking Lemma \ref{lem:inequality_sup_empirical_process}, we see that
	\beqas
	P_{1,j}\leq C\exp\left(-\frac{2^{2j}nr^2}{1296C^2\Phi_\Gamma^2}\right).
	\eeqas
	Now we deal with $P_{2,j}$. First note that $\mathcal{F}^2=\mathcal{F}_+^2+\mathcal{F}_-^2$ where $\mathcal{F}_{\pm}:=\{f_{\pm}:f \in \mathcal{F}\}$. Here $f_+:=f \vee 0$ and $f_-:=-(f\wedge 0)$. Suppose $\{[u_{i},l_i]\}_{i=1}^N$ is an $\epsilon$-bracket of $\mathcal{F}$ under $l_\nu$. Then $\{[u_i \vee 0, l_i\vee 0]\}_{i=1}^N$ and $\{[-(l_i\wedge 0),-(u_i\wedge 0)]\}_{i=1}^N$ are $\epsilon$-brackets for $\mathcal{F}_+$ and $\mathcal{F}_-$, respectively. These bracketing functions are all non-negative. Hence their squares yield $2\Gamma \epsilon$-brackets for $\mathcal{F}_+^2$ and $\mathcal{F}_-^2$. Hence
	\beqa
	\log \mathcal{N}_{[\,]}(4\Gamma\epsilon, \mathcal{F}^2, l_\nu)&\leq 2\log \mathcal{N}_{[\,]}(\epsilon, \mathcal{F}, l_\nu).
	\eeqa
	By similar arguments as in (\ref{ineq:connect_bern_l2}), for $f_1,f_2 \in \mathcal{F}(2^j r)^2$ we have
	\beqas
	P_\nu \big(\exp(\abs{f_1-f_2})-1-\abs{f_1-f_2}\big)&\leq P_\nu (f_1-f_2)^2 \exp(8\Gamma^2)
	\eeqas
	by noting that $\abs{f_1-f_2}\leq 8\Gamma^2$. This implies that $\pnorm{f_1-f_2}{P,B}\leq l_\nu(f_1,f_2)\sqrt{2}\exp(4\Gamma^2)$. This means that
	\beqa
	\log \mathcal{N}_{[\,]}(\epsilon, \mathcal{F}(2^jr)^2,\pnorm{\cdot}{P,B})&\leq 
	\log \mathcal{N}_{[\,]}(\epsilon/(\sqrt{2}\exp(4\Gamma^2)),\mathcal{F}(2^jr)^2,l_\nu)\\
	&\leq 2 \log \mathcal{N}_{[\,]}(\epsilon/(4\sqrt{2}\Gamma \exp(4\Gamma^2)),\mathcal{F}(2^jr),l_\nu).
	\eeqa
	Now imposing conditions
	\beqa\label{cond:risk_bound_random_design_2}
	4\sqrt{2}\Gamma \exp(4\Gamma^2) \geq \sqrt{1/96}\vee (\Gamma/48);\quad 4\sqrt{2}\Gamma \exp(4\Gamma^2)  \leq \sqrt{n}r_n/18C,
	\eeqa
	we get the same estimate 
	\beqas
	P_{2,j}\leq C\exp\left(-\frac{2^{2j}nr^2}{1296C^2(4\sqrt{2}\Gamma \exp(4\Gamma^2) )^2}\right).
	\eeqas
	Combining conditions (\ref{cond:risk_bound_random_design_1}), (\ref{cond:risk_bound_random_design_2}) we have
	\beqa
	\Phi_\Gamma \wedge (4\sqrt{2}\Gamma \exp(4\Gamma^2)) \geq \sqrt{1/96} \vee (\Gamma/48);\quad \Phi_\Gamma \vee (4\sqrt{2}\Gamma \exp(4\Gamma^2))  \leq \sqrt{n}r_n/18C.
	\eeqa
	Thus
	\beqas
	\mathbb{P}\left(l_\nu(\hat{f}_n^{\mathrm{LS}},f_0)>r\right)\leq \sum_{j\geq 1}(P_{1,j}+P_{2,j})\leq 2C\sum_{j\geq 1}\exp\left(-\frac{2^{2j}nr^2}{1296C^2(\Phi_\Gamma^2 \vee 32\Gamma^2\exp(8\Gamma^2))}\right).
	\eeqas
	Consequently, 
	\beqas
	\mathbb{E}\big[l_\nu^2(\hat{f}_n^{\mathrm{LS}},f_0)\big]&\leq  r_n^2 +2C\sum_{j\geq 1}\int_0^\infty \exp\left(-\frac{2^{2j}nt}{1296C^2(\Phi_\Gamma^2 \vee 32\Gamma^2\exp(8\Gamma^2))}\right)\ \d{t}\\
	&=r_n^2 +\frac{864C^3(\Phi_\Gamma^2 \vee 32\Gamma^2\exp(8\Gamma^2))}{n}.
	\eeqas
	The proof is complete.
\end{proof}

Next we prove Lemma \ref{lem:local_entropy_estimate}. We will need the following result,  extending Theorem 2.6 in \cite{guntuboyina2014covering} to general polytopal domains.
\begin{lemma}\label{lem:entropy_boundary_control}
	Suppose $\Omega \in \mathscr{P}_k$. Then for $\epsilon\leq C_{d,0}\sqrt{\Omega}\Gamma$, 
	\beqas
	\log \mathcal{N}_{[\,]}(\epsilon,\mathcal{C}_2(r;\Omega)\cap \mathcal{C}(\Gamma;\Omega),l_2(\Omega))\leq C_d k^{(d+4)/4}\bigg(\frac{r}{\epsilon}\bigg)^{d/2}\bigg(\log \frac{C_dk\Gamma^2 \abs{\Omega}}{\epsilon^2}\bigg)^{d(d+4)/4}.
	\eeqas
\end{lemma}

Before the proof of Lemma \ref{lem:entropy_boundary_control}, we will need some additional notation. A polytope $\Omega$ is of simplex type if $\Omega=\cap_{i=1}^{d+1}E_i$ where $E_i$ are half spaces. A polytope $\Omega$ is called a parallelotope if 
\beqas
\Omega=\cap_{i=1}^d \{x:p_{i,1}\leq \iprod{x}{v_i}\leq p_{i,2}\}\equiv  \cap_{i=1}^d E(p_{i,1},p_{i,2};v_i)
\eeqas
for some linearly independent $\{v_i\}$ and $\{(p_{i,1},p_{i,2})\}$. For fixed $\eta\in [0,1]$, and parallelotope $\Omega$, define
$\Omega_\eta\equiv \Omega(\eta)\equiv\cap_{i=1}^{d}E(\eta p_{i,1}+(1-\eta)p_{i,2},(1-\eta)p_{i,1}+\eta p_{i,2};v_i)$.

We will also need the following result for Lemma \ref{lem:entropy_boundary_control}.
\begin{lemma}\label{lem:boundary_entropy_parallelotope}
	Let $\Omega$ be a parallelotope. Then for $\eta<1/2$,
	\beqas
	\log \mathcal{N}_{[\,]}(\epsilon,\mathcal{C}_2(1;\Omega),l_2(\Omega(\eta)))\leq C_d\epsilon^{-d/2}\bigg(\log \frac{1}{\eta}\bigg)^{d(d+4)/4}.
	\eeqas
\end{lemma}
\begin{proof}[Proof of Lemma \ref{lem:boundary_entropy_parallelotope}]
	By affine transformation we may assume without loss of generality that $\Omega$ is the hypercube $[0,1]^d$. Now the conclusion follows from Theorem 2.6 \cite{guntuboyina2014covering} by replacing metric entropy by bracketing entropy. This is valid since the only place in the proof of Theorem 2.6 there where entropy is involved lies in the partitioning scheme Theorem 3.1, which in turn only depends on Theorem 2.1 \cite{guntuboyina2014covering}. The corresponding analogous result of Theorem 2.1 there in terms of bracketing entropy is proved in Theorem 1.1 \cite{gao2015entropy}. We omit book-keeping details for simplicity.
\end{proof}
\begin{proof}[Proof of Lemma \ref{lem:entropy_boundary_control}]
	First we assume that $\Omega$ is of simplex type. Note that $\Omega$ can be covered by finitely many parallelotopes, i.e. there exists a sequence of parallelotopes $\{A_i\}_{i=1}^{\pi_d}$ such that $A_i\subset \Omega$ and $\Omega=\cup_{i=1}^{\pi_d} A_i$. Now we can apply Lemma \ref{lem:boundary_entropy_parallelotope} to conclude that 
	\beqas
	\log \mathcal{N}_{[\,]}(\epsilon,\mathcal{C}_2(r;A_i),l_2(A_i(\eta)))\leq C_d\bigg(\frac{r}{\epsilon}\bigg)^{d/2}\bigg(\log \frac{1}{\eta}\bigg)^{d(d+4)/4}.
	\eeqas
	On the other hand,
	\beqas
	\abs{\Omega\setminus\cup_{i}A_i(\eta)}\leq\sum_{i} \abs{A_i\setminus A_i(\eta)}=(1-(1-2\eta)^d)\sum_{i}\abs{A_i}\leq 2d\pi_d\abs{\Omega}\eta
	\eeqas
	holds for $\eta\leq 1/2$. Then choose $\eta = \frac{\epsilon^2}{32\pi_d d \abs{\Omega}\Gamma^2}$, for any $f,g \in \mathcal{C}(\Gamma;\Omega)$, we have $l_2(f-g;\abs{\Omega\setminus\cup_{i}A_i(\eta)})\leq \epsilon/2$, and thus
	\beqas
	\log \mathcal{N}_{[\,]}(\epsilon,\mathcal{C}_2(r;\Omega)\cap \mathcal{C}(\Gamma;\Omega),l_2(\Omega))&\leq \log \mathcal{N}_{[\,]}(\epsilon/2, \mathcal{C}_2(r;\cup_{i}A_i),l_2(\cup_{i}A_i(\eta)))\\
	&\leq \sum_{i}\log \mathcal{N}_{[\,]}(\epsilon/2\sqrt{\pi_d},\mathcal{C}_2(r;A_i),l_2(A_i(\eta)))\\
	&\leq C_d \bigg(\frac{r}{\epsilon}\bigg)^{d/2}\bigg(\log \frac{C_d\Gamma^2 \abs{\Omega}}{\epsilon^2}\bigg)^{d(d+4)/4}.
	\eeqas
	Now for a general polytope $\Omega \in \mathscr{P}_k$, it can be partitioned in to $k$ simplices, i.e. $\Omega = \cup_{i=1}^k \Omega_i$ where $\Omega_i$ is of simplex type. Then
	\beqas
	&\log \mathcal{N}_{[\,]}(\epsilon,\mathcal{C}_2(r;\Omega)\cap \mathcal{C}(\Gamma;\Omega),l_2(\Omega))\\
	&\leq \sum_{i=1}^k \log \mathcal{N}_{[\,]}(\epsilon/\sqrt{k},\mathcal{C}_2(r;\Omega_i)\cap \mathcal{C}(\Gamma;\Omega_i),l_2(\Omega_i))\\
	&\leq C_d k^{(d+4)/4}\bigg(\frac{r}{\epsilon}\bigg)^{d/2}\bigg(\log \frac{C_dk\Gamma^2 \abs{\Omega}}{\epsilon^2}\bigg)^{d(d+4)/4}.
	\eeqas
	The proof is complete.
\end{proof}

Now we are ready to prove Lemma \ref{lem:local_entropy_estimate}.

\begin{proof}[Proof of Lemma \ref{lem:local_entropy_estimate}]
	Let $g \in \mathcal{P}_m$ be a piecewise affine function with at most $m$ affine components $\Omega_1,\cdots,\Omega_m$. Then each of these $d$-dimensional polytopal regions are intersected by at most $k+m$ hyperplanes, and thus have at most $C_d(k+m)^d$ facets where the constant $C_d$ depends on $d$, see Proposition 6.1.1 and the following remarks in \cite{matouvsek2002lectures}. Note that
	\beqa\label{ineq:entropy_decom}
	\mathcal{N}_{[\,]}(\epsilon,S(g,r,\Gamma),l_2(\Omega)\leq \prod_{i=1}^m \mathcal{N}_{[\,]}(\epsilon/\sqrt{m},S(g_i,r,\Gamma;\Omega_i),l_{2}(\Omega_i)).
	\eeqa
	where $g_i = g|_{\Omega_i}$. Since the map $f\mapsto f-g_i$ gives an injection from $S(g_i,r,\Gamma;\Omega_i)$ to $S(0,r,\Gamma+w(\Omega)L(g);\Omega_i)$, it follows that
	\beqas \mathcal{N}_{[\,]}(\epsilon/\sqrt{m},S(g_i,r,\Gamma;\Omega_i),l_2(\Omega_i))\leq \mathcal{N}_{[\,]}(\epsilon/\sqrt{m},S(0,r,\Gamma+w(\Omega)L(g);\Omega_i),l_{2}(\Omega_i)).
	\eeqas
	Now apply Lemma \ref{lem:entropy_boundary_control} with $\Omega=\Omega_i$ and $l_2=l_{2}(\Omega_i)$, combined with (\ref{ineq:entropy_decom}) we see that
	\beqas
	&\log \mathcal{N}_{[\,]}(\epsilon,S(g,r,\Gamma),l_2(\Omega))\\
	&\leq \sum_{i=1}^m \log \mathcal{N}_{[\,]}(\epsilon/\sqrt{m},S(0,r,\Gamma+w(\Omega)L(g);\Omega_i),l_{2}(\Omega_i))\\
	&\leq C_dm(m\vee k)^{d(d+4)/4}\left(\frac{r}{\epsilon}\right)^{d/2}\\
	&\qquad\qquad \times\bigg(\log\frac{C(m\vee k)^{d}(\Gamma^2\vee w(\Omega)^2L^2(g))\abs{\Omega}}{\epsilon^2}\bigg)^{d(d+4)/4}.
	\eeqas
	Note in the first step we required $\epsilon/\sqrt{m}\leq C_{d,0}\min_{1\leq i\leq m}\sqrt{\abs{\Omega_i}}\Gamma$. Now since for any $f\in S(f_0,r,\Gamma)$, it follows that 
	\beqas
	l^2_2(f,g)\leq 2l_2^2(f,f_0)+2l_2^2(f_0,g)\leq 2r^2+2l_2^2(f_0,g),
	\eeqas
	we thus conclude that
	\beqas
	\mathcal{N}_{[\,]}(\epsilon,S(f_0,r,\Gamma),l_2)\leq \mathcal{N}_{[\,]}\big(\epsilon,S\big(g,\sqrt{2r^2+l_2^2(f_0,g)},\Gamma\big),l_2\big),
	\eeqas
	completing the proof by taking infimum over all $g \in \mathcal{P}_m$ and $m \in \N$. Note then it suffices to require $\epsilon\leq C_{d,0}\min_{1\leq i\leq m}\sqrt{\abs{\Omega_i}}\Gamma$. 
\end{proof}

Now we are in position to prove Theorem \ref{thm:rates_conv_lse}.
\begin{proof}[Proof of Theorem \ref{thm:rates_conv_lse}]
	We will separate the cases for which $\Omega$ is a polytope or general convex body. 
	
	\noindent\textbf{[Case I.]} First consider the case when $\Omega$ is a polytope with at most $k$ simplices. If in addition $f_0 \in \mathcal{P}_{m_{f_0}}$ for some $m_{f_0} \in \N$, then the bound in Lemma \ref{lem:local_entropy_estimate} becomes
	\beqas
	&\quad \log \mathcal{N}_{[\,]}(\epsilon,S(f_0,r),l_2)\\
	&\leq c_1m_{f_0}(m_{f_0}\vee k)^{d(d+4)/4}r^{d/2}\epsilon^{-d/2}\bigg(\log\frac{c_1(m_{f_0}\vee k)^{d}(\Gamma^2\vee w(\Omega)^2L^2(f_0))\abs{\Omega}}{\epsilon^2}\bigg)^{d(d+4)/4}\\
	&\equiv c_2 r^{d/2}\epsilon^{-d/2}\bigg(\log \frac{c_3}{\epsilon^2}\bigg)^{\gamma}.
	\eeqas
	Hence in this case,
	\beqas
	& \int_{r^2/196\bar{\mu}_\Gamma}^r \sqrt{\log \mathcal{N}_{[\,]}(\epsilon,S(f_0,r),l_2)}\ \d{\epsilon}\\
	&\leq  J_{[\,]}(r)\equiv c_dc_2^{1/2}
	\begin{cases}
		r \big(\log\frac{14c_3^{1/4}\bar{\mu}_\Gamma^{1/2}}{r}\big)^{\gamma/2},& d\leq 3,\\
		r\big(\log\frac{(14c_3^{1/4}\bar{\mu}_\Gamma^{1/2})\vee (392\bar{\mu}_\Gamma)}{r}\big)^{\gamma/2+1}.& d= 4,\\
		\bar{\mu}_\Gamma^{(d-4)/4}r^{2-d/4} \big(\log\frac{14c_3^{1/4}\bar{\mu}_\Gamma^{1/2}}{r}\big)^{\gamma/2},& d\geq 5. 
	\end{cases}
	\eeqas
	By Theorem \ref{thm:risk_bounds_random_design}, $r_n$ is characterized by $\frac{J_{[\,]}(r_n)}{\sqrt{n}r_n^2}\leq \frac{1}{13C\bar{\mu}_\Gamma}$. Hence by the above calculation we can take
	\beqa\label{eqn:rates_polytope_polyhedral}
	r_n=c_d
	\begin{cases}
		\big(c_3^{1/4}\bar{\mu}_\Gamma ^{1/2}\vee c_2^{1/2}\bar{\mu}_\Gamma\big)n^{-1/2}(\log n)^{d(d+4)/8}, & d\leq 3,\\
		\big(c_3^{1/4}\bar{\mu}_\Gamma ^{1/2}\vee c_2^{1/2}\bar{\mu}_\Gamma\vee \bar{\mu}_\Gamma \big)n^{-1/2}(\log n)^{5},& d=4,\\
		\big(c_3^{1/4}\bar{\mu}_\Gamma^{1/2}\vee c_2^{2/d}\bar{\mu}_\Gamma \big)n^{-2/d}(\log n)^{(d+4)/2}, & d\geq 5.
	\end{cases}
	\eeqa
	For a general convex function $f_0$, by Lemma \ref{lem:entropy_cvx_func} the entropy estimate holds with 
	\beqas
	&\int_{r^2/196\bar{\mu}_\Gamma}^r \sqrt{\log \mathcal{N}_{[\,]}(\epsilon,S(f_0,r),l_2)}\ \d{\epsilon}\\
	&\leq \int_{r^2/196\bar{\mu}_\Gamma}^r \sqrt{\log \mathcal{N}_{[\,]}(\epsilon,\mathcal{C}(\Gamma),l_2)}\ \d{\epsilon}\\
	&\leq J_{[\,]}(r)\equiv c_d k^{1/2}\abs{\Omega}^{d/8}\Gamma^{d/4}
	\begin{cases}
		r^{1-d/4}, & d\leq 3,\\
		\log(196\bar{\mu}_\Gamma/r), & d=4,\\
		\bar{\mu}_\Gamma^{(d-4)/4}r^{2-d/2}, & d\geq 5.
	\end{cases}
	\eeqas
	This yields the following rates
	\beqa\label{eqn:rates_polytope_general}
	r_n=c_d
	\begin{cases}
		\bar{\mu}_\Gamma^{4/(d+4)}k^{2/(d+4)}(\abs{\Omega}^{1/2}\Gamma)^{d/(d+4)} n^{-2/(d+4)}, & d\leq 3,\\
		\bar{\mu}_\Gamma^{1/2}\big(k^{1/4}(\abs{\Omega}^{1/2}\Gamma)^{1/2}\vee 196\bar{\mu}_\Gamma\big) n^{-1/4}(\log n)^{1/2}, & d=4,\\
		\bar{\mu}_\Gamma^{1/2}k^{1/d}(\abs{\Omega}^{1/2}\Gamma)^{1/2} n^{-1/d}, & d\geq 5.
	\end{cases}
	\eeqa
	\noindent\textbf{[Case II.]} Suppose $\Omega$ is a smooth convex body. Then for any convex function $f_0$,
	\beqas
	J_{[\,]}(r)=c_d
	\begin{cases}
		(\abs{\Omega}^{1/2}\Gamma)^{1/2} r^{1/2}\left(\log \frac{14\bar{\mu}_\Gamma^{1/2}\abs{\Omega}^{1/4}\Gamma^{1/2}}{r}\right)^{3/4}, & d= 2,\\
		(\abs{\Omega}^{1/2}\Gamma)\log{(196\bar{\mu}_\Gamma/r)}, & d=3,\\
		(\abs{\Omega}^{1/2}\Gamma)^{(d-1)/2}r^{2-(d-1)}, & d\geq 4,
	\end{cases}
	\eeqas
	which yield the rates of convergence
	\beqa\label{eqn:rates_cvxbody_general}
	r_n=c_d
	\begin{cases}
		\big(\bar{\mu}_\Gamma^{2/3}(\abs{\Omega}^{1/2}\Gamma)^{1/3}\vee \bar{\mu}_\Gamma^{1/2}\abs{\Omega}^{1/4}\Gamma^{1/2}\big) n^{-1/3}(\log n)^{1/2}, & d= 2,\\
		\big(\bar{\mu}_\Gamma^{1/2}(\abs{\Omega}^{1/2}\Gamma)^{1/2}\vee \bar{\mu}_\Gamma\big) n^{-1/4}(\log n)^{1/2}, & d=3,\\
		\bar{\mu}_\Gamma^{1/2}(\abs{\Omega}^{1/2}\Gamma)^{1/2}n^{-1/(2(d-1))}, & d\geq 4.
	\end{cases}
	\eeqa
	This completes the proof.
\end{proof}

\section{Proofs for Section 4}

\begin{proof}[Proof of Theorem \ref{thm:risk_fixed_model}]
	Fix $m\geq 1$, $g \in \mathcal{P}_m$, and covariate vectors $\underline{X}^n=(X_1,\cdots,X_n)$. We first prove (\ref{ineq:fixed_model_bound_deviation_ineq}) and (\ref{ineq:fixed_model_bound_deviation_ineq_discrete}). Let $S_m(g,r):=\{h\in \mathcal{P}_m: l_{\underline{X}^n}(g,h)\leq r\}$. Then it is easy to check that
	\beqa\label{ineq:risk_fixed_model_1}
	\mathcal{N}(\epsilon,S_m(g,r),l_{\underline{X}^n})\leq \mathcal{N}(\sqrt{n}\epsilon,\{F_m(\underline{X}^n)-\bm{g}\}\cap B_n(0,\sqrt{n}r), \pnorm{\cdot}{2}),
	\eeqa
	where $\bm{g}:=(g(X_1),\ldots,g(X_n))$. Since translation does not change the pseudo-dimension of a set, by assumption we see that $\{F_m(\underline{X}^n)-\bm{g}\}\cap B_n(0,\sqrt{n}r)$ has pseudo-dimension at most $m$ and is uniformly bounded by $\sqrt{n}r$. Now an application of Lemma \ref{lem:entropy_pdim} yields that
	\beqa\label{ineq:risk_fixed_model_2}
	\mathcal{N}(\epsilon,S_m(g,r),l_{\underline{X}^n})\leq \bigg(4+\frac{2\sqrt{n}r}{\epsilon}\bigg)^{\kappa D_m}.
	\eeqa
	This implies that
	\beqa\label{ineq:risk_fixed_model_3}
	\int_0^r \sqrt{\log \mathcal{N}(\epsilon,S_m(g,r),l_{\underline{X}^n})}\ \d{\epsilon}&\leq \sqrt{\kappa D_m}\int_0^r \sqrt{\log\bigg(4+\frac{2\sqrt{n}r}{\epsilon}\bigg)}\ \d{\epsilon}\\
	&= r\sqrt{\kappa D_m}\int_1^\infty \frac{\sqrt{\log(4+2\sqrt{n}x)}}{x^2}\ \d{x}\\
	&\leq \sqrt{\kappa}r \sqrt{2D_m\log n}.
	\eeqa
	Here the last inequality follows by noting
	\beqas
	\log(4+2\sqrt{n}x)\leq \log(6\sqrt{n}x)= \log \sqrt{n}+\log(6x),
	\eeqas
	and hence with $n\geq 7$, we have
	\beqas
	\int_1^\infty \frac{\sqrt{\log(4+2\sqrt{n}x)}}{x^2}\ \d{x}&\leq \sqrt{\log n/2} \int_1^\infty \frac{\sqrt{1+2\log(6x)/\log 7}}{x^2}\ \d{x}\leq \sqrt{2\log n}.
	\eeqas
	By Lemma \ref{lem:rate_convergence_generic_misspecification}, for the choice
	\beqa\label{delta_n:risk_bounds_fixed_model}
	\delta_n=l_{\underline{X}^n}^2(f_0,g)+\frac{102400\sigma^2 \kappa D_m\log n}{n},
	\eeqa
	the deviation inequality holds for $\delta:=\delta_n+\frac{\sigma^2t}{n}$:
	\beqa\label{const:fixed_model_prob}
	&\quad \mathbb{P}\left[l_{\underline{X}^n}^2(\hat{f}_m,g)>8l_{\underline{X}^n}^2(f_0,g)+(409600)\frac{\sigma^2 \kappa D_m\log n}{n}+\frac{4\sigma^2 t}{n}\bigg\lvert \underline{X}^n\right]\\
	&\leq 4\sum_{j\geq 0}\exp\left(-\frac{2^{j}t}{73728}\right)\wedge 1:=\mathfrak{p}(t).
	\eeqa
	Now taking total expectation and using the triangle inequality we see that with probability at least $1-\mathfrak{p}(t)$, it holds that
	\beqa\label{const:fixed_model_deviation_ineq_discrete}
	l_{\underline{X}^n}^2(\hat{f}_m,f_0)\leq 18l_{\underline{X}^n}^2(f_0,g)+(819200)\frac{\sigma^2 \kappa D_m\log n}{n}+\frac{8\sigma^2 t}{n}.
	\eeqa
	Now with the constraint that $\mathcal{F}$ is uniformly bounded by $\Gamma$ and $f_0$ is bounded by $\Gamma$, we invoke Lemma \ref{lem:relate_metrics}, together with the fact that
	\beqa\label{ineq:discrete_continuous_connect}
	l_\nu^2(g,f_0)&\leq 2\bigg(l_\nu(g,f_0)-2l_{\underline{X}^n}(g,f_0)\bigg)_+^2+8l_{\underline{X}^n}^2(g,f_0),\\
	l_{\underline{X}^n}^2(g,f_0)&\leq 2\bigg(l_{\underline{X}^n}(g,f_0)-2l_\nu(g,f_0)\bigg)_+^2+8l_{\nu}^2(g,f_0),\\
	\eeqa
	we see with probability at least $1-\mathfrak{p}(t)-6\exp(-t)$ that
	\beqa\label{const:fixed_model_deviation_ineq_cont}
	l_{\nu}^2(\hat{f}_m,f_0)&\leq 2\bigg(l_\nu(\hat{f}_m,f_0)-2l_{\underline{X}^n}(\hat{f}_m,f_0)\bigg)_+^2+8\bigg((819200)\frac{\sigma^2 \kappa D_m\log n}{n}+\frac{8\sigma^2 t}{n}\bigg)\\
	&\quad +144l_{\underline{X}^n}^2(f_0,g)\\
	&\leq 2\bigg(l_\nu(\hat{f}_m,f_0)-2l_{\underline{X}^n}(\hat{f}_m,f_0)\bigg)_+^2+8\bigg((819200)\frac{\sigma^2 \kappa D_m\log n}{n}+\frac{8\sigma^2 t}{n}\bigg)\\
	&\quad\quad +288\bigg(l_{\underline{X}^n}(g,f_0)-2l_\nu(g,f_0)\bigg)_+^2+1152l_{\nu}^2(g,f_0)\\
	& \leq 1152l_\nu^2(f_0,g)+(7.56\times 10^6) \frac{(\sigma^2\vee \Gamma^2)D_m\log n}{n}+(3.35\times 10^5) \frac{(\sigma^2\vee \Gamma^2)t}{n},
	\eeqa
	where in the last inequality we used the fact that $\log(4+24\sqrt{2}n)\leq 3\log n$ for $n\geq 7$. The conclusion follows since $g$ is taken arbitrarily in $\mathcal{P}_m$ and the probability statement is uniform in $g$. 
	Now we prove (\ref{ineq:fixed_model_bound}) and (\ref{ineq:fixed_model_bound_discrete}). By Lemma \ref{lem:rate_convergence_generic_misspecification} and (\ref{delta_n:risk_bounds_fixed_model}), we find that
	\beqa\label{const:fixed_model_bound_discrete}
	&\mathbb{E}\big[l^2_{\underline{X}^n}(\hat{f}_m,f_0)|\underline{X}^n\big]\\
	&\leq 10l_{\underline{X}^n}^2(f_0,g)+8\bigg(l_{\underline{X}^n}^2(f_0,g)+\frac{102400\sigma^2 \kappa D_m\log n}{n}\bigg)+(4.8\times 10^6)\frac{\sigma^2}{n}\\
	&\leq 18l_{\underline{X}^n}^2(f_0,g)+(5.62\times 10^6)\frac{\sigma^2 \kappa D_m\log n}{n}
	\eeqa
	holds for all $g \in \mathcal{P}_m$. Here in the last line we used that $\kappa D_m\log n\geq 1$. Now (\ref{ineq:fixed_model_bound_discrete}) follows by taking total expectation. For  (\ref{lem:rate_convergence_generic_misspecification}),  by (\ref{ineq:discrete_continuous_connect}), we
	take conditional expectation to get
	\beqas
	\mathbb{E}\big[l_\nu^2(\hat{f}_m,f_0)|\underline{X}^n\big]&\leq 2\sup_{g \in \mathcal{P}_m(\Gamma)}\bigg(l_\nu(g,f_0)-2l_{\underline{X}^n}(g,f_0)\bigg)_+^2\\
	&\qquad\qquad+8\bigg(18l_{\underline{X}^n}^2(f_0,g)+(5.62\times 10^6)\frac{\sigma^2 \kappa D_m\log n}{n}\bigg).
	\eeqas
	Now we take expectation with respect to $\nu$ followed by infimum over $g \in \mathcal{P}_m$, by Lemma \ref{lem:relate_metrics} we see that
	\beqa\label{const:fixed_model_bound_cont}
	\mathbb{E}\big[l_\nu^2(\hat{f}_m,f)\big]
	&\leq 144\inf_{g \in \mathcal{P}_m(\Gamma)}l^2_\nu(f_0,g)+(4.5\times 10^7)\frac{(\sigma^2\vee \Gamma^2)\kappa D_m\log n}{n},
	\eeqa
	as desired.
\end{proof}

\begin{proof}[Proof of Theorem \ref{thm:lepski_generic}]
	We first consider (\ref{ineq:lepski_risk_bound_cont}) in the case $\ast=up$ and we start with the case for continuous norm. Suppose the regression function $f_0 \in \mathcal{P}_{m_0}$ and $m_0 \leq \mathfrak{M}_n^{up,c}$. For any $m\in\{1,\ldots,\mathfrak{M}_n^{up,c}\}$, define the event $A_{m,c}$ by
	\beqa\label{def:A_m_cont}
	A_{m,c}:=\bigg\{l_\nu^2(\hat{f}_m,\hat{f}_{m'})\leq \mathfrak{t}^c\frac{(\sigma^2\vee \Gamma^2)\kappa D_{m'}\log n}{n},\forall m'\in\{m,\ldots,\mathfrak{M}_n^{up,c}\}\bigg\}.
	\eeqa
	We write
	\beqa\label{ineq:lepski_bound_separation}
	\mathbb{E}\left[l_\nu^2(\hat{f}_{\hat{m}^{up,c}},f_0)\right]&= \mathbb{E}\left[l_\nu^2(\hat{f}_{\hat{m}^{up,c}},f_0)\bm{1}_{\hat{m}^{up,c}\leq m_0}\right]+\mathbb{E}\left[l_\nu^2(\hat{f}_{\hat{m}^{up,c}},f_0)\bm{1}_{\hat{m}^{up,c}> m_0}\right]\\
	&:=(I)+(II).
	\eeqa
	Then for $(I)$, we have
	\beqa\label{ineq:lepski_bound_1}
	(I)&\leq 2\mathbb{E}\left[l_\nu^2(\hat{f}_{\hat{m}^{up,c}},\hat{f}_{m_0})\bm{1}_{\hat{m}^{up,c}\leq m_0}\right]+2\mathbb{E}\left[l_\nu^2(\hat{f}_{m_0},f_0)\right]\\
	&\leq 2(\mathfrak{t}^c+\bar{\mathfrak{k}^c})\frac{(\sigma^2\vee \Gamma^2)\kappa D_{m_0}\log n}{n},
	\eeqa
	where the last inequality follows from the definition of $\hat{m}^{up,c}$ in (\ref{def:tuning_lepski}) and the risk bound (\ref{ineq:fixed_model_bound}). Now we consider $(II)$. Note that if $A_{m_0,c}$ holds, then $\hat{m}^{up,c}\leq m_0$. Thus
	\beqas
	(II) &\leq 4\Gamma^2\mathbb{P}(\hat{m}^{up,c}>m_0)\leq 4\Gamma^2 \mathbb{P}(A_{m_0}^c)\\
	&\leq 4\Gamma^2\sum_{m=m_0}^{\mathfrak{M}_n^{up,c}}\mathbb{P}\left[l_\nu^2(\hat{f}_{m_0},\hat{f}_m)>\mathfrak{t}^c\frac{(\sigma^2\vee\Gamma^2)\kappa D_m\log n}{n}\right]\\
	&\leq 4\Gamma^2\sum_{m=m_0}^{\mathfrak{M}_n^{up,c}}\bigg(\mathbb{P}\left[l_\nu^2(\hat{f}_{m_0},f_0)>\frac{\mathfrak{t}^c}{2}\frac{(\sigma^2\vee\Gamma^2)\kappa D_m\log n}{n}\right]\\
	&\qquad\qquad\qquad +\mathbb{P}\left[l_\nu^2(\hat{f}_{m},f_0)>\frac{\mathfrak{t}^c}{2}\frac{(\sigma^2\vee\Gamma^2)\kappa D_m\log n}{n}\right]\bigg)\\
	&:=4\Gamma^2\sum_{m=m_0}^{\mathfrak{M}_n^{up,c}}\big(\mathfrak{P}_m^{(1)}+\mathfrak{P}_m^{(2)}\big).
	\eeqas
	Note that
	\beqas
	\mathfrak{P}_m^{(1)}&\leq \mathbb{P}\left[l_\nu^2(\hat{f}_{m_0},f_0)>\frac{\mathfrak{t}^c}{2}\frac{(\sigma^2\vee\Gamma^2)\kappa D_{m_0}\log n}{n}\right].
	\eeqas
	Now by taking $t=4(\mathfrak{v}\vee 1)\kappa D_{m_0}\log n$ in (\ref{ineq:fixed_model_bound_deviation_ineq}), we see that with $\mathfrak{t}^c = 2(\mathfrak{k}^c+4(\mathfrak{v}\vee 1)\mathfrak{d}^c)$, the above display is further bounded by
	\beqas
	\mathfrak{P}_m^{(1)}&\leq 4\sum_{j=0}^\infty \exp\bigg(-\frac{2^j4(\mathfrak{v}\vee 1)\kappa D_{m_0}\log n}{\mathfrak{v}}\bigg)+6\exp(-4(\mathfrak{v}\vee 1)\kappa D_{m_0}\log n)\\
	&\leq \frac{5}{n^2}+\frac{6}{n^4}\leq \frac{11}{n^2}
	\eeqas
	since 
	\beqas
	\sum_{j=0}^\infty \exp\big(-2^{j}4\log n\big)\leq {1 \over n^4}+{1 \over n^4}\sum_{j=1}^\infty n^{-2^j\cdot 4+4}<\frac{1.07}{n^4}.
	\eeqas
	For $\mathfrak{P}_m^{(2)}$, note that $f_0\in \mathcal{P}_{m_0}\subset \mathcal{P}_m$ by our assumption that $\{\mathcal{P}_m\}_{m\in \N}$ is nested and that $m\geq m_0$. Now by setting $t=4(\mathfrak{v}\vee 1)\kappa D_{m}\log n$ in (\ref{ineq:fixed_model_bound_deviation_ineq}) again and repeating the above argument we see that $\mathfrak{P}_m^{(2)}\leq 11/n^2$. Hence
	\beqa\label{ineq:lepski_bound_2}
	(II)\leq 88\Gamma^2\frac{\mathfrak{M}_n^{up,c}}{n^4}\leq 88\frac{(\sigma^2\vee \Gamma^2)\kappa D_{m_0}\log n}{n}.
	\eeqa
	Combining (\ref{ineq:lepski_bound_separation}), (\ref{ineq:lepski_bound_1}) and (\ref{ineq:lepski_bound_2}) we see that
	\beqas
	\mathbb{E}\left[l_\nu^2(\hat{f}_{\hat{m}^{up,c}},f_0)\right]&\leq 2(\mathfrak{t}^c+\bar{\mathfrak{k}^c}+44)\frac{(\sigma^2\vee \Gamma^2)\kappa D_{m_0}\log n}{n}.
	\eeqas
	Now if $m_0\geq \mathfrak{M}_n^{up,c}$, possibly $m_0 =\infty$ where $\mathcal{P}_\infty:=\mathcal{F}$. Then by definition (\ref{def:tuning_lepski}), $\hat{m}^{up,c}\leq \mathfrak{M}_n^{up,c}$, and hence 
	\beqas
	\mathbb{E}\left[l_\nu^2(\hat{f}_{\hat{m}^{up,c}},f_0)\right]&\leq 2\mathbb{E}\left[l_\nu^2(\hat{f}_{\mathfrak{M}_n^{up,c}},f_0)\right]+2\mathbb{E}\left[l_\nu^2(\hat{f}_{\hat{m}^{up,c}},\hat{f}_{\mathfrak{M}_n^{up,c}})\right]\\
	&\leq 2\left(\bar{\mathfrak{c}^c}\inf_{g \in \mathcal{P}_{\mathfrak{M}_n^{up,c}}}l_\nu^2(f_0,g)+\bar{\mathfrak{k}^c}\frac{(\sigma^2\vee \Gamma^2)\kappa D_{\mathfrak{M}_n^{up,c}}\log n}{n}\right)\\
	&\qquad\qquad +2\mathfrak{t}^c \frac{(\sigma^2\vee \Gamma^2)\kappa D_{\mathfrak{M}_n^{up,c}}\log n}{n}\\
	&\leq 2\bar{\mathfrak{c}^c} \mathfrak{e}(f_0,\Omega)\mathfrak{G}(\mathfrak{M}_n^{up,c})+2(\bar{\mathfrak{k}^c}+\mathfrak{t}^c)\frac{(\sigma^2\vee \Gamma^2)\kappa D_{\mathfrak{M}_n^{up,c}}\log n}{n}\\
	&\leq 2(\bar{\mathfrak{c}^c} \mathfrak{e}(f_0,\Omega)+\bar{\mathfrak{k}^c}+\mathfrak{t}^c)\inf_{m \in \N}\bigg(\mathfrak{G}(m)+\frac{(\sigma^2\vee \Gamma^2)\kappa D_{m}\log n}{n}\bigg).
	\eeqas
	Here the second inquality follows from (\ref{ineq:fixed_model_bound}) and the definition of $\hat{m}^{up,c}$, while the last one follows from the definition of $\mathfrak{M}_n^{up,c}$. For the case of discrete norm case (\ref{ineq:lepski_risk_bound_discrete}), define the event $A_{m,d}$ instead of (\ref{def:A_m_cont}) as follows:
	\beqa
	A_{m,d}:=\bigg\{l_{\underline{X}^n}^2(\hat{f}_m,\hat{f}_{m'})\leq \mathfrak{t}^d\frac{\sigma^2\kappa D_{m'}\log n}{n},\forall m'\in\{m,\ldots,\mathfrak{M}_n^{up,d}\}\bigg\}.
	\eeqa
	Then for $f_0\in \mathcal{P}_{m_0}$, similarly we separate the risk into two terms:
	\beqa\label{ineq:lepski_bound_separation_discrete}
	\mathbb{E}\left[l_{\underline{X}^n}^2(\hat{f}_{\hat{m}^{up,d}},f_0)\right]&= \mathbb{E}\left[l_{\underline{X}^n}^2(\hat{f}_{\hat{m}^{up,d}},f_0)\bm{1}_{\hat{m}^{up,d}\leq m_0}\right]+\mathbb{E}\left[l_{\underline{X}^n}^2(\hat{f}_{\hat{m}^{up,d}},f_0)\bm{1}_{\hat{m}^{up,d}> m_0}\right]\\
	&:=(I)+(II).
	\eeqa
	The first term $(I)$ can be bounded along the same lines as in  (\ref{ineq:lepski_bound_1}), and we get :
	\beqa\label{ineq:lepski_bound_1_discrete}
	(I)&\leq 2(\mathfrak{t}^d+\bar{\mathfrak{k}^d})\frac{\sigma^2\kappa D_{m_0}\log n}{n}.
	\eeqa
	Now we handle $(II)$. Note that for any $m > m_0$ (possibly random),  it holds that $
	\pnorm{Y-\hat{f}_m}{n}^2\leq \pnorm{Y-f_0}{n}^2$. Plugging in $Y=f_0+\epsilon$ we see that $\pnorm{\hat{f}_m-f_0-\epsilon}{n}^2\leq \pnorm{\epsilon}{n}^2$. Thus
	\beqas
	\pnorm{\hat{f}_m}{n}\leq \pnorm{\hat{f}_m-f_0-\epsilon}{n}+\pnorm{f_0+\epsilon}{n}\leq \pnorm{f_0}{n}+2\pnorm{\epsilon}{n}.
	\eeqas
	It follows that $l_{\underline{X}^n}^2(\hat{f}_{\hat{m}^{up,d}},f_0)\leq 8(\pnorm{f_0}{n}^2+\pnorm{\epsilon}{n}^2)$. Now the second term $(II)$ can be further bounded by
	\beqa\label{ineq:second_term_discrete_1}
	(II) &\leq 8\mathbb{E}\big[(\pnorm{f_0}{n}^2+\pnorm{\epsilon}{n}^2)\bm{1}_{\hat{m}^{up,d}>m_0}\big]\\
	&\leq 8\pnorm{f_0}{\infty}^2 \mathbb{P}\big(\hat{m}^{up,d}>m_0\big)+8\mathfrak{u} \mathbb{P}\big(\hat{m}^{up,d}>m_0\big)+8\mathbb{P}\big(\pnorm{\epsilon}{n}^2>\mathfrak{u}\big)\\
	&\leq 8\pnorm{f_0}{\infty}^2 \mathbb{P}\big(\hat{m}^{up,d}>m_0\big)+8\mathfrak{u} \mathbb{P}\big(\hat{m}^{up,d}>m_0\big)+8\mathfrak{u}^{-1}\sigma^2.
	\eeqa
	Here in the last inequality we used Markov's inequality for the last term and the sub-Gaussianity of $\epsilon$: $
	\mathbb{P}\big(\pnorm{\epsilon}{n}^2>\mathfrak{u}\big)\leq \mathfrak{u}^{-1}\mathbb{E} \epsilon_1^2\leq \mathfrak{u}^{-1} \sigma^2$. 
	By essentially the same argument as for the case for continuous norm, we set $t=4(\mathfrak{v}\vee 1)\kappa D_{m_0}\log n$ in (\ref{ineq:fixed_model_bound_deviation_ineq_discrete}), we see that with $\mathfrak{t}^d = 2(\mathfrak{k}^d+4(\mathfrak{v}\vee 1)\mathfrak{d}^d)$, the following bound holds:
	\beqa\label{ineq:second_term_discrete_2}
	\mathbb{P}\big(\hat{m}^{up,d}>m_0\big)\leq \frac{10\mathfrak{M}_n^{up,d}}{n^4}.
	\eeqa
	Combining (\ref{ineq:second_term_discrete_1}) and (\ref{ineq:second_term_discrete_2}), with $\mathfrak{u}=n^{2}$, we see that
	\beqa
	(II)&\leq 168\frac{(\sigma^2\vee \pnorm{f}{\infty}^2)\mathfrak{M}_n^{up,d}}{n^2}\leq 168\frac{(\sigma^2\vee \pnorm{f}{\infty}^2)\kappa D_{m_0}\log n}{n}.
	\eeqa
	The rest of the proofs for the discrete are the same as to the continuous norm, and the proofs for the case $\ast =un$ are completely analogous so we shall omit the details.
\end{proof}

\begin{proof}[Proof of Theorem \ref{thm:adaptive_estimator_bound_discrete}]
	We first observe that for any $f \in \mathcal{F}$, it holds that
	\beqas
	\pnorm{f-f_0}{n}^2=\gamma_n(f)+2\iprod{\epsilon}{f}_n+\pnorm{f_0}{n}^2.
	\eeqas
	Thus it holds for any $m' \in \N$, and any $f_m \in \mathcal{P}_m$ that
	\beqa\label{ineq:basic_inequality_adaptive_estimation}
	\pnorm{\hat{f}_{m'}-f_0}{n}^2= \pnorm{{f}_m-f_0}{n}^2+\gamma_n(\hat{f}_{m'})-\gamma_n(f_m)+2\iprod{\epsilon}{\hat{f}_{m'}-{f}_m}_n.
	\eeqa
	By definition of $\hat{m}$, we have
	\beqa\label{ineq:basic_inequality_adaptive_estimation_discrete}
	\pnorm{\hat{f}_{\hat{m}}-f_0}{n}^2\leq \pnorm{{f}_m-f_0}{n}^2+\mathrm{pen}(m)-\mathrm{pen}(\hat{m})+2\iprod{\epsilon}{\hat{f}_{\hat{m}}-{f}_m}_n.
	\eeqa
	Now our goal is to control the random term $\iprod{\epsilon}{\hat{f}_{\hat{m}}-{f}_m}_n$. Fix $m \in \N$ and $f_m\in \mathcal{P}_m$. We proceed by conditioning on $\underline{X}^n$.  For any $m' \in \N$, we define for $g \in \mathcal{P}_{m'}$ the Gaussian process
	\beqas
	Z(g):=\frac{\iprod{\epsilon}{g-{f}_m}_n}{\omega(m',g)}
	\eeqas
	where $\omega(m',g):=\pnorm{g-f_0}{n}^2+\pnorm{{f}_m-f_0}{n}^2+\frac{x_{m'}}{2n}$ where $x_{m'}$ is a constant to be specified later. By Borell's inequality (cf. Proposition A.2.1 \cite{van1996weak}) we have
	\beqa\label{ineq:Borell_inequality}
	\mathbb{P}\bigg[\sup_{g \in \mathcal{P}_{m'}}Z(g)\geq \mathscr{E}+t\bigg\lvert \underline{X}^n\bigg]\leq \exp\left(-\frac{t^2}{2v^2}\right)
	\eeqa
	where $\mathscr{E}\geq \mathbb{E}\big[\sup_{g \in \mathcal{P}_{m'}}Z(g)\big\lvert \underline{X}^n\big]$ and $v^2\geq \sup_{g \in \mathcal{P}_{m'}}\mathrm{Var}[Z(g)\lvert \underline{X}^n]$. Note that
	\beqa\label{ineq:lower_bound_omega}
	\omega(m',g)\geq \frac{1}{2}\bigg(\pnorm{g-{f}_m}{n}^2+\frac{x_{m'}}{n}\bigg)\geq \pnorm{g-{f}_m}{n}\sqrt{\frac{x_{m'}}{n}}.
	\eeqa
	We now establish a bound for $\mathscr{E}$. To this end, note that
	\beqa\label{ineq:birge_adaptive_1}
	&\mathbb{P}\big(\sup_{g \in \mathcal{P}_{m'}} \abs{Z(g)}\geq t\big\lvert \underline{X}^n\big)\\
	&\leq \mathbb{P}\bigg(\sup_{g \in \mathcal{P}_{m'}}\frac{2 }{\pnorm{g-f_m}{n}^2+x_{m'}/n}\abs{{1 \over n}\sum_{i=1}^n\epsilon_i(g-f_m)(X_i)}>t\bigg\lvert \underline{X}^n\bigg)\\
	&\leq \mathbb{P}\bigg(\sup_{g \in \mathcal{P}_{m'}:\pnorm{g-f_m}{n}<\sqrt{x_{m'}/n}}\abs{{1 \over n}\sum_{i=1}^n\epsilon_i(g-f_m)(X_i)}>t\frac{x_{m'}}{2n}\bigg\lvert \underline{X}^n\bigg)\\
	&\quad +\sum_{j=1}^\infty \mathbb{P}\bigg(\sup_{g \in \mathcal{P}_{m'}: 2^{j-1}\sqrt{x_{m'}/n}\leq \pnorm{g-f_m}{n}< 2^j\sqrt{x_{m'}/n}}\abs{{1 \over n}\sum_{i=1}^n\epsilon_i(g-f_m)(X_i)}>t\frac{2^{2j-3} x_{m'}}{n}\bigg\lvert \underline{X}^n\bigg)\\
	&\leq \sum_{j=0}^\infty \mathbb{P}\bigg(\sup_{g \in \mathcal{P}_{m'}:\pnorm{g-f_m}{n}<2^j\sqrt{x_{m'}/n}}\abs{{1 \over n}\sum_{i=1}^n\epsilon_i(g-f_m)(X_i)}>t\frac{2^{2j-3}x_{m'}}{n}\bigg\lvert \underline{X}^n\bigg)\\
	&\equiv \sum_{j=0}^\infty P_j.
	\eeqa
	For each $j\geq 0$, denote $\mathcal{G}_j:=\{g \in \mathcal{P}_{m'}:\pnorm{g-f_m}{n}<2^j \sqrt{x_m'/n}\}$. It is easy to check that, by the same arguments as in (\ref{ineq:risk_fixed_model_1}) and (\ref{ineq:risk_fixed_model_2}), with $\omega_j=2^j\sqrt{x_m'/n}$ we have
	\beqas
	\mathcal{N}(\epsilon,\mathcal{G}_j,\pnorm{\cdot}{n})&\leq  \mathcal{N}(\sqrt{n}\epsilon,\{F_{m'}-\bm{f}_m\}\cap B_n(0,\sqrt{n}\omega_j),\pnorm{\cdot}{2})\leq \bigg(4+\frac{2\sqrt{n}\omega_j}{\epsilon}\bigg)^{\kappa D_{m'}},
	\eeqas
	where $\bm{f}_m:=(f_m(X_1),\ldots,f_m(X_n))$. Then by the same calculation as in (\ref{ineq:risk_fixed_model_3}) (but now using $\sqrt{2}<2$), we see that
	\beqas
	\int_0^{\omega_j}\sqrt{\log\mathcal{N}(\epsilon,\mathcal{G}_j,\pnorm{\cdot}{n})}\ \d{\epsilon}&\leq 2\sqrt{\kappa}\omega_j \sqrt{D_{m'}\log n}\\
	&= 2^{j+1}\sqrt{\kappa D_{m'}}\sqrt{\frac{x_{m'}\log n}{n}}.
	\eeqas
	Now  by Lemma \ref{lem:sup_gaussian_process}, choose $\delta_n$ so that 
	\beqas
	\sqrt{n}\delta_n&= \left(24C_2  2^{j+1}\sqrt{\kappa D_{m'}}\sqrt{\frac{x_{m'}\log n}{n}}\right) \vee \left(\sqrt{1152\log 2}C_2 2^j\sqrt{\frac{x_{m'}}{n}}\right) \\
	& = 48 C_2 2^j \sqrt{\kappa D_{m'}}\sqrt{\frac{x_{m'}\log n}{n}}.
	\eeqas
	Then for all $t>0$ such that
	\beqas
	t\frac{2^{2j-3}x_{m'}}{n}\geq \delta_n= 48 C_2 2^j \sqrt{\kappa D_{m'}}\sqrt{\frac{x_{m'}\log n}{n^2}},
	\eeqas
	or, equivalently 
	\beqa\label{ineq:birge_adaptive_2}
	t\geq 6\cdot 2^{-j+6} C_2\sqrt{\frac{\kappa D_{m'}\log n}{x_{m'}}}
	\eeqa
	we have that
	\beqa\label{ineq:birge_adaptive_3}
	P_j\leq 2C_1\exp\left[-\frac{2^{2j}x_{m'}t^2}{1152\cdot 2^6 C_2^2}\right].
	\eeqa
	It therefore follows from (\ref{ineq:birge_adaptive_1}), (\ref{ineq:birge_adaptive_2}) and (\ref{ineq:birge_adaptive_3}) that for all $t\geq 6\cdot 2^6C_2\sqrt{\frac{\kappa D_{m'}\log n}{x_{m'}}}$,
	\beqas
	\mathbb{P}\left(\sup_{g \in \mathcal{P}_{m'}} \abs{Z(g)}\geq t\bigg\lvert\underline{X}^n\right)\leq \sum_{j=0}^\infty P_j\leq 2C_1\sum_{j=0}^\infty \exp\left[-\frac{2^{2j}x_{m'}t^2}{1152\cdot 2^6 C_2^2}\right]
	\eeqas
	for $C_1,C_2>0$ taken from Lemma \ref{lem:sup_gaussian_process}. Since the $\epsilon_i$'s are sub-Gaussian with parameter $\sigma^2$, we can take $C_1=2,C_2=\sqrt{2}\sigma$. Thus
	\beqas
	\mathbb{E}\bigg[\sup_{g \in \mathcal{P}_{m'}} \abs{Z(g)}\bigg\lvert \underline{X}^n\bigg]&\leq 6\cdot 2^6 C_2\sqrt{\frac{\kappa D_{m'}\log n}{x_{m'}}}+2C_1\sum_{j=0}^\infty \int_0^\infty \exp\left[-\frac{2^{2j}x_{m'}t^2}{1152\cdot 2^6 C_2^2}\right]\ \d{t}\\
	&\leq 6\cdot 2^{6}\sqrt{2}\sigma \sqrt{\frac{\kappa D_{m'}\log n}{x_{m'}}}+\frac{3\cdot 2^{10}\sigma}{\sqrt{x_{m'}}}\\
	&\leq 3\cdot 2^8 \sigma \sqrt{\frac{\kappa D_{m'}\log n+32}{x_{m'}}}
	\eeqas
	where in the last step we used the inequality $\sqrt{a}+\sqrt{b}\leq \sqrt{2}\sqrt{a+b}$ holds for all $a,b\geq 0$. Hence we can take 
	\beqa\label{ineq:birge_adaptive_4}
	\mathscr{E}=3\cdot 2^8\sigma \sqrt{\frac{\kappa D_{m'}\log n+32}{x_{m'}}}.
	\eeqa
	Note that 
	\beqa
	\mathrm{Var}[Z(g)|\underline{X}^n]=\frac{\sigma^2\pnorm{g-{f}_m}{n}^2}{n\omega(m',g)^2}\leq \frac{\sigma^2}{x_{m'}},
	\eeqa
	where the inequality follows from (\ref{ineq:lower_bound_omega}). Thus we can take $v^2=\sigma^2/x_{m'}$ in (\ref{ineq:Borell_inequality}). Now set 
	\beqa\label{ineq:birge_adaptive_5}
	t:=\sqrt{2\sigma^2(u+L_{m'}D_{m'})/x_{m'}},
	\eeqa
	and 
	\beqa\label{ineq:birge_adaptive_6}
	x_{m'}:=9\cdot 2^{21}\big(\sigma^2\kappa D_{m'}\log n+32\sigma^2\big)+64\big(\sigma^2u+\sigma^2L_{m'}D_{m'}\big).
	\eeqa
	It therefore follows from (\ref{ineq:birge_adaptive_4}), (\ref{ineq:birge_adaptive_5}) and (\ref{ineq:birge_adaptive_6}) that
	\beqas
	\mathscr{E}+t
	&\leq \sqrt{\frac{9\cdot 2^{17}\sigma^2(\kappa D_{m'}\log n+32)+4(\sigma^2 u+\sigma^2 L_{m'}D_{m'})}{x_{m'}}}\leq \frac{1}{4}
	\eeqas
	we obtain by  (\ref{ineq:Borell_inequality}) that
	\beqa\label{ineq:birge_adaptive_7}
	\mathbb{P}\big[Z(\hat{f}_{m'})\geq 1/4\big\lvert \underline{X}^n\big]\leq \mathbb{P}\left[\sup_{g \in \mathcal{P}_{m'}}Z(g)\geq 1/4\bigg\lvert \underline{X}^n\right]\leq \exp(-u-L_{m'}D_{m'}).
	\eeqa
	Summing the inequalities (\ref{ineq:birge_adaptive_7}) with respect to $m'$, it follows that
	\beqas
	\mathbb{P}\bigg[\sup_{m'\in \N} \frac{\iprod{\epsilon}{\hat{f}_{m'}-f_m}_n}{\omega(m',\hat{f}_{m'})}\geq \frac{1}{4}\bigg\lvert \underline{X}^n\bigg]\leq \exp(-u)\sum_{m' \in \N}\exp(-L_{m'}D_{m'})=\Sigma \exp(-u).
	\eeqas
	Thus, conditional on $\underline{X}^n$, with probability at least $1-\Sigma\exp(-u)$,
	\beqas
	4\iprod{\epsilon}{\hat{f}_{\hat{m}}-f_m}_n& \leq \pnorm{\hat{f}_{\hat{m}}-f_0}{n}^2+\pnorm{f_m-f_0}{n}^2\\
	&\quad +\frac{9\cdot 2^{20}\big(\sigma^2\kappa D_{\hat{m}}\log n+32\sigma^2\big)+32\big(\sigma^2u+\sigma^2L_{\hat{m}}D_{\hat{m}}\big)}{n}.
	\eeqas
	Combined with (\ref{ineq:basic_inequality_adaptive_estimation_discrete}), it follows that
	\beqa\label{ineq:birge_adaptive_8}
	\pnorm{\hat{f}_{\hat{m}}-f_0}{n}^2&\leq 3\pnorm{{f}_m-f_0}{n}^2+2\mathrm{pen}(m)-2\mathrm{pen}(\hat{m})\\
	&\quad +\frac{9\cdot 2^{20}\big(\sigma^2\kappa D_{\hat{m}}\log n+32\sigma^2\big)+32\big(\sigma^2u+\sigma^2L_{\hat{m}}D_{\hat{m}}\big)}{n}.\\
	&=3\pnorm{{f}_m-f_0}{n}^2+2\mathrm{pen}(m)+\frac{9\cdot 2^{25}\sigma^2}{n}+\frac{32\sigma^2u}{n}
	\eeqa
	with conditional probability at least $1-\Sigma \exp(-u)$. Here the penalty is defined by
	\beqa\label{const:p_adaptive}
	\mathrm{pen}(m)=\frac{16\sigma^2D_{m}}{n}\big(9\cdot 2^{15} \kappa \log n+L_m).
	\eeqa
	Let
	\beqa
	V:=\bigg(\pnorm{\hat{f}_{\hat{m}}-f_0}{n}^2-3\pnorm{{f}_m-f_0}{n}^2-2\mathrm{pen}(m)-\frac{9\cdot 2^{25}\sigma^2}{n}\bigg)\vee 0.
	\eeqa
	Then it follows from (\ref{ineq:birge_adaptive_8}) that $\mathbb{P}[V>32\sigma^2 u/n|\underline{X}^n]\leq \Sigma \exp(-u)$, and thus $\mathbb{E}[V|\underline{X}^n]\leq 32\sigma^2\Sigma/n$. Taking conditional expectations yields
	\beqas
	\mathbb{E}\big[\pnorm{\hat{f}_{\hat{m}}-f_0}{n}^2\big\lvert\underline{X}^n\big]&\leq 3\pnorm{{f}_m-f_0}{n}^2+2\mathrm{pen}(m)+\frac{9\cdot 2^{25}\sigma^2}{n}+\mathbb{E}[V|\underline{X}^n]\\
	& \leq 3\pnorm{{f}_m-f_0}{n}^2+2\mathrm{pen}(m)+32(9\cdot 2^{20}+\Sigma)\frac{\sigma^2}{n}.
	\eeqas
	Finally by taking expectations across the last display we see that
	\beqa\label{const:adaptive_generic_discrete}
	\mathbb{E}\big[l_{\underline{X}^n}^2(\hat{f}_{\hat{m}},f_0)\big]
	&\leq 3l_\nu^2(f_m,f_0)+\big(3.02\times 10^8+9\cdot 2^{20}\kappa+32\Sigma)\frac{\sigma^2D_mL_m \log n}{n}.
	\eeqa
	Now the conclusion follows since $m$ and $f_m\in \mathcal{P}_m$ are arbitrarily chosen.
\end{proof}

\begin{proof}[Proof of Lemma \ref{lem:pdim_P_m}]
	Note that each $f \in \mathcal{P}_1$ is an affine function on $\Omega$ so $F_1$ is a linear space with dimension at most $d+1$, hence the pseudo-dimension of $F_1$ is at most $d+1$ (cf. pp. 15 \cite{pollard1990empirical}). Each $f \in \mathcal{P}_m$ where $m>1$ corresponds to a triangulation of $\Omega$ with no more than $m$ many pieces of $d$-dimensional convex bodies on which $f$ is affine. Since $f$ is the pointwise maximum over all these affine functions extended to the whole region $\Omega$, we see by the argument of Lemma 5.1 \cite{pollard1990empirical} that the pseudo-dimension of $F_m$ can be bounded by the smallest integer $l$ for which
	\begin{equation}\label{ineq:pseudo-dim}
	\binom{l}{0}+\cdots+\binom{l}{d+1}<2^{l/m}.
	\end{equation}
	Following arguments as in the proof of Lemma B.1 of  \cite{guntuboyina2012optimal}, the left hand side of the above display is bounded from above by $\left(\frac{1+\alpha}{\alpha}\right)^l\alpha^{d}$ for any $\alpha>0$. Choose $\alpha=\left(2^{1/2m}-1\right)^{-1}$, and by using the inequality $(x-1)^{-1}<(\log x)^{-1}$ to $x=2^{1/2m}$ we see that $\alpha<2m/\log 2$. Now in order that (\ref{ineq:pseudo-dim}) holds, we only need to consider $l \in \N$ for which
	\begin{equation}
	2^{l/{2m}}>\bigg(\frac{2m}{\log 2}\bigg)^{d+1}.
	\end{equation}
	Taking logarithms on both side of the above display we arrive at $l>6md\log 3m$. This completes the proof.	
\end{proof}
\begin{proof}[Proof of Lemma \ref{lem:approximation_error_P_m}]
	For a general convex function $f_0$ supported on a convex body $\Omega$ with bounded Lipschitz constant $L$, i.e. $\pnorm{f_0}{L}\leq L$, we may assume without loss of generality that $\pnorm{f}{\infty}\leq w(\Omega)L:=\Gamma$, where $w(\Omega)$ is the width of $\Omega$. We now consider the  circumscribed polytope $U_\Gamma(f_0)$ with at most $m$ facets.  For any such polytope $P$, define $g_P(x):=\inf\{t:(x,t)\in P\}$, then $g_P \in \mathcal{P}_m$. Note that $\pnorm{g_P}{L}\leq L$. Now we claim that $\pnorm{f_0-g_P}{\infty}\leq \sqrt{1+L^2}d_H(U_\Gamma(f_0),U_\Gamma(g_P))$ where $d_H$ denotes the Hausdorff distance between two sets and $U_\Gamma(f_0)$ the epigraph of $f_0$ truncated at the level $\Gamma$. To see this, for any $x,y \in \Omega$, 
	\beqas
	\abs{f_0(x)-g_P(x)}&\leq \abs{f_0(x)-f_0(y)}+\abs{f_0(y)-g_P(x)}\\
	&\leq L \pnorm{x-y}{2}+\abs{f_0(y)-g_P(x)}\\
	&\leq \sqrt{1+L^2}\pnorm{(x,g_P(x))-(y,f_0(y))}{2}.
	\eeqas
	By taking the infimum over $y$ followed by supremum over $x$ we get one direction. The other direction follows similarly, and thus the claim follows. Now we see that for any $g_P$,
	\beqas
	\inf_{g\in \mathcal{P}_m(\Gamma)}l_\nu^2(f_0,g)&\leq l_\nu^2(f_0,g_P)\leq\nu_{\max} \pnorm{f_0-g_P}{\infty}l_1(f_0,g_P)\\
	&\leq 2\nu_{\max}(1\vee L)d_H(U_\Gamma(f_0),U_\Gamma(g_P))d_N(U_\Gamma(f_0),U_\Gamma(g_P))
	\eeqas
	where $d_N$ denotes the Nikodym metric defined by $d_N(U,V)=\abs{U\Delta V}$. Since $d_N(U,V)\leq d_H(U,V)\sigma(\partial(U))$ where $\sigma(\partial(U))$ denotes the surface area (cf. \cite{bronstein2008approximation} page 732), we see that the above display is bounded by
	\beqas
	2\nu_{\max}(1\vee L)\sigma(\partial(U))d_H^2(U_\Gamma(f_0),U_\Gamma(g_P)).
	\eeqas
	Now by well-known facts in convex geometry (cf. \cite{dudley1974metric}, \cite{bronshteyn1975approximation}, page 324 in \cite{gruber1991handbook}, or page 729 \cite{bronstein2008approximation}),  for $U_\Gamma(f_0)$, we can find a circumscribing polytope $P_0$ with at most $m$ facets so that $d_H(U_\Gamma,P_0)\leq c_d{\abs{U_\Gamma(f_0)}}/{m^{2/d}}$. Then $U_\Gamma(f_0)\subset U_\Gamma(g_{P_0})\subset P_0$, $U_\Gamma{g_{P_0}}\in \mathcal{P}_m$, and hence
	\beqas
	d_H(U_\Gamma(f_0),U_\Gamma(g_{P_0}))\leq c_d \frac{\abs{U_\Gamma(f_0)}}{m^{2/d}}\leq c_d\Gamma \abs{\Omega}m^{-2/d}.
	\eeqas
	This implies that
	\beqa\label{const:approximation_error}
	\inf_{g \in \mathcal{P}_m}l_\nu^2(f_0,g)\leq \inf_{g \in \mathcal{P}_m(\Gamma)}l_\nu^2(f_0,g)\leq c_d\sigma(\partial(U_\Gamma(f_0))) \nu_{\max}(1\vee L)\Gamma^2\abs{\Omega}^2 m^{-4/d},
	\eeqa	
	as desired.
\end{proof}

\begin{proof}[Proof of Lemma \ref{lem:pseudo_dim_set_estimation}]
	The proof is essentially the same as Lemma B.1 \cite{guntuboyina2012optimal} by noting that the uniform boundedness assumption is not necessary in the proof there.
\end{proof}
\begin{proof}[Proof of Lemma \ref{lem:approximation_error_set_estimation}]
	By Theorem 1.8.11 \cite{schneider1993convex}, it follows that for two convex bodies $K,K'$,
	\beqas
	\sup_{u \in \mathbb{S}^{d-1}}\abs{h_K(u)-h_{K'}(u)}=d_H(K,K'),
	\eeqas
	where $d_H(\cdot,\cdot)$ denotes the Hausdorff distance. Now the conclusion follows by noting 
	\beqas
	l_r^2(K,K')\vee l_f^2(K,K')\leq \sup_{u \in \mathbb{S}^{d-1}}\abs{h_K(u)-h_{K'}(u)}^2=d_H^2(K,K'),
	\eeqas
	and the result from \cite{bronshteyn1975approximation}, or see results in Section 4.1 in \cite{bronstein2008approximation}.
\end{proof}

\section{Proofs  for Section 5}
\subsection{Proof of Lemma \ref{lem:inconsist_blse_zero}}

\begin{proof}[Proof of Lemma \ref{lem:inconsist_blse_zero}]
	Note that $\hat{f}_n$ is the least squares estimator if and only if
	\beqa\label{ineq:character_constrained_cvx_lse}
	\sum_{i=1}^n (g(X_i)-\hat{f}_n(X_i))(Y_i-\hat{f}_n(X_i))\leq 0
	\eeqa
	holds for all $g \in \mathcal{C}(1)$. This is a direct result of Moreau's decomposition theorem (cf. \cite{moreau1962decomposition}; and see \cite{seijo2011nonparametric} Lemma 2.4). Now suppose $\hat{f}_n(0)\to_p 0$. By passing to a subsequence we may strenghthen the convergence to almost surely convergence. Thus for $n$ large enough we may assume $\hat{f}_n(0)\leq 0.99$. Since $\mathbb{E}[l_2^2(\hat{f}_n,f_0)]\to 0$ by Theorem \ref{thm:rates_conv_lse}, by passing to a further subsequence we have $l_2(\hat{f}_n,f_0)\to 0$ almost surely. This means there is a further subsequence so that $\hat{f}_n \to f_0$ almost everywhere on $[0,1]$ almost surely, and hence uniformly within the interior of $[0,1]$ by Theorem 10.8 in  \cite{rockafellar1997convex}. Thus the convergence is uniform near $0$ almost surely by assumed convergence of $\hat{f}_n$ at $0$ via Lemma \ref{lem:uniform_conv_convexseq}. We will work with this subsequence in the sequel. Since $X_{(1)}\to 0$ almost surely, by uniform convergence it follows that $\hat{f}_n(X_{(1)})\to 0$ almost surely.  Now we choose a test function $g$ that agrees with $\hat{f}_n$ on $[X_{(2)},1]$ and is linear on $[0,X_{(2)}]$ with $g(0)=1$. This is well defined almost surely for $n$ large enough. Since $\hat{f}_n(X_{(i)})\to 0$ for $i=2,3$ and hence the convexity is guaranteed for $n$ large. Then it can be seen by the characterization (\ref{ineq:character_constrained_cvx_lse}) that
	\beqas
	(g(X_{(1)})-\hat{f}_n(X_{(1)}))(Y_{(1)}-\hat{f}_n(X_{(1)}))\leq 0,
	\eeqas
	where $Y_{(j)}$ is the value of $Y_j$ corresponding to $X_{(j)}$. Since $g(X_1)\geq \hat{f}_n(X_1)$ by construction and the fact $\hat{f}_n(0)\leq 0.99$, we have necessarily $Y_{(1)}\leq \hat{f}_n(X_{(1)})$ holds for all $n \in \N$ almost surely. Taking $n \to \infty$ we find $0\geq Y_1=_d\epsilon_1$ almost surely, a contradiction. 
\end{proof}
\begin{lemma}\label{lem:uniform_conv_convexseq}
	Let $f_n$ be a sequence of convex functions converging pointwise to a continuous convex function $f$ on $[0,1]$. Then the convergence is uniform over $[0,1]$.
\end{lemma}
\begin{proof}[Proof of Lemma \ref{lem:uniform_conv_convexseq}]
	We only need to prove uniform convergence near the boundary $0$ since uniform convergence within the interior of $[0,1]$ is guaranteed by Theorem 10.8 in \cite{rockafellar1997convex}. Let $x_n\in \arg\min_{x \in [0,1]}f_n(x)$. By passing to a subsequence we assume $x_n \to x^\ast \in [0,1]$. If $x^\ast \neq 0$, then choose $\delta \in (0,x^\ast)$ so that $x_n>\delta$ for $n$ large enough. In this case, $f_n$'s and $f_0$ are all decreasing on $[0,\delta]$ for $n$ large enough, and thus $f_n$ converges uniformly to $f_0$ on $[0,\delta]$. To see this, for fixed $\epsilon>0$, since $f$ is uniformly continuous, we can find a sequence  $\{t_i\}_{i=1}^m \subset [0,\delta]$ with $t_1=0,t_m=\delta$ so that $\abs{f(t_i)-f(t_{i-1})}\leq \epsilon$ for all $i$. For $n$ large enough we have $\abs{f_n(t_i)-f(t_i)}\leq \epsilon$ for all $i$. Now note that for any $x \in [0,\delta]=\cup_{i=1}^{m-1}[t_i,t_{i+1}]$ where, for $x \in [t_i,t_{i+1}]$, by \emph{convexity} of $f_n$,
	\beqa\label{ineq:uniform_conv_convexseq_1}
	f_n(x)\leq \max\{f_n(t_i),f_n(t_{i+1})\} \leq \max\{f(t_i),f(t_{i+1})\}+\epsilon\leq f(x)+2\epsilon.
	\eeqa
	On the other hand, by \emph{monotonicity} of $f_n$, 
	\beqa
	f_n(x)-f(x)\geq f_n(t_{i+1})-f(t_{i})=\big(f_n(t_{i+1})-f(t_{i+1})\big)+\big(f(t_{i+1})-f(t_{i})\big)\geq -2\epsilon.
	\eeqa
	This establishes our claim that $f_n$ converges uniformly to $f_0$ on $[0,\delta]$. Now we consider the case $x^\ast =0$. First note that $f$ must be non-decreasing. By $f_n(0)\geq f_n(x_n)$ we see that $\limsup_{n \to \infty}f_n(x_n)\leq f(0)$. Suppose there is a subsequence so that $f_n(x_n)<f(0)-\eta$ for some $\eta>0$ for $n$ large enough. Note for any fixed $\zeta>0$, $f_n(\zeta)>f(\zeta)-\eta/2\geq f(0)-\eta/2$ for $n$ large. Thus by convexity of $f_n$, for $n$ so large that $x_n<\zeta$, we have
	\beqas
	f_n(1)&\geq \frac{(1-x_n)f_n(\zeta)-(1-\zeta)f_n(x_n)}{\zeta-x_n}\\
	&\geq \frac{(1-x_n)\big(f(0)-\eta/2\big)-(1-\zeta)\big(f(0)-\eta\big)}{\zeta-x_n}.
	\eeqas
	Now taking $n \to \infty$ followed by $\zeta \searrow 0$ we see that $\liminf_{n \to \infty}f_n(1)=+\infty$, a contradiction. Thus $f_n(x_n)\to f(0)$. By uniform continuity of $f$, for fixed $\epsilon$, $f(0)\geq f(t)-\epsilon$ for small enough $t>0$. Hence
	\beqas
	f_n(t)\geq f_n(x_n)\geq f(0)-\epsilon\geq f(t)-2\epsilon
	\eeqas
	holds uniformly in $t>0$ small enough and $n$ large enough. Finally note that (\ref{ineq:uniform_conv_convexseq_1}) holds regardless the value of $x^\ast$, completing the proof.
\end{proof}

\subsection{Proof of Theorem \ref{thm:lse_without_log}}\label{section:proof_withoutlog_lse}
The proof of Theorem \ref{thm:lse_without_log} makes use of a recent result by \cite{chatterjee2014new}. He showed that in the model $Y=\mu+\epsilon$ where $\epsilon \sim_d \mathcal{N}(0,I_n)$ and $\mu \in K$ for a closed convex set $K$,  the risk in discrete $l_2$ norm for least squares estimation of the mean vector $\mu \in \R^n$ can be characterized by the maxima of the map
\beqa\label{map:lse_risk}
t \mapsto \mathbb{E}\bigg(\sup_{\nu \in K:\pnorm{\nu-\mu}{2}\leq t} \iprod{\epsilon}{\nu-\mu}\bigg)-\frac{t^2}{2}.
\eeqa
In the setup of univariate convex regression on $[0,1]$ in \cite{guntuboyina2013global}, we can take 
\beqa\label{eqn:cone_convex}
K:=\{\mu \in \R^n: \mu_i=f(x_i), \forall i=1,\ldots,n, f\textrm{ convex}\}.
\eeqa
The supremum in (\ref{map:lse_risk}) for $K$ as in (\ref{eqn:cone_convex}) is computed in Theorem \ref{thm:convex_t_mu}, which becomes the key ingredient to derive the risk bounds without logarithmic factors in Theorem \ref{thm:lse_without_log}. 

We follow the convention that when the supremum is taken over the empty set, the value is set to be $-\infty$. Let $t_c:=\inf_{\nu \in K}\pnorm{\nu-\mu}{}$.
\begin{theorem}[Theorem 1.1 \cite{chatterjee2014new}]
	$f_\mu(t)=-\infty$ when $t<t_c$, is finite and strictly concave when $t \in [t_c,\infty)$, and decays to $-\infty$ as $t \to \infty$. Hence the maximizer
	\beqas
	t_\mu:=\argmax_{t\geq 0} f_\mu(t)
	\eeqas
	exists and is unique. Moreover, for any $x\geq 0$, we have
	\beqas
	\mathbb{P}\bigg(\abs{\pnorm{\mu-\hat{\mu}}{2}-t_\mu}\geq x\sqrt{t_\mu}\bigg)\leq 3\exp\bigg(-\frac{x^4}{32(1+x/\sqrt{t_\mu})^2}\bigg).
	\eeqas
\end{theorem}
This immediately entails the following result.
\begin{corollary}
	There exists an absolute constant $C$ such that
	\begin{enumerate}
		\item $t_\mu^2-Ct_\mu^{3/2}\leq \mathbb{E}\pnorm{\mu-\hat{\mu}}{}^2\leq t_\mu^2+Ct_\mu^{3/2}$ if $t_\mu\geq 1$, and
		\item $\mathbb{E}\pnorm{\mu-\hat{\mu}}{}^2\leq C$ if $t_\mu <1$.
	\end{enumerate}
\end{corollary}

\begin{proposition}
	Let $0\leq r_1\leq r_2$.
	\begin{enumerate}
		\item If $f_\mu(r_1)\leq f_\mu(r_2)$, then $t_\mu\geq r_1$.
		\item If $f_\mu(r_1)\geq f_\mu(r_2)$, then $t_\mu\leq r_2$.
	\end{enumerate}
	In particular, if $f_\mu(r)\leq 0$, then $t_\mu \leq r$.
\end{proposition}

Recall that the fixed design is $\{x_k:=\frac{k-1}{n-1}\}_{k=1}^n$. Our goal is to prove the following theorem.
\begin{theorem}\label{thm:convex_t_mu}
	Consider the convex set $K$ defined by
	\beqas
	K:=\left\{\mu \in \R^n : \mu_i=f\left(x_i\right),\forall i=1,\cdots,n, f\textrm{ convex}\right\}.
	\eeqas
	Then for fixed $\mu \in K$, $t_\mu\leq C(1+\Delta \mu)^{1/5}n^{1/10}$ when $n\geq \inf \{n \in \N:(1+\Delta \mu)^{1/10} \log n\leq C' n^{1/5}\}$. Here $C,C'$ are absolute constants, and $\Delta \mu:=\mu_{\max}-\mu_{\min}$.
\end{theorem}

To prove Theorem \ref{thm:convex_t_mu}, we will need the following lemma.

\begin{lemma}\label{lem:autentropybound}
	Let $a<b$ be two real numbers. Let
	\beqas
	Q_{a,b}:=\left\{\bigg(\big(f(x_1),\cdots,f(x_n)\big)\wedge b\bm{1}\bigg)\vee a\bm{1}: f \textrm{ is convex}\right\}.
	\eeqas
	Then for any $0<t<C_1$, the following holds for $n\geq C_2$:
	\beqas
	\log \mathcal{N}\big(t,Q_{a,b},\pnorm{\cdot}{2}\big)\leq C\frac{n^{1/4}(b-a)^{1/2}}{t^{1/2}},
	\eeqas
	where $C$ is an absolute constant, $C_2=\inf\{n\in \N:2\sqrt{C_1}\log n\leq Cn^{1/4}\}$.
\end{lemma}
\begin{proof}[Proof of Lemma \ref{lem:autentropybound}]
	First note that we can require $f \geq a$ in the definition of $Q_{a,b}$. We only have to show that
	\beqa\label{ineq:canonicalentbound}
	\log \mathcal{N}\big(t,Q_{0,1},\pnorm{\cdot}{2}\big)\leq C\frac{n^{1/4}}{t^{1/2}}.
	\eeqa
	This can be seen by the following simple rescaling argument: Let $L$ be the linear transformation mapping $a$ to $0$ and $b$ to $1$, i.e. $L(x)=(x-a)/(b-a)$. Let $\mathcal{Q}$ be a $t/(b-a)$-cover of $Q_{0,1}$. If we can show (\ref{ineq:canonicalentbound}), then we can choose $\mathcal{Q}$ of cardinality bounded by $\exp(C\frac{n^{1/4}(b-a)^{1/2}}{t^{1/2}})$ and
	\beqas
	\sup_{\mu \in Q_{0,1}}\inf_{\nu \in \mathcal{Q}} \pnorm{\mu-\nu}{2}\leq {t \over {b-a}}.
	\eeqas
	However, by the bijection $L:Q_{a,b} \to Q_{0,1}$ we know that
	\beqas
	\sup_{\mu \in Q_{a,b}}\inf_{\nu \in \mathcal{Q}} \pnorm{\mu-L^{-1}\nu}{2}\leq t.
	\eeqas
	This implies that $L^{-1}(\mathcal{Q})$ is a $t$-cover of $Q_{a,b}$. Hence we only have to show (\ref{ineq:canonicalentbound}). For simplicity of notation we denote $Q_{0,1}$ by $Q$. For $0\leq k \leq n-1$ and $1\leq j \leq k+1$, we define
	\beqas
	Q_{n-k}^{(j)}:=\left\{\bigg(\big(f(x_1),\cdots,f(x_n)\big)\wedge\bm{1}\bigg)\vee \bm{0}: 0\leq f\vert_{[x_j,x_{j+n-k-1}]}\leq 1, f \textrm{ convex}\right\}.
	\eeqas
	Then $Q=\cup_{k=0}^{n-1}\cup_{j=1}^{k+1} Q_{n-k}^{(j)}$. Consider the following claim:
	\beqa\label{ineq:canonicalentbound2}
	\log \mathcal{N}\big(t,Q_0,\pnorm{\cdot}{2}\big)\leq C\frac{n^{1/4}}{t^{1/2}},
	\eeqa
	where $Q_0:=\{\big(f(x_1),\cdots,f(x_n)\big):f\in [0,1]\textrm{ and convex}\}$, and $C$ an absolute constant which may change from line to line. Note for each $Q_{n-k}^{(j)}$ we only have to consider covers on the subinterval $[x_j,x_{j+n-k-1}]$ since we can set the value of each covering point to be $1$ or $0$ elsewhere. If we can show (\ref{ineq:canonicalentbound2}), by rescaling we find the covering problem for $Q_{n-k}^{(j)}$ is equivalent to that of finding a $t$-covering with $n-k$ data points interpolating $[0,1]$, i.e.
	\beqas
	\log \mathcal{N}\big(t,Q_{n-k}^{(j)},\pnorm{\cdot}{2}\big)\leq C\frac{{(n-k)}^{1/4}}{t^{1/2}},
	\eeqas
	holds for all $1\leq j \leq k+1$. Now since $\{Q_{n-k}^{(j)}\}_{j,k}$ gives a partition of $Q$, we have
	\beqas
	\mathcal{N}\big(t,Q,\pnorm{\cdot}{2}\big) &\leq \sum_{k=0}^{n-1}\sum_{j=1}^{k+1}\mathcal{N}\big(t,Q_{n-k}^{(j)},\pnorm{\cdot}{2}\big)\\
	&\leq \sum_{k=1}^{n-1}(k+1)\exp\bigg(C\frac{{(n-k)}^{1/4}}{t^{1/2}}\bigg)\\
	&\leq n^2 \exp\bigg(C\frac{{n}^{1/4}}{t^{1/2}}\bigg)\leq \exp\bigg(\frac{2Cn^{1/4}}{t^{1/2}}\bigg),\\
	\eeqas
	where the last inequality follows when $2\sqrt{C_1}\log n \leq C n^{1/4}$ and $0<t<C_1$. This implies we only have to show (\ref{ineq:canonicalentbound2}). Let $\mathcal{F}$ be all convex functions on $[0,1]$ with values in $[0,1]$. Then by entropy estimate of one-dimensional bounded convex functions on $[0,1]$ (cf. Lemma \ref{lem:entropy_cvx_func}), for any $\epsilon>0$ we have
	\beqas
	\log \mathcal{N}\big(\epsilon,\mathcal{F},l_2\big)\leq C\epsilon^{-1/2}.
	\eeqas
	Now for fixed $t>0$, let $\epsilon:={t}/{\sqrt{24n}}$, we can choose a finite subset $\mathcal{G}\subset \mathcal{F}$ such that $\log \abs{\mathcal{G}}\leq C\epsilon^{-1/2}$ and
	$\sup_{f \in \mathcal{F}}\inf_{g \in \mathcal{G}}\pnorm{f-g}{l_2}\leq \epsilon$. This gives the map $\mathfrak{g}:\mathcal{F}\to \mathcal{G}$ by assigning each $f \in \mathcal{F}$ to an element $\mathfrak{g}(f) \in \mathcal{G}$ so that $\pnorm{f-\mathfrak{g}(f)}{l_2}\leq \epsilon$. For any $\mu \in Q_0$, let $f^\mu:[0,1]\to [0,1]$ be the linear interpolation on $\{(x_k,\mu_k)\}_{k=1}^n$. Denote this map $\mathfrak{f}:Q_0\to \mathcal{F}$, and the composite map $\mathfrak{G}:=\mathfrak{g}\circ \mathfrak{f}:Q_0 \to \mathcal{G}$. Conversely, for any element $g \in \mathcal{G}$, the map $\mathfrak{Q}:\mathcal{G}\to Q_0$ is defined by assigning $\mathfrak{Q}(g)$ to an element in $Q_0$ so that $\mathfrak{G}\big[\mathfrak{Q}(g)]=g$. Now consider a subset $\mathcal{Q}_0\subset Q_0$ defined by the image of $\mathfrak{Q}$, i.e. $\mathcal{Q}_0=\mathfrak{Q}(\mathcal{G})$. Then clearly $\abs{\mathcal{Q}_0}\leq \abs{\mathcal{G}}\leq \exp(C\epsilon^{-1/2})$. For each $\mu \in Q_0$, we have
	\beqas
	\pnorm{f^\mu-f^{\mathfrak{Q}\circ\mathfrak{G}(\mu)}}{l_2}\leq \pnorm{f^\mu-\mathfrak{g}(f^\mu)}{l_2}+\pnorm{\mathfrak{g}(f^\mu)-f^{\mathfrak{Q}\circ\mathfrak{G}(\mu)}}{l_2}\leq 2\epsilon,
	\eeqas
	where the last inequality follows from the definition of the map $\mathfrak{g}$ and the observation that
	\beqas
	\mathfrak{g}(f^\mu)=\mathfrak{g}\circ \mathfrak{f}(\mu)=\mathfrak{G}(\mu)=\big(\mathfrak{G}\circ \mathfrak{Q}\big)\circ \mathfrak{G}(\mu)=\mathfrak{G}\big(\mathfrak{Q}\circ \mathfrak{G}(\mu)\big)=\mathfrak{g}(f^{\mathfrak{Q}\circ\mathfrak{G}(\mu)}).
	\eeqas
	Here we used the fact that $\mathfrak{G}\circ \mathfrak{Q}=\mathrm{id}_{\mathcal{G}}$ by definition of $\mathfrak{Q}$.

	On the other hand, for $\mu,\nu \in Q_0$, we have
	\beqas
	\pnorm{f^\mu-f^\nu}{l_2}^2&=\sum_{i=1}^{n-1}\int_{x_i}^{x_{i+1}}\big(f^\mu(x)-f^\nu(x)\big)^2\ \d{x}\\
	&=\sum_{i=1}^{n-1}{1 \over {n-1}}\int_0^1\big(y(\mu_{i+1}-\nu_{i+1})+(1-y)(\mu_i-\nu_i)\big)^2\ \d{y}\\
	&= \sum_{i=1}^{n-1} \frac{1}{n-1} \bigg\{c_{i+1}^2 \int_0^1 y^2\ \d{y}+2c_{i+1}c_i\int_0^1 y(1-y)\ \d{y}+c_i^2\int_0^1 (1-y)^2\ \d{y}\bigg\}\\
	&=\frac{1}{3(n-1)}\sum_{i=1}^{n-1}\big(c_{i+1}^2+c_{i+1}c_i+c_i^2\big)\\
	&\geq {1 \over {6n}}\sum_{i=1}^{n-1} \big(c_{i+1}^2+c_{i}^2\big)\geq \frac{1}{6n}\sum_{i=1}^n c_i^2=\frac{1}{6n}\sum_{i=1}^n (\mu_i-\nu_i)^2,
	\eeqas
	where $c_i\equiv \mu_i-\nu_i$. Here we used the inequality that $a^2+ab+b^2\geq(a^2+b^2)/2$ for all $a,b \in \R$. This implies
	\beqas
	\pnorm{\mu-\mathfrak{Q}\circ\mathfrak{G}(\mu)}{2}\leq \sqrt{6n}\pnorm{f^\mu-f^{\mathfrak{Q}\circ\mathfrak{G}(\mu)}}{l_2}\leq \sqrt{24n}\epsilon\leq t.
	\eeqas
	Thus $\mathcal{Q}_0$ is a $t$-cover for $Q_0$, and the cardinality 
	\beqas
	\abs{\mathcal{Q}_0}\leq \exp\left(C\epsilon^{-1/2}\right)=\exp\left(C\frac{n^{1/4}}{t^{1/2}}\right),
	\eeqas
	as desired.
\end{proof}

Now we are ready to prove Theorem \ref{thm:convex_t_mu}.

\begin{proof}[Proof of Theorem \ref{thm:convex_t_mu}]
	Fix $\mu \in K$. Let $l$ be an integer to be determined later. Let $K'$ be the `truncation' set $Q_{L,R}$ defined in Lemma \ref{lem:autentropybound} where
	\beqas
	L=\mu_{\textrm{min}}-2^l,R=\mu_{\textrm{max}}+2^l.
	\eeqas
	For fixed $0<t<C_1$, let $K'':=\{\nu \in K':\pnorm{\nu-\mu}{2}\leq t\}$. Then by Dudley's entropy bound (cf. Lemma \ref{lem:dudleyentbound}) and Lemma \ref{lem:autentropybound}, we find that
	\beqa\label{estimate1}
	\mathbb{E}\sup_{\nu \in K''} \epsilon\cdot(\nu-\mu)&\leq \int_0^t \sqrt{C(2^{l+1}+\Delta \mu)^{1/2}n^{1/4}s^{-1/2}}\ \d{s}\\
	&= C(2^{l+1}+\Delta\mu)^{1/4} n^{1/8}t^{3/4},
	\eeqa
	when $n\geq C_2$ where $C_2=\inf\{n\in \N:2\sqrt{C_1}\log n\leq Cn^{1/4}\}$. Here $\Delta \mu=\mu_{\textrm{max}}-\mu_{\textrm{min}}$. Take $\nu \in K$ such that $\pnorm{\nu-\mu}{2}\leq t$. Let $\nu'(l):=\big(\nu \wedge (\mu_{\textrm{max}}+2^l)\big)\vee (\mu_{\textrm{min}}-2^l)$. It is easy to see that $\nu' \in K''$. By convexity of $\nu$, the set $\{i:\nu_i \neq \nu_i'(l)\}$ can be partitioned into three sets (possibly empty)
	\beqas
	I_1(l)=[1,\pi_L(l)]\cap \N,\quad I_2(l)=[m_L(l),m_R(l)]\cap \N,\quad I_3(l)=[\pi_R(l),n]\cap \N,
	\eeqas
	where $\nu_i-\mu_{\textrm{max}}\geq 2^l$ for $i \in I_1\cup I_3$ and $\nu_i-\mu_{\textrm{min}} \leq -2^l$ for $i \in I_2$. Note that for $L>0$,
	\beqas
	\abs{\{i: \nu_i-\mu_{\textrm{max}}\geq L\}}\vee \abs{\{i: \nu_i-\mu_{\textrm{min}}\leq -L\}}\leq \frac{t^2}{L^2},
	\eeqas
	otherwise $\pnorm{\nu-\mu}{2}>t$. This implies that $\max_{i=1,2,3}\abs{I_i(l)}\leq \frac{t^2}{2^{2l}}$. Furthermore we have
	\beqas
	\epsilon\cdot (\nu-\nu'(l))&\leq \sum_{i \in I_1(l)} \abs{\epsilon_i}\abs{\nu_i-\nu_i'(l)}\\
	&\qquad+\sum_{i \in I_2(l)} \abs{\epsilon_i}\abs{\nu_i-\nu_i'(l)}+\sum_{i \in I_3(l)} \abs{\epsilon_i}\abs{\nu_i-\nu_i'(l)}\\
	&=\sum_{k\geq l}\bigg(\sum_{i \in I_1(k)\setminus I_1(k+1)}\abs{\epsilon_i}\abs{\nu_i-\nu_i'(l)}\\
	&\qquad +\sum_{i \in I_2(k)\setminus I_2(k+1)}\abs{\epsilon_i}\abs{\nu_i-\nu_i'(l)}+\sum_{i \in I_3(k)\setminus I_3(k+1)}\abs{\epsilon_i}\abs{\nu_i-\nu_i'(l)}\bigg)\\
	&\leq \sum_{k\geq l}2^{k+1} \big(\sum_{i \in I_1(k)\setminus I_1(k+1)}\abs{\epsilon_i}+\sum_{i \in I_2(k)\setminus I_2(k+1)}\abs{\epsilon_i}+\sum_{i \in I_3(k)\setminus I_3(k+1)}\abs{\epsilon_i}\big).\\
	\eeqas
	Hence
	\beqa\label{estimate2}
	\mathbb{E}\sup_{\nu:\pnorm{\nu-\mu}{2}\leq t}\epsilon \cdot (\nu-\nu'(l))\leq\sum_{k\geq l}2^{k+1}\cdot\big(3\frac{t^2}{2^{2k}}\big)=\frac{12t^2}{2^{l}}.
	\eeqa
	Combining (\ref{estimate1}) and (\ref{estimate2}), we obtain
	\beqas
	\mathbb{E}\sup_{\nu \in K:\pnorm{\nu-\mu}{2}\leq t}\epsilon\cdot(\nu-\mu)&\leq \mathbb{E}\sup_{\nu \in K:\pnorm{\nu-\mu}{2}\leq t}\epsilon\cdot(\nu-\nu'(l))\\
	&\qquad+\mathbb{E}\sup_{\nu \in K:\pnorm{\nu-\mu}{2}\leq t}\epsilon\cdot(\nu'(l)-\mu)\\
	& \leq \mathbb{E}\sup_{\nu \in K:\pnorm{\nu-\mu}{2}\leq t}\epsilon\cdot(\nu-\nu'(l))+\mathbb{E}\sup_{\nu \in K''}\epsilon\cdot(\nu-\mu)\\
	& \leq \frac{12t^2}{2^{l}}+C(2^{l+1}+\Delta\mu)^{1/4} n^{1/8}t^{3/4}
	\eeqas
	Choose $l \in \N$ such that $l\geq \log_2(48)$. Then
	\beqas
	f_\mu(t)\leq C'(1+\Delta\mu)^{1/4}n^{1/8}t^{3/4}-{1 \over 4}t^2,
	\eeqas
	holds when $0<t\leq C_1$ and $n\geq C_2$ where $C_2=\inf\{n \in \N: 2\sqrt{C_1}\log n\leq Cn^{1/4}\}$. Here $C,C'$ are absolute constants. Let $C_1=r=\big(4C'(1+\Delta \mu)^{1/4}\big)^{4/5}n^{1/10}$. Then $f_\mu(r)=0$ and hence $t_\mu\leq r=\big(4C'(1+\Delta \mu)^{1/4}\big)^{4/5}n^{1/10}$ when $n\geq \inf \{n \in \N:2(4C'(1+\Delta \mu)^{1/4})^{2/5} \log n\leq C n^{1/5}\}$.
\end{proof}

\section{Technical lemmas}\label{section:technical_lemma}
 
Our goal is to prove the following result.
\begin{lemma}\label{lem:relate_metrics}
	Let $\mathcal{P}_m$ be a function class uniformly bounded by $\Gamma$. Suppose $F_m(\underline{X}^n)\leq D_m$ holds where $F_m$ is defined in (\ref{def:F_m}). Then for any $f \in \mathcal{F}$ uniformly bounded by $\Gamma$, and any probability measure $\nu$ on $\Omega$ and $u>0$,
	\beqa\label{ineq:exp_bound_cont_disc}
	&\quad \mathbb{P}\bigg[\sup_{g \in \mathcal{P}_m}\left(l_\nu(f,g)-2l_{\underline{X}^n}(f,g)\right)_{+}^2>\frac{1152\Gamma^2}{n}\left(\kappa D_m\log (4+24\sqrt{2} n)+u\right)\bigg] \\
	&\vee \mathbb{P}\bigg[\sup_{g \in \mathcal{P}_m}\left(l_{\underline{X}^n}(f,g)-2l_{\nu}(f,g)\right)_{+}^2>\frac{1152\Gamma^2}{n}\left(\kappa D_m\log (4+24\sqrt{2} n)+u\right)\bigg] \\
	&\leq 3\exp(-u).
	\eeqa
	Consequently,
	\beqa\label{ineq:E_bound_cont_disc}
	& \mathbb{E}\sup_{g \in \mathcal{P}_m}\left(l_\nu(f,g)-2l_{\underline{X}^n}(f,g)\right)_{+}^2 \vee \mathbb{E}\sup_{g \in \mathcal{P}_m}\left(l_{\underline{X}^n}(f,g)-2l_\nu(f,g)\right)_{+}^2 \\
	&\leq 6912\kappa \frac{\Gamma^2}{n}D_m\log (4+24\sqrt{2}n).
	\eeqa
\end{lemma}
The key ingredient to prove the above lemma is the following result:
\begin{lemma}\label{lem:relate_metrics_general}
	Let $\mathcal{F}$ be a class of uniformly bounded functions on $\R^d$ with $B:=\sup_{f \in \mathcal{F}}\pnorm{f}{\infty}$. Let $\nu$ be a probability measure on $\R^d$. Then
	\beqas
	&\quad \mathbb{P}_{\nu^{\otimes n}}\left[\pnorm{f}{l_\nu}-2\pnorm{f}{l_{\underline{X}^n}}>\epsilon \textrm{ for some }f \in \mathcal{F}\right]\\
	&\vee \mathbb{P}_{\nu^{\otimes n}}\left[\pnorm{f}{l_{\underline{X}^n}}-2\pnorm{f}{l_\nu}>{\epsilon} \textrm{ for some }f \in \mathcal{F} \right]\\
	&\leq 3\mathbb{E}_{\nu^{\otimes 2n}}\mathcal{N}\left(\frac{\sqrt{2}}{24}\epsilon,\mathcal{F},l_{\underline{X}^{2n}}\right)\exp\left(-\frac{n\epsilon^2}{288B^2}\right).
	\eeqas
\end{lemma}
\begin{proof}[Proof of Lemma \ref{lem:relate_metrics_general}]
	The proof for the first statement readily follows from Theorem 11.2 \cite{gyorfi2002distribution}. For the second, we indicate the key step in the proof of Theorem 11.2 \cite{gyorfi2002distribution}. Let $\underline{X}^{n'}:=(X_{n+1},\ldots,X_{2n})$ be i.i.d. ghost samples distributed according to the probability law $\nu$. For simplicity of notation, we denote $\pnorm{f}{n}\equiv \frac{1}{n}\sum_{i=1}^n f(X_i)$, $\pnorm{f}{n}'\equiv \frac{1}{n}\sum_{i=n+1}^{2n}f(X_i)$ and $\pnorm{f}{}\equiv\pnorm{f}{l_\nu}$. Let $f^\ast$ be any function in $\mathcal{F}$ such that $\pnorm{f^\ast}{n}-2\pnorm{f^\ast}{}>\epsilon/2$. Note that $f^\ast$ depends only on $\underline{X}^n$. Then we argue by bounding from below as follows:
	\beqa\label{ineq:flutuation_emp_1}
	& \mathbb{P}\left( \pnorm{f}{n}-\pnorm{f}{n}'>\epsilon/4\textrm{ for some }f \in \mathcal{F}\right)\\
	&\geq \mathbb{P}\left(\frac{1}{2}\pnorm{f^\ast}{n}-\frac{1}{2}\pnorm{f^\ast}{n}'>\frac{\epsilon}{8}\right)\\
	&\geq \mathbb{P}\left(\frac{1}{2}\pnorm{f^\ast}{n}-\frac{1}{2}\pnorm{f^\ast}{n}'+\frac{\epsilon}{8}>\frac{\epsilon}{4},\frac{1}{2}\pnorm{f^\ast}{n}'-\frac{\epsilon}{8}<\pnorm{f^\ast}{}\right)\\
	&\geq \mathbb{P}\bigg(\frac{1}{2}\pnorm{f^\ast}{n}-\pnorm{f^\ast}{}>\frac{\epsilon}{4}, \frac{1}{2}\pnorm{f^\ast}{n}'-\frac{\epsilon}{8}<\pnorm{f^\ast}{}\bigg)\\
	&=\mathbb{E}\left[\bm{1}_{\left\{\frac{1}{2}\pnorm{f^\ast}{n}-\pnorm{f^\ast}{}>\frac{\epsilon}{4}\right\}}\mathbb{P}\left(\frac{1}{2}\pnorm{f^\ast}{n}'-\frac{\epsilon}{8}<\pnorm{f^\ast}{}\bigg\lvert \underline{X}^n\right)\right].
	\eeqa
	Since
	\beqa\label{ineq:flutuation_emp_2}
	& \mathbb{P}\left(\frac{1}{2}\pnorm{f^\ast}{n}'-\frac{\epsilon}{8}<\pnorm{f^\ast}{}\bigg\lvert \underline{X}^n\right)\\
	&\geq \mathbb{P}\left(\frac{1}{4}(\pnorm{f^\ast}{n}')^2-\frac{\epsilon^2}{64}<\pnorm{f^\ast}{}^2\bigg\lvert \underline{X}^n\right)\\
	&=1-\mathbb{P}\bigg(\frac{1}{4}(\pnorm{f^\ast}{n}')^2\geq \frac{\epsilon^2}{64}+\pnorm{f^\ast}{}^2\bigg\lvert \underline{X}^n\bigg)\\
	&=1-\mathbb{P}\bigg((\pnorm{f^\ast}{n}')^2-\pnorm{f^\ast}{}^2\geq 3\pnorm{f^\ast}{}^2+\frac{\epsilon^2}{16}\bigg\lvert \underline{X}^n\bigg)\\
	&\geq 1-\frac{\mathrm{Var}\big[{1 \over n}\sum_{i=n+1}^{2n}\abs{f^\ast(X_i)}^2\big\lvert \underline{X}^n\big]}{\left(3\pnorm{f^\ast}{}^2+\frac{\epsilon^2}{16}\right)^2}\\
	&\geq 1-\frac{B^2\pnorm{f^\ast}{}^2}{n\left(3\pnorm{f^\ast}{}^2+\frac{\epsilon^2}{16}\right)^2}\geq 1-\frac{4B^2}{3n\epsilon^2}\geq 1-\frac{64B^2}{3n\epsilon^2}.
	\eeqa
	where in the last line we used the inequality $(a+b)^2\geq 4ab$ with $a=3\pnorm{f^\ast}{}^2$ and $b=\epsilon^2/16$. Here we boost the constant from $4$ to $64$ to match the corresponding results in page 188 of  \cite{gyorfi2002distribution}. Now  we see that for $n\geq 64B^2/\epsilon^2$, it follows that 
	\[
	\mathbb{P}\left(\frac{1}{2}\pnorm{f^\ast}{n}'-\frac{\epsilon}{8}<\pnorm{f^\ast}{}\bigg\lvert \underline{X}^n\right)\geq \frac{2}{3},
	\]
	and hence by (\ref{ineq:flutuation_emp_1}) and (\ref{ineq:flutuation_emp_2}),
	\beqas
	\mathbb{P}&\left( \pnorm{f}{n}-2\pnorm{f}{}>{\epsilon} \textrm{ for some }f \in \mathcal{F}\right)\\
	&\leq \mathbb{P}\left( \frac{1}{2}\pnorm{f}{n}-\pnorm{f}{}>\frac{\epsilon}{4} \textrm{ for some }f \in \mathcal{F}\right)\\
	&=\mathbb{P}\left(\frac{1}{2}\pnorm{f^\ast}{n}-\pnorm{f^\ast}{}>\frac{\epsilon}{4}\right)\\
	&\leq \frac{3}{2} \mathbb{P}\left(\pnorm{f}{n}-\pnorm{f}{n}'>\epsilon/4 \textrm{ for some }f \in \mathcal{F}\right).
	\eeqas
	By symmetry we see that
	\beqas
	\mathbb{P}\left( \pnorm{f}{n}-\pnorm{f}{n}'>\epsilon/4\textrm{ for some }f \in \mathcal{F}\right)=\mathbb{P}\left( \pnorm{f}{n}'-\pnorm{f}{n}>\epsilon/4 \textrm{ for some }f \in \mathcal{F}\right).
	\eeqas
	This gives the same estimate as in page 189 \cite{gyorfi2002distribution}, and hence we are done.
\end{proof}
\begin{proof}[Proof of Lemma \ref{lem:relate_metrics}]
	Let
	\[\rho(t):=\mathbb{P}\bigg[\sup_{g \in \mathcal{P}_m}\left(l_\nu(f,g)-2l_{\underline{X}^n}(f,g)\right)_{+}^2>t\bigg].\]
	Then if we denote $\mathcal{H}:=\mathcal{P}_m-f$, it is easy to see that
	\beqas
	\rho(t)&\leq \mathbb{P}\bigg[\sup_{h \in \mathcal{H}}\pnorm{h}{l_\nu}-2\pnorm{h}{l_{\underline{X}^n}}>\sqrt{t}\bigg]\\
	&\leq 3\mathbb{E}\mathcal{N}\bigg({\sqrt{2} \over 24}\sqrt{t},\mathcal{H},l_{\underline{X}^{2n}}\bigg)\exp\bigg(-\frac{nt}{1152\Gamma^2}\bigg)\wedge 1.
	\eeqas
	Here the second inequality makes use of Lemma \ref{lem:relate_metrics_general}. Note that 
	\beqas
	\mathcal{N}(\epsilon,\mathcal{H},l_{\underline{X}^{n}})\leq \mathcal{N}(\epsilon,\{x \in \R^n: (h(X_1),\ldots,h(X_n)), h \in \mathcal{H}\},\pnorm{\cdot}{2}),
	\eeqas
	and that the set $\{x \in \R^n: (h(X_1),\ldots,h(X_n)), h \in \mathcal{H}\}$ is a translation of $F_m(\underline{X}^n)$ by a given vector $\bm{f}=(f(X_1),\ldots,f(X_n))$. Furthermore $\mathrm{pdim}(\mathcal{H})=\mathrm{pdim}(F_m(\underline{X}^n))\leq D_m$. Hence by Lemma \ref{lem:entropy_pdim}, we can further bound the above display by
	\beqa\label{ineq:cont_disc_1}
	\rho(t)\leq 3\bigg(4+\frac{48\sqrt{n}\Gamma}{\sqrt{2t}}\bigg)^{\kappa D_m}\exp\bigg(-\frac{nt}{1152\Gamma^2}\bigg)\wedge 1.
	\eeqa
	Now (\ref{ineq:exp_bound_cont_disc}) follows by taking $t=\frac{1152\Gamma^2}{n}\big(\kappa D_m\log(4+24\sqrt{2}n)+u\big)$.
	For (\ref{ineq:E_bound_cont_disc}), integrating $\rho$ in (\ref{ineq:cont_disc_1}) from $0$ to $\infty$, and splitting the integral into two parts with partitioning point $\tau>0$ yields
	\beqas
	\int_0^\infty \rho(t)\ \d{t}&\leq \tau+3\int_\tau^\infty \bigg(4+\frac{24\sqrt{2}\sqrt{n}\Gamma}{\sqrt{t}}\bigg)^{\kappa D_m}\exp\bigg(-\frac{nt}{1152\Gamma^2}\bigg)\ \d{t}\\
	&\leq \tau+3\bigg(4+\frac{24\sqrt{2}\sqrt{n}\Gamma}{\sqrt{\tau}}\bigg)^{\kappa D_m}\int_\tau^\infty\exp\bigg(-\frac{nt}{1152\Gamma^2}\bigg)\ \d{t}\\
	&=\tau+\frac{3456\Gamma^2}{n}\bigg(4+\frac{24\sqrt{2}\sqrt{n}\Gamma}{\sqrt{\tau}}\bigg)^{\kappa D_m}\exp\bigg(-\frac{n\tau}{1152\Gamma^2}\bigg).
	\eeqas
	By choosing $\tau=\frac{1152\Gamma^2}{n}\kappa D_m\log(4+24\sqrt{2}n)$, the second term in the above display becomes
	\beqas
	&\frac{3456\Gamma^2}{n}\bigg(4+\frac{24\sqrt{2}n}{\sqrt{1152\kappa D_m\log (4+24\sqrt{2}n)}}\bigg)^{\kappa D_m} (4+24\sqrt{2}n)^{-\kappa D_m}
	\eeqas
	which is bounded by $\frac{3456\Gamma^2}{n}$. This completes the proof.
\end{proof}

\section{Auxiliary results}\label{section:auxiliary_results}
\subsection{Auxiliary results from empirical process theory}
\begin{lemma}[Dudley's entropy bound]\label{lem:dudleyentbound}
	Suppose that $\{X_f\}_{f \in \mathcal{F}}$ is a centered Gaussian process indexed by $\mathcal{F}$. For $f,g \in \mathcal{F}$ define $d_X^2(f,g):= \mathbb{E}(X_f-X_g)^2$ for each $f,g \in \mathcal{F}$. Then there is an absolute constant $C>0$ such that
	\beqas
	\mathbb{E}\sup_{f \in \mathcal{F}}X_f \leq C\int_0^{\textrm{diam}(\mathcal{F})}\sqrt{\log \mathcal{N}(\epsilon,\mathcal{F},d_X)}\ \d{\epsilon}.
	\eeqas
\end{lemma}

\begin{lemma}[Lemma 3.2 \cite{van2000empirical}]\label{lem:sup_gaussian_process}
	For fixed $x_1,\ldots,x_n$, let $\pnorm{g}{Q_n}^2:={1 \over n}\sum_{i=1}^n g(x_i)^2$. Suppose for some constants $C_1,C_2>0$ and all $a>0$, each $\gamma_1,\ldots,\gamma_n$, the random variables $\epsilon_1,\ldots,\epsilon_n$ satisfy
	\beqa\label{ineq:cond_peeling_emp1}
	\mathbb{P}\left(\abs{\sum_{i=1}^n \epsilon_i\gamma_i}\geq a\right)\leq C_1\exp\left[-\frac{a^2}{C_2^2\sum_{i=1}^n\gamma_i^2}\right].
	\eeqa
	Assume $\sup_{g \in \mathcal{G}} \pnorm{g}{Q_n}\leq R$. Then for all $\delta>0$ satisfying
	\beqa\label{ineq:cond_peeling_emp2}
	\sqrt{n}\delta \geq \left(24C_2\int_0^R \sqrt{\log \mathcal{N}(\epsilon,\mathcal{G},\pnorm{\cdot}{Q_n})}\ \d{\epsilon}\right)\vee\left(\sqrt{1152\log 2}C_2 R\right),
	\eeqa
	it holds that
	\beqas
	\mathbb{P}\left(\sup_{g \in \mathcal{G}}\abs{{1 \over n}\sum_{i=1}^n\epsilon_ig(x_i)}\geq \delta\right)\leq 2C_1\exp\left[-\frac{n\delta^2}{1152C_2^2R^2}\right].
	\eeqas
\end{lemma}

\begin{lemma}\label{lem:rate_convergence_generic_misspecification}
	Suppose the errors $\epsilon_i$'s are i.i.d. sub-Gaussian with parameter $\sigma^2$. Fix any polyhedral convex function $g \in \mathcal{P}_m$. Suppose $J$ is a function on $(0,\infty)$ such that
	\beqa
	J(r)\geq \int_0^r \sqrt{\log \mathcal{N}(\epsilon,S_m(g,r),l_{\underline{X}^n})}\ \d{\epsilon}
	\eeqa
	and that $J(r)/r^2$ is decreasing on $(0,\infty)$. Then 
	\beqa
	\mathbb{P}\big(l_{\underline{X}^n}^2(\hat{f}_m,g)>4l_{\underline{X}^n}^2(f_0,g)+4\delta|\underline{X}^n\big)\leq 4\sum_{j\geq 0}\exp\bigg(-\frac{2^{j}n\delta}{73728\sigma^2}\bigg)
	\eeqa
	holds for all $\delta>\delta_n$ with $\delta_n$ satisfying
	\begin{enumerate}
		\item $\delta_n\geq l_{\underline{X}^n}^2(f_0,g)$;
		\item $\sqrt{n}\delta_n\geq 40\sigma \big(J(\sqrt{32\delta_n})\vee \sqrt{32\delta_n}\big)$.
	\end{enumerate}
	In particular,
	\beqa
	\mathbb{E}\left[l_{\underline{X}^n}^2(\hat{f}_m,f_0)\bigg\lvert \underline{X}^n\right]\leq 10l_{\underline{X}^n}^2(f_0,g)+8\delta_n+(4.8\times 10^6)\frac{\sigma^2}{n}.
	\eeqa
\end{lemma}
\begin{proof}[Proof of Lemma \ref{lem:rate_convergence_generic_misspecification}]
	We give some details of the discussion in pages 184-185 \cite{van2000empirical} for the reader's convenience. By definition of $\hat{f}_m$, we have $\pnorm{Y-\hat{f}_m}{n}^2\leq \pnorm{Y-g}{n}^2$. Plugging in $Y=f_0+\epsilon$ we see that
	\beqas
	\pnorm{\hat{f}_m-f_0}{n}^2\leq \pnorm{f_0-g}{n}^2+2\iprod{\epsilon}{\hat{f}_m-g}_n.
	\eeqas
	Since we need to connect the empirical process part with a quadratic lower bound in $\hat{f}_m-g$, we have
	\beqas
	\pnorm{\hat{f}_m-g}{n}^2&\leq 2\pnorm{\hat{f}_m-f_0}{n}^2+2\pnorm{f_0-g}{n}^2\\
	&\leq 4\pnorm{f_0-g}{n}^2+4\iprod{\epsilon}{\hat{f}_m-g}_n.
	\eeqas
	Fix $\delta\geq\pnorm{f_0-g}{n}^2$, let $u=4\pnorm{f_0-g}{n}^2+4\delta$. Then it follows that
	\beqas
	\mathbb{P}\left(\pnorm{\hat{f}_m-g}{n}^2>u\big\lvert \underline{X}^n\right)&\leq \sum_{j=0}^\infty \mathbb{P}\bigg(4\cdot2^j\delta<\pnorm{\hat{f}_m-g}{n}^2-4\pnorm{f_0-g}{n}^2\leq 4\cdot 2^{j+1}\delta\bigg\lvert\underline{X}^n\bigg)\\
	&\leq \sum_{j=0}^\infty  \mathbb{P}\left(\sup_{f\in\mathcal{P}_m: \pnorm{f-g}{n}^2\leq 4\pnorm{f_0-g}{n}^2+2^{j+3}\delta}\iprod{\epsilon}{f-g}_n>2^j\delta\big\lvert\underline{X}^n\right)\\
	&\leq \sum_{j=0}^\infty  \mathbb{P}\left(\sup_{f\in\mathcal{P}_m: \pnorm{f-g}{n}^2\leq 2^{j+5}\delta}\iprod{\epsilon}{f-g}_n>2^j\delta\big\lvert\underline{X}^n\right).
	\eeqas
	In our framework with $\epsilon_1,\ldots,\epsilon_n$ sub-Gaussian with parameter $\sigma^2$, (\ref{ineq:cond_peeling_emp1}) is satisfied with $C_1=2,C_2=\sqrt{2}\sigma$. We now take $\delta_n>0$ such that
	\beqas
	\sqrt{n}\delta_n\geq 40\sigma \bigg(\int_0^{\sqrt{32\delta_n}}\sqrt{\log \mathcal{N}(\epsilon,S_m(g,\omega),l_{\underline{X}^n})}\ \d{\epsilon}\vee \sqrt{32\delta_n}\bigg).
	\eeqas
	By requiring $\delta\geq \delta_n$, for any $j\geq 0$, (\ref{ineq:cond_peeling_emp2}) is satisfied with $R\equiv \sqrt{2^{j+5} \delta}$. Hence the series of probabilities can be bounded further by
	\beqas
	4\sum_{j=0}^\infty \exp\bigg(-\frac{n2^{2j}\delta^2}{2304\sigma^2 2^{j+5}\delta}\bigg)=4\sum_{j=0}^\infty \exp\bigg(-\frac{2^jn\delta}{73728\sigma^2}\bigg),
	\eeqas
	as long as $\delta\geq \delta_n \vee \pnorm{f_0-g}{n}^2$. Thus we have
	\beqas
	\mathbb{E}\left[\big(l_{\underline{X}^n}^2(\hat{f}_m,g)-4l_{\underline{X}^n}^2(f_0,g)\big)/4\bigg\lvert\underline{X}^n\right]&\leq \delta_n+4\sum_{j\geq 0}\int_0^\infty \exp\bigg(-\frac{2^jn\delta}{73728\sigma^2}\bigg)\ \d{\delta}\\
	&\leq \delta_n+(6\times 10^5)\frac{\sigma^2}{n}.
	\eeqas
	Now the conclusion follows since $l_{\underline{X}^n}^2(\hat{f}_m,f_0)\leq 2l_{\underline{X}^n}^2(\hat{f}_m,g)+2l_{\underline{X}^n}^2(f_0,g)$.
\end{proof}

\subsection{Auxiliary results from convex geometry}\label{section:auxiliary_results_convex_geometry}
In this section, we will largely follow \cite{baddeley2007random} and   \cite{barany2008random}. Recall that　 $\mathscr{C}$ is the set of all convex bodies in $\R^d$. A halfspace $H$ has the form 
\beqa
H\equiv H(a\leq t)\equiv\{x \in \R^d: a\cdot x\leq t\}
\eeqa for some $a\in \R^d, t\in \R$. A \emph{cap} of a convex body $\Omega \in \mathscr{C}$ is a set of the form $C=\Omega\cap H$ where $H$ is a closed halfspace. The \emph{width} of a cap $C=\Omega\cap H(a\leq t)$ is the minimum value of $w$ so that $H(a=t-w)$ is a supporting hyperplane to $\Omega$. The \emph{center} of the cap $C$ is now defined as the center of gravity of the set $\Omega \cap H(a=t-w)$. The blown-up copy of $C$ from its center by a factor $\lambda>0$ is denoted by $C^\lambda$. It is clear that $C^\lambda$ lies between $H(a=t-w)$ and $H(a=t-w+\lambda w)$. For $\lambda\geq 1$, it follows from convexity that $\Omega \cap H(a\leq t-w+\lambda w)\subset C^\lambda$, 
and thus 
\beqa\label{ineq:blown_up_volume_estimate}
\abs{\Omega\cap H(a\leq t-w+\lambda w)}\leq \lambda^d \abs{C}.
\eeqa
Define the function $v:\Omega \to \R$ by
\beqa
v_\Omega (x):=\min\{\abs{\Omega \cap H}: x\in H, H\textrm{ halfspace}\}.
\eeqa
The \emph{minimal cap} belonging to $x \in \Omega$ is a cap $C(x)$ with $x \in C(x)$ and $\abs{C(x)}=v(x)$. The minimal cap $C(x)$ may not be unique, so we choose one of the minimzers; this will not cause any trouble in the sequel. Now define the level sets by
\beqa
\Omega(v\ast t):= \{x \in \Omega: v(x)\ast t\},
\eeqa
where $\ast\in \{<,\leq, =,>,\geq \}$. The \emph{wet part} of $\Omega$ with parameter $t>0$ is defined by $\Omega(t):=\Omega(v\leq t)=\{x \in \Omega:v(x)\leq t\}$.
Note that the function $v:\Omega \to \R$ is invariant under non-degenerate linear transformations $A:\R^d\to \R^d$, i.e.
\beqas
v_{A\Omega}(Ax)=\abs{\det A}v_{\Omega}(x).
\eeqas
This shows that the quantity ${\abs{\Omega(t\abs{\Omega})}}/{\abs{\Omega}}$
is invariant under non-degenerate linear transformation. Thus we only need to consider all convex bodies admiting unit volume; we shall denote this class $\mathscr{C}_1$. It is known (cf. \cite{schutt1990convex}) that for any convex body $\Omega$, the limit
\beqa\label{eqn:floating_body_limit}
\lim_{t \to 0} t^{-2/(d+1)}\abs{\Omega(t)}
\eeqa
exists and is finite. If further the convex body is $C^2$ smooth, then the limit is 
\beqa\label{eqn:floating_body_limit_2}
c_d^{-1}\int_{\partial \Omega}\big(\kappa(x)\big)^{1/(d+1)}\ \d{\sigma(x)}
\eeqa
where $\kappa$ is the Gaussian curvature and $\sigma(\cdot)$ is the surface measure, and $c_d=2\left(\frac{\abs{B_{d-1}}}{d+1}\right)^{2/(d+1)}$.  The limiting value (\ref{eqn:floating_body_limit_2}) is maximized in the class $\mathscr{C}_1$ by ellipsoids. On the other hand, when $\Omega\equiv P$ is a polytope, the following holds (cf. Theorem 4.2, page 82 in \cite{baddeley2007random}):
\beqa\label{eqn:wet_part_poly}
\abs{P(t)}=\frac{\mathfrak{T}(P)}{d^{d-1}d!} t\bigg(\log \frac{1}{t}\bigg)^{d-1}\big(1+o(1)\big),\quad \textrm{as } t \to 0.
\eeqa
Here $\mathfrak{T}(P)$ is the number of towers of $P$ where a \emph{tower} of $P$ is a chain of faces $F_0\subset F_1\subset \cdots \subset F_{d-1}$ where $F_i$ is $i$-dimensional. In this case, the limit in (\ref{eqn:floating_body_limit}) is $0$. It can also be shown that the order of growth in (\ref{eqn:wet_part_poly}) provides a lower bound for $\Omega(t)$. As a result (cf. Theorem 6.3, page 344 in  \cite{barany2008random}),
\begin{corollary}\label{cor:wet_part_estimate}
	For any convex bodies $\Omega \in \mathscr{C}_1$, and for all $t>0$ small enough, the following holds:
	\beqa\label{ineq:estimate_wet_part}
	t\bigg(\log\frac{1}{t}\bigg)^{d-1}\lesssim_{d,\Omega} \abs{\Omega(t)}\lesssim_{d,\Omega} t^{2/(d+1)}.
	\eeqa
\end{corollary}

One may wonder what happens in between the two extreme cases in the estimate (\ref{ineq:estimate_wet_part}). Actually we have the following result (cf. Theorem 4.7, page 83 in \cite{baddeley2007random}):
\begin{theorem}\label{thm:unpredictable_wet_part}
	Suppose $\omega(t)\to 0$ and $\gamma(t)\to \infty$ as $t \to 0$. Then the set consisting of convex bodies $\Omega$ that satisfy the following properties
	\begin{enumerate}
		\item for an infinite sequence of $t \to 0$: $\abs{\Omega(t)}\geq \omega(t)t^{2/(d+1)}$;
		\item for another infinite sequence of $t \to 0$: $\abs{\Omega(t)}\leq \gamma(t)t\big(\log\frac{1}{t}\big)^{d-1}$
	\end{enumerate}
	is comeagre(i.e. the complement is of first category) in the Baire space $(\mathscr{C}_1,d_H)$ where $d_H$ is the Hausdorff distance.
\end{theorem}

This says that for `most' convex bodies, the volume behavior near the boundary is unpredictable.

Another useful result in studying the boundary behavior of convex bodies is the following \emph{Economic Covering Cap Theorem} (cf. Theorem 7.1, page 345 in \cite{barany2008random}).
\begin{theorem}\label{thm:economic_cap_covering}
	Assume $\Omega \in \mathscr{C}_1$ and $0<\epsilon<d^{-1}3^{-d}$. Then there are caps $C_1,\ldots,C_m$ and pairwise disjoint convex sets $C_1',\ldots,C_m'$ such that $C_i'\subset C_i$ for each $i=1,\ldots,m$ and 
	\begin{enumerate}
		\item $\cup_{i=1}^m C_i' \subset \Omega(\epsilon) \subset \cup_{i=1}^m C_i$;
		\item $(4d)^{-d}\epsilon/2 \leq \abs{C_i'}\leq \abs{C_i}\leq (10d+1)^d\epsilon$ for each $i=1,\ldots,m$;
		\item for each cap $C$ with $C\cap K(v>\epsilon)=\emptyset$, we can find some $C_i$ so that $C_i\supset C$.
	\end{enumerate}
\end{theorem}
\begin{remark}\label{rmk:num_cover_boundary}
	By Corollary \ref{cor:wet_part_estimate} we know that $\abs{\Omega(\epsilon)}\lesssim _{d,\Omega} \epsilon^{2/(d+1)}$. Hence by (2) of the above theorem, it follows that $m$ is on the order of $\epsilon^{-(d-1)/(d+1)}$.
\end{remark}

\section*{Acknowledgements}
The authors would like to thank T. Tony Cai,  Aditya Guntuboyina, Johannes Lederer, Richard Samworth and Bohdi Sen for very helpful and stimulating conversations at various stages of this work.


\end{document}